%% file: thomnegative_-_arXiv.tex
\documentclass[12pt,A4,oneside,english]{amsart}
\usepackage[hmargin=2cm,bottom=3cm,textwidth=17.3cm]{geometry}
\usepackage{amsthm}
\usepackage{amsmath}
\usepackage[hidelinks]{hyperref}
\usepackage{bbm}
\usepackage{booktabs}
\usepackage{array}
\usepackage{tikz}
\usepackage{cancel}
\usepackage{xcolor}
\usepackage{graphicx} 
\usetikzlibrary{matrix,calc}
\usepackage{booktabs,longtable,array}

\newcommand{\de}{\delta}
\newcommand{\qu}{\sigma_q}

\newcommand{\localheading}[1]{%
	\par\bigskip
	\noindent{\bfseries #1}\par
	\smallskip
}

\newcommand{\cubpic}[1]{%
	\begin{tikzpicture}[scale=0.42,baseline=-0.55ex,line cap=round,line join=round]
		#1
	\end{tikzpicture}%
}

\newcommand{\diagNodal}{\cubpic{
		\draw[thick,domain=-1.45:1.45,samples=180,smooth,variable=\t]
		plot ({0.60*(\t*\t - 1)}, {0.42*(\t*\t*\t - \t)});
}}

\newcommand{\diagNu}{\cubpic{
		\draw[thick,domain=-1.45:1.45,samples=150,smooth,variable=\t]
		plot ({0.35*\t*\t*\t},{0.75*\t*\t} );
}}

\newcommand{\diagTheta}{\cubpic{
		\draw[thick] (0,0) ellipse [x radius=0.65,y radius=0.95];
		\draw[thick] (-1,-0.25) -- (1,0.25);
}}

\newcommand{\diagOmega}{\cubpic{
		\draw[thick] (0,0) ellipse [x radius=0.95,y radius=0.65];
		\draw[thick] (-1.25,-0.65) -- (1.25,-0.65);
}}

\newcommand{\diagA}{\cubpic{
		\draw[thick] (-1.15,-0.45) -- (1.15,-0.45);
		\draw[thick] (-0.75,1.05) -- (0.95,-1.05);
		\draw[thick] (0.75,1.05) -- (-0.95,-1.05);
}}

\newcommand{\diagZh}{\cubpic{
		\draw[thick] (-1.25,-0.85)--(1.25,0.85);
		\draw[thick] (-1.30,0)--(1.30,0);
		\draw[thick] (-0.80,1.10)--(0.80,-1.10);
}}

\newcommand{\diagSmooth}{\cubpic{
		\draw[thick,domain=-1.25:1.25,samples=120,smooth,variable=\t]
		plot ({\t}, {0.45*(\t*\t*\t - \t)});
}}

\newcommand{\Zh}{%
	\mathord{%
		\begin{tikzpicture}[baseline=-0.55ex,scale=0.13,line cap=round,line join=round]
			\draw[line width=0.55pt] (0,-1) -- (0,1);
			\draw[line width=0.55pt] (-1,-1) -- (0,0) -- (-1,1);
			\draw[line width=0.55pt] (1,-1) -- (0,0) -- (1,1);
		\end{tikzpicture}%
	}%
}

\usepackage{tabularx} 
\newcolumntype{Y}{>{\raggedright\arraybackslash}X} 
\newcommand{\LTPcell}[1]{%
	 \(\scriptsize\begin{aligned}[t]#1\end{aligned}\)%
}

\newtheorem*{theorem*}{Theorem}
\newtheorem*{lemma*}{Lemma}
\newtheorem*{definition*}{Definition}
\numberwithin{equation}{section}

\input{abbrev.tex}

\allowdisplaybreaks[2]

\makeatletter
\renewcommand*\l@subsection{\@tocline{2}{0pt}{2.5pc}{5pc}{}}
\makeatother

\address{\'Akos K.\ Matszangosz, HUN-REN Alfr\'ed R\'enyi Institute of Mathematics, Re\'altanoda utca 13-15, 1053 Budapest, Hungary}
\email{matszangosz.akos@gmail.com}

\address{L\'aszl\'o M. Feh\'er, 
	Department of Analysis, Institute of Mathematics, E\"otv\"os Lor\'and University,
	Pázmány Péter sétány 1/C, 1117 Budapest, Hungary and  HUN-REN Alfr\'ed R\'enyi Institute of Mathematics, Re\'altanoda utca 13-15, 1053 Budapest, Hungary}
\email{lfeher63@gmail.com}

\thanks{\'A. K. M. is supported by the Hungarian National Research, Development and Innovation Office, NKFIH PD 145995.}
\thanks{L. M. F. enjoyed the hospitality of the R\'enyi Institute for a year while working on this paper}
\title[Obstructions for Morin and fold maps]{Obstructions for Morin and fold maps: Stiefel-Whitney classes and Euler characteristics of singularity loci}
\author{L\'aszl\'o M. Feh\'er, \'Akos K. Matszangosz}
\subjclass[2020]{32S20, 57R45, 58K30, 05E05, 55N91}
\keywords{Thom polynomial, Legendre Thom polynomial, contact singularities, function singularities, equivariant cohomology, Schur polynomials, Q-polynomials.}

\newtheorem*{rremark*}{Remark}
\newenvironment*{remark*}{\begin{rremark*}\small \rm}{\end{rremark*}}

\usepackage{accents}

\usepackage{tikz-cd}
\usetikzlibrary{decorations.pathreplacing}

\newcommand{\gbinom}[2]{ {{#1}\brace{#2}} }
\begin{document}
  \title{Thom series in negative relative codimension}
\maketitle
\begin{abstract} We develop a theory of Thom series for contact function singularities. Quadratic stabilization $\si_q$ of a contact singularity $\eta\in J^k(n,1)$ gives contact singularities $\si_q^{i}\eta\in J^k(n+i,1)$.  We study the stable Thom polynomials (i.e.\ the Thom polynomial in quotient variables) of $\si_q^{i}\eta$ as $i$ increases. To this end, we define the \emph{Thom series of a contact function singularity} $\eta$, which for any given $i$ is a linear combination of Schur polynomials $s_\la$. For large enough $i$ the value of the Thom series at $1-(n+i)$ is the stable Thom polynomial of $\si_q^i\eta$ and for small $i$ it determines its coefficients outside the kernel of a specialization map.
	
	We prove that the Thom series has two main properties: 1) as $i$ increases, the partitions $\la$ follow a simple stabilization pattern, and there is a finite set of $\la$ which generates the support of the entire Thom series via this stabilization; 2) the coefficients of the $s_\la$ are polynomials in $i$ with explicit degree bounds.
	
	We describe the precise relationship between unstable and stable Thom polynomials of contact function singularities and Legendre Thom polynomials and we carry out computations of all three, as follows. We compute the complete family of stable Thom polynomials for function singularities with $\ga\leq 6$. We compute unstable and Legendre Thom polynomials for several families of binary and ternary singularities. We define a class of multi-binary singularities, and compute their Thom polynomials. We discuss second order Thom-Boardman classes, and indicate difficulties that arise beyond function singularities.
	
	The results of the paper will be used in a companion paper \cite{FeherMatszangoszupcoming}, where Thom polynomials will be applied to problems in enumerative geometry.
\end{abstract}
\tableofcontents
\section{Introduction}
\addtocontents{toc}{\protect\setcounter{tocdepth}{1}}
\subsection*{General theory of Thom polynomials}
Thom \cite{Thom1955} noticed that given a generic smooth map $f:N^n\to P^p$ and a contact singularity type $\eta\subset J^k(n,p)$, where $J^k(n,p)$ denotes the space of $k$-jets for a large $k$, there exists a universal 
\emph{Thom polynomial}
\[[\eta]\in \F_2[a_1\stb a_n,b_1\stb b_p]\] 
which computes the fundamental cohomology class of the singularity locus $\eta(f)$:
\[
[\eta(f)\subset N]=[\eta](w_1(N)\stb w_n(N),f^*w_1(P)\stb f^*w_p(P))
\]
where $w_i(N)$ and $w_i(P)$ are the Stiefel-Whitney classes of the tangent bundles of $N$ and $P$.  The analogous Thom polynomials exist in the complex category, with Chern classes replacing Stiefel-Whitney classes. In this paper we only work in the complex category. Complex Thom polynomials can be used to solve a wide array of enumerative problems. The results of the paper will be used in a companion paper \cite{FeherMatszangoszupcoming}, where Thom polynomials will be applied to problems in enumerative geometry.

The first structure theorem about Thom polynomials of contact singularities is the Damon theorem \cite{damon1972thom}, which we briefly recall.  Given a contact singularity $\eta\subset J^k(n,p)$ there is an infinite sequence of singularities $\sigma^i(\eta)\subset J^k(n+i,p+i)$ obtained by \emph{trivial unfolding}, where the trivial unfolding $\sigma(g)\in J(n+1,p+1)$  of $g\in J^k(n,p)$ is given by 
\[\sigma(g)(x_1,\dots,x_{n+1})=(g_1,\dots,g_p,x_{n+1}).\]

Damon's theorem states that the Thom polynomials of the iterated trivial unfoldings of a contact singularity $\eta\subset J^k(n,p)$ are the same in the following sense: 

\begin{theorem*}[Damon] For any $k$-determined contact singularity $\eta \subset J^k(n,p)$ there is a unique  polynomial  $\Tp(\eta)\in \Z[c_1,\dots,c_i,\dots]$ that we will call the \emph{stable Thom polynomial of $\eta$}, such that
	\[ [\sigma^i(\eta)\subset J^k(n+i,p+i)]=\rho^{n+i\to p+i}(\Tp(\eta)),\]
where define the ring homomorphism
\[\rho^{u\to v}:   \Z[c_1,\dots,c_i,\dots]\to \Z[a_1,\dots,a_u,b_1,\dots,b_v]\]
by the formal power series
\[1+c_1+c_2+\cdots=\frac{1+b_1+b_2+\cdots+b_v}{1+a_1+a_2+\cdots+a_u}.\]
\end{theorem*}

In Section \ref{sec:tp-and-factorization} we provide some classical examples of stable Thom polynomials and show how to calculate them using the Schur Factorization Theorem \ref{thm:schur-factorization}.

\bigskip

For the further study of stable Thom polynomials we need to distinguish the $l:=p- n\geq0$ and the $l<0$ cases. 
\subsection*{A second structure theorem for $p\geq n$}

For $l=p-n\geq0$, a second structure theorem was introduced in \cite{FeherRimanyi2007} based on the \emph{zero unfolding} operation: given $f=(f_1,\dots,f_p)\in J^k(n,p)$, we define its zero unfolding as $\sigma_0(f):=(f_1,\dots,f_p,0)$. This operation readily extends to contact singularities. The key result of \cite{FeherRimanyi2007} is that the stable Thom polynomials of the sequence of iterated zero unfoldings of a contact singularity can be arranged into a \emph{Thom series}. A classical example is due to Ronga \cite{Ronga1972}: for the singularity $A_2=\Si^{1,1}$, for $l\geq 0$:
\[
\Tp(A_2(l))=s_{2^{l+1}}+2s_{2^l,1^2}+4s_{2^{l-1},1^3}+\ldots
\]
Notice that instead of the notation $\sigma_0^l(A_2)$ we use the classical and convenient notation $A_2(l)$, where $l$ refers to the relative codimension $p-n$. Later several Thom series were calculated (e.g \cite{BercziSzenes,FeherRimanyi2012,Ozturk,Kazarian}) and now the theory is well developed. A large collection of  Thom series and Thom polynomials for $l\geq 0$ are available on \cite{tpp}. 

In Section  \ref{sec:p>n} we review the theory of Thom series and show how to calculate $\Tp(A_2(l))$ for $l\geq 0$ using the Schur Factorization Theorem \ref{thm:schur-factorization}.

\subsection*{Theory of Thom polynomials for $p< n$.} Despite the fact that the $l\leq 0$ cases have more enumerative applications, this topic is less developed. The ambitious goal of the authors was to develop the theory for the $p< n$ case. We were successful in the case of function singularities, in particular we introduce an analogue of the Thom series.

In negative relative codimension $l< 0$  the notion of zero unfolding is replaced by the notion of \emph{quadratic unfolding}: given $f=f(x_1,\dots,x_n)$ we define $\sigma_q(f)=f+x_{n+1}^2$, see Section \ref{sec:quadratic-unfolding}. In particular, we define an analogue of Thom series for contact function singularities. The following example demonstrates the situation: with  $r=-l> 0$ and $A_2(-r)=\Si^{r+1,1}(-r)$, we have by \cite[\S 6]{FeherKomuves}
\[
\Tp(A_2(-r))=(r+1)s_{2+r} + 2 s_{r+1,1}.
\]
In particular, the support of Schur polynomials is finite, but the coefficients are no longer constant. Note that for $\eta\subset J^k(n,1)$ we use the traditional notation $\eta(-r)$ instead of $\sigma_q^{r+1-n}(\eta)$. In general, we will prove the following theorem (see Theorem \ref{thm:poly4p=1}  for a more precise statement):
\begin{theorem*}
	For a stable function singularity $\eta$ the stable Thom polynomial $\Tp(\eta(-r))$ is a polynomial in $r$ in the following sense: for $r\geq \ga(\eta)-2$,
	\[ \Tp(\eta(-r))=\sum_{|\alpha|+i+1=\gamma(\eta)} g_{\alpha,i}(r)s_{r+1+i,\alpha},\]
	where $i\geq 0$, $\al=(\al_2\stb \al_n)$ a partition, $g_{\alpha,i}(r)$ are polynomials in $r$ of degree at most $i$  and $\gamma(\eta)=\codim(\eta(-r))-r$. 
\end{theorem*}

We call the right hand side the \emph{Thom series} of $\eta$, which for large $r$ gives the value of the stable Thom polynomial $\Tp(\eta(-r))$, but for small values, corrections are needed --- reminiscent of the relationship between the Hilbert function and polynomial. We introduce the Thom series of function singularities in Section \ref{sec:p<n}.

The proof of Theorem \ref{thm:poly4p=1} is given in Section \ref{sec:ltp-to-tp}. It is based on \emph{Kazarian's structure theorem} \cite{Kazarian2003} (see Theorem \ref{thm:kazarian}), which relates the \emph{Legendre Thom polynomial} $K_\eta$ \cite{Vassiliev1993} of the function singularity $\eta$ and the (unstable) Thom polynomials of the quadratic unfoldings of $\eta$.

 We describe the range and extent to which the different Thom polynomials determine each other on Figure \ref{fig:intro}. (Here we indicate the source and target dimension for $\eta\subset J(n,p)$ by the notation  $\eta(n,p)$.) For further details see Section \ref{sec:thompolynomialsconnection}.

\subsection*{Computations}
A consequence of the proof of Theorem \ref{thm:poly4p=1} is that the Legendre Thom polynomial $K_\eta$ contains the same information as the Thom series of $\eta$ and thus determines the stable Thom polynomials $\Tp(\eta(-r))$ for $r$ large enough. To calculate the entire family $\{\Tp(\eta(-r))\}$ we need to calculate $\Tp(\eta(-r))$ for small values of $r$. For this we use the restriction equation method of Rim\'anyi  \cite{Rimanyi}.%

This implies that there are 3 levels to calculate with increasing information:
\begin{enumerate}
	\item \textbf{The unstable Thom polynomial $[\eta(n,1)]$.} We calculate a wide variety of examples in Section \ref{sec:binary}, including infinite families. In particular we calculate the unstable Thom polynomials of $J_{10}$, $X_{10}$, $X_{11}$, $T_{2,5,5}$, $Z_{11}$, $Z_{12}$, $W_{12}$, $W_{13}$.
	\item \textbf{The entire family of unstable Thom polynomials $[\eta(n+i,1)]$.} This is equivalent to the calculation of the Legendre Thom polynomial $K_\eta$, and also  equivalent to the calculation of the Thom series of $\eta$. We obtain the Legendre Thom polynomials of the ternary  singularities $T_{3,3,4}, T_{3,4,4}, T_{4,4,4},$ ${Q_{10}, S_{11}}$ in Section \ref{sec:binary}. 
	\item \textbf{The entire family of stable Thom polynomials $\{\Tp(\eta(-r))\}$.} In addition to the above, this requires calculating the stable Thom polynomials $\Tp(\eta(-r))$ for small $r$. We use the restriction equation method to calculate these for contact function singularities up to $\ga\leq 6$, that is:  $A_2,A_3, A_4, D_4, A_5, D_5, A_6, D_6, E_6$, see Section \ref{sec:calculations}.
\end{enumerate}

 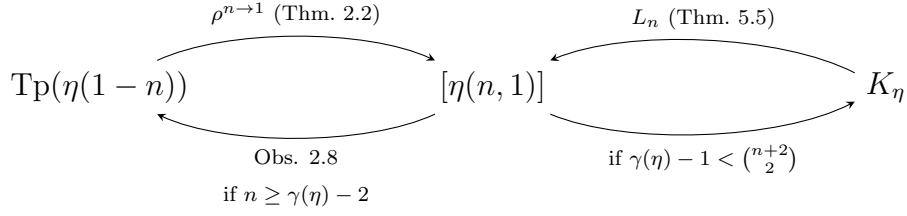
\begin{figure*}
 	\centering
 	\begin{tikzpicture}[>=stealth]
 		\def\d{5.2}
 		\node (A) at (0,0) {$\Tp(\eta(1-n))$};
 		\node (B) at (\d,0) {$[\eta(n,1)]$};
 		\node (C) at (2*\d,0) {$K_\eta$};
 		
 		\draw[->, bend left=25] (A) to[looseness=0.7]  node[above] {\tiny $\rho^{n\to 1}$ (Thm.~\ref{damon})} (B);
 		\draw[->, bend left=25] (B) to[looseness=0.7]  node[below,align=center] {\tiny Obs. \ref{obs:tptostable}\\\tiny if $n\geq \ga(\eta)-2$} (A);
 		
 		\draw[<-, bend left=25] (B) to[looseness=0.7] node[above] {\tiny $L_n$ (Thm.~\ref{thm:kazarian})} (C);
 		\draw[<-, bend left=25] (C) to[looseness=0.7]  node[below] {\tiny if $\ga(\eta)-1<\binom{n+2}{2}$} (B);
 	\end{tikzpicture}
 	\caption{The relationship between the stable Thom polynomials, Thom polynomials and Legendre Thom polynomials for function singularities.}\label{fig:intro}
 \end{figure*}

Earlier, computations of Thom polynomials for $l=-1$ have been carried out in \cite{Ando1996}, \cite{PragaczWeber}, \cite{Ohmoto2009}. Some Thom series for $l<0$ appear in guise in \cite{FeherKomuves}. Although some stable Thom polynomials for $l=-1$ have been determined in \cite{Ohmoto2009}, not all computed formulas are easily accessible in the literature \cite[p.\ 24]{OhmotoJap}.
Kazarian computed several Legendre Thom polynomials \cite{Kazarian2003}, \cite{Kazarian2003b}. 

The papers \cite{MikoszPragaczWeber2009}, \cite{MikoszPragaczWeber2011} demonstrated that $Q$-polynomials (discussed in Section \ref{sec:Q}) form a convenient basis for Legendre Thom polynomials, in particular they are positive in the $Q$-basis. More importantly Kazarian's structure theorem \ref{thm:kazarian} has a particularly simple description in the $Q$-basis, what we explain in Section \ref{sec:QL}. Our calculations of Legendre Thom polynomials use this description and the $Q$-factorization Theorem \ref{thm:q-factorization}.
\smallskip

Our computations proceed as follows.

(1) Some of the singularities listed above (e.g.\ $X_{10}, Z_{11}, W_{12}$) are specified by having 2 variables and their lowest degree $d$ behaviour, we call the resulting contact singularities \emph{binary singularities}. So the study of the Thom polynomials of such binary singularities reduces to the study of $\GL_2$-equivariant classes of the coincident-root stratification in the representation $\Pol^k(\C^2)$, computed in \cite{FeherNemethiRimanyi}. We compute Thom polynomials of binary singularities based on these considerations. 

We also generalize the coincident root locus description to a class of singularities which are specified by the property that in some interval $[d,d+r]$ of degrees, the degree $d+i$ part shares a root of multiplicity $b_i$; we call such singularities \emph{multi-binary}. By generalizing the equivariant class computation of \cite{FeherNemethiRimanyi}, we compute the classes of such \emph{shared-root loci} in $\Pol^d(\C^2)\oplus \ldots\oplus \Pol^{d+r}(\C^2)$ and derive the Thom polynomials of multi-binary singularities.

(2)  Following this line of inquiry, we use the results of \cite{Komuves} on the equivariant classes of ternary cubic strata to compute the following Thom polynomials in $J(3,1)$. These determine the Legendre Thom polynomials of the singularities listed in (2) above. By Figure 1, together with Kazarian's structure theorem (Theorem \ref{thm:kazarian}) these results also determine the Thom series for these singularities.

(3) Finally, we apply Rim\'anyi's restriction equation / interpolation method \cite{Rimanyi} to compute $\Tp(\eta(-r))$ for small $r$.

\bigskip

A similar structure theorem for contact, non-function singularities is missing. In Section \ref{sec:beyond} we illustrate some of the difficulties on the example of the Thom--Boardman singularity $\Si^{i,1}(2-i)$: polynomiality still holds, but the stabilization of the terms in the Schur expansion is not unique, furthermore the support of the Schur expansion is no longer finite. 

In an upcoming companion paper \cite{FeherMatszangoszupcoming}, we will give enumerative applications of Thom polynomials. The results of this paper can be applied extensively to such problems. In Section \ref{sec:enum} we sketch some examples, for instance on counting projective $3$-planes in $\PP^5$ intersecting a given complete intersection 3-fold in an $E_6$-singularity.

In Section \ref{sec:open} we collected some open problems.  We explain the details of the restriction equations in Appendix \ref{sec:weights}. In Appendix \ref{app:fibered} we explain a lemma on fibered resolutions needed for calculations of Thom polynomials of multi-binary singularities.

\subsection*{Acknowledgements}

We thank Toru Ohmoto, Rich\'ard Rim\'anyi and Andrzej Weber for inspiring discussions.

\part{Theory}
\addtocontents{toc}{\protect\setcounter{tocdepth}{2}}
\section{Thom polynomials and the role of the factorization formula} \label{sec:tp-and-factorization}
For technical details on Thom polynomials we refer the reader to \cite{FeherRimanyi2012}. Two recent surveys \cite{Rimanyi2025}, \cite{Ohmoto2026} give a wider perspective on the theory of  Thom polynomials.

As explained in \cite{FeherRimanyi2012} the study of Thom polynomials can be reduced to the study of the action of the contact groups $\mathcal{K}=\mathcal{K}^k(n,p)$ on the jet spaces $J^k(n,p)$.  We call a $\mathcal{K}$-invariant subvariety $\eta$ a (contact) \emph{singularity} or \emph{singularity type}. We will always assume that $\eta$ is $k$-determined. Such a singularity $\eta$ admits a $\mathcal{K}$-equivariant cohomology class 
\begin{equation}\label{eq:TP}
 [\eta\subset J^k(n,p)] \in H^*_\mathcal{K}\iso\Z[a_1,\dots,a_n,b_1,\dots,b_p],
\end{equation}
and we call this class \emph{the Thom polynomial of $\eta$}. Increasing $k$ does not change the Thom polynomial (using the natural identification of the equivariant cohomology rings $H^*_{\mathcal{K}^k}$ for different $k$'s). Singularity types for different values of $n$ and $p$ with $p-n$ fixed can be identified via trivial unfoldings, and their Thom polynomials are essentially the same (see Theorem \ref{damon}). This leads to the notion of stable Thom polynomial that we now discuss.

\subsection{The stable Thom polynomial}
The trivial unfolding operation 
\[ \sigma: J^k(n,p)\to J^k(n+1,p+1),\ \ \sigma(f)=(f_1,\dots,f_p,x_{n+1}) \]
for $f=(f_1,\dots,f_p)\in J^k(n,p)$ assigns a singularity $\sigma(\eta)\subset  J^k(n+1,p+1)$ to a singularity $\eta \subset J^k(n,p)$. Trivial unfolding doesn't change the quotient algebra  (for the definition of the quotient algebra see \cite{FeherRimanyi2012}), the codimension of the singularity and the effect on the Thom polynomial is well understood: in \quot{quotient variables} the Thom polynomials of the iterated trivial unfoldings are the same, in the following sense:
  
  \begin{definition} We define a ring homomorphism
  	\[\rho^{n\to p}:   \Z[c_1,\dots,c_i,\dots]\to \Z[a_1,\dots,a_n,b_1,\dots,b_p]\]
  by the formal power series
  \[1+c_1+c_2+\cdots=\frac{1+b_1+b_2+\cdots+b_p}{1+a_1+a_2+\cdots+a_n}.\]
   \end{definition}
 For example $\rho^{n\to p}(c_1)=b_1-a_1$  and $\rho^{n\to p}(c_2)=b_2-a_2+a_1^2-a_1b_1$ (if $n>1$ and $p>1$).

\begin{theorem}[Damon] \label{damon} For any $k$-determined contact singularity $\eta \subset J^k(n,p)$ there is a polynomial  $\Tp(\eta)\in \Z[c_1,\dots,c_i,\dots]$ that we will call the \emph{stable Thom polynomial}, such that
	\[ [\sigma^i(\eta)\subset J^k(n+i,p+i)]=\rho^{n+i\to p+i}(\Tp(\eta)).\]
\end{theorem}

\begin{remark}(Notation for stable singularities  $\eta(l)$)\label{rmk:stable_singularities}
	If we refer to a singularity $\eta\subset J^k(n,n+l)$ or any of its trivial unfoldings \[\eta(n+i,n+i+l):=\si^i(\eta)\subset J^k(n+i,n+i+l),\] then we use the notation $\eta(l)$, where $l$ denotes the relative codimension. So formally  $\eta(l)$ is the direct limit of trivial unfoldings of $\eta$, and because of that it is also called a stable singularity (not related to the notion of local or infinitesimal stability). This notation hides the fact that there is a ``first" instance of a singularity appearing, e.g.\ $I_{2,2}(0)$ appears first in $J(2,2)$.
\end{remark}
\begin{remark}
In fact, classically $\Tp(\eta)$ is what is commonly called the Thom polynomial. However, for our discussion, it is important to distinguish these notions, therefore we will call $[\eta(n,p)]$ the Thom polynomial, and $\Tp(\eta)$ the stable Thom polynomial.
\end{remark}
\subsection{The factorization formula for Schur polynomials}\label{sec:factorization}
The graded ring homomorphisms $\rho^{n\to p}$ play a key role in our study. Before stating their key properties, let us fix some notational conventions. Given an alphabet $c$ in graded variables $c_i$ and a partition $\la=(\la_1\stb \la_r)$, define the Schur polynomial
\begin{equation}\label{eq:schur}
	s_\la(c) =\det[c_{\la^T_i+j-i}]_{i,j=1}^{\ell(\la^T)}
\end{equation}
where $\la^T$ denotes the partition conjugate to $\la$. See \cite[Appendix B]{FeherMatszangosz2025} for a short discussion about  different possible conventions for the choice of $s_\la$ in Thom polynomial theory.  We will also occasionally use  \quot{\emph{Chern root variables}}, i.e. treat $c_i$ as the $i$-th elementary symmetric polynomial of some variables $x_1,\dots,x_n$.

The image $\rho^{n\to p}(s_\lambda)$ behaves differently depending on the shape of the partition $\lambda$. So to express $\rho^{n\to p}(s_\lambda)$ we introduce the following:

\begin{definition}\label{def:thick}
	Let $n$ and $p$ be fixed. Then partitions fall into three distinct categories: 
	\begin{enumerate}
		\item If $n^p \not\subset \lambda$ then we call $\lambda$ \emph{thin}.
		\item If $n^p \subset \lambda$ but $(n+1)^{(p+1)}\not\subset\lambda $ then we call $\lambda$ \emph{nice}. 
		\item If  $(n+1)^{(p+1)}\subset\lambda $ then we call $\lambda$ \emph{thick}.	
	\end{enumerate}
\end{definition}
	These names refer to the shape of the Young diagrams, see Figure \ref{fig:young}. Notice that a nice partition can be uniquely written in the form  $(n^p+\beta,\alpha)$, where $l(\beta)\le p$ and $\alpha_1\le n$;
i.e.~$\lambda_i=n+\beta_i$ for $i\leq p$ and $\lambda_i=\alpha_{i-p}$ for $i>p$.

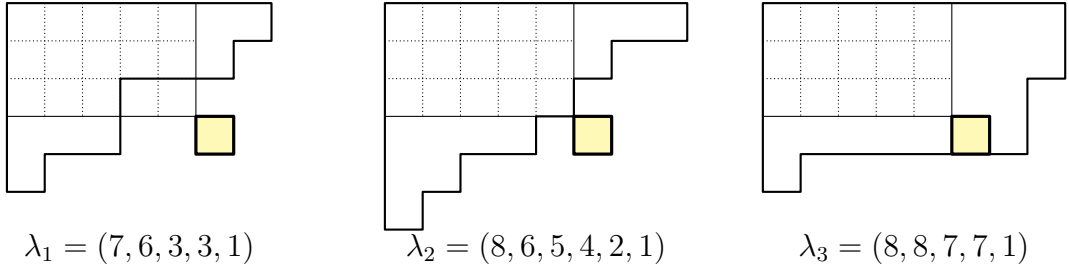
\begin{figure}
	\begin{tikzpicture}[scale=0.5, line join=round]
	
	\newcommand{\referencepartition}{
		\foreach \x in {1,2,3,4} {
			\draw[densely dotted] (\x,0) -- (\x,-3);
		}
		
		\foreach \y in {-1,-2} {
			\draw[densely dotted] (0,\y) -- (5,\y);
		}
	}
	
	\newcommand{\referenceboundary}{
		\draw[thin] (0,0) rectangle (5,-3);
	}
	
	\newcommand{\highlightbox}[2]{
		\pgfmathtruncatemacro{\row}{#1-1}
		\pgfmathtruncatemacro{\col}{#2-1}
		
		\fill[yellow!35] (\col,-\row) rectangle ++(1,-1);
		\draw[very thick] (\col,-\row) rectangle ++(1,-1);
	}
	
	\begin{scope}[xshift=0cm]
	\referencepartition
	\highlightbox{4}{6}
	
	\draw[thick]
	(0,0) --
	(7,0) --
	(7,-1) --
	(6,-1) --
	(6,-2) --
	(3,-2) --
	(3,-4) --
	(1,-4) --
	(1,-5) --
	(0,-5) --
	cycle;
	
	\referenceboundary
	
	\node at (3.5,-6.5) {$\lambda_1=(7,6,3,3,1)$};
	\end{scope}
	
	\begin{scope}[xshift=10cm]
	\referencepartition
	\highlightbox{4}{6}
	
	\draw[thick]
	(0,0) --
	(8,0) --
	(8,-1) --
	(6,-1) --
	(6,-2) --
	(5,-2) --
	(5,-3) --
	(4,-3) --
	(4,-4) --
	(2,-4) --
	(2,-5) --
	(1,-5) --
	(1,-6) --
	(0,-6) --
	cycle;
	
	\referenceboundary
	
	\node at (4,-6.5) {$\lambda_2=(8,6,5,4,2,1)$};
	\end{scope}
	
	\begin{scope}[xshift=20cm]
	\referencepartition
	\highlightbox{4}{6}
	
	\draw[thick]
	(0,0) --
	(8,0) --
	(8,-1) --
	(8,-1) --
	(8,-2) --
	(7,-2) --
	(7,-4) --
	(1,-4) --
	(1,-5) --
	(0,-5) --
	cycle;
	
	\referenceboundary
	
	\node at (4,-6.5) {$\lambda_3=(8,8,7,7,1)$ };
	\end{scope}
	\end{tikzpicture}
	\caption{Thin ($\la_1$), nice ($\la_2$) and thick ($\la_3$) Young diagrams for $n=5$, $p=3$}
	\label{fig:young}
\end{figure}
\begin{theorem}\label{thm:schur-factorization}\mbox{}  The ring homomorphisms $\rho^{n\to p}$ satisfy the following properties:
\begin{enumerate}
\item[(i)\,] (thick property) $\ker (\rho^{n\to p})=\langle s_\lambda: \lambda \text{ is thick}\rangle$, where $\langle\ \rangle$ means the generated $\Z$-module.
\item[(ii)] (nice  property) For $\lambda=(n^p+\be,\al)$ nice,  
\[ 
\rho^{n\to p}(s_\lambda)
=
(-1)^{|\alpha|}\rho^{n\to p}(s_{n^p})s_{\alpha^T}(a)s_\beta(b).\]
\item[(iii)] (thin property) Suppose that $\rho^{n\to p}(P)=\rho^{n\to p}(s_{n^p})Q$. Then in Schur basis all thin coefficients of P are zero.
\end{enumerate}
\end{theorem}
\noindent
The proof of (i) and (ii) can be found
e.g.~in \cite[\S 3.2]{FultonPragacz}. 

Part (i) is a corollary of a result
of Pragacz \cite{Pragacz1988} on universally supported classes
({\em avoiding ideal} in the terminology of \cite{FeherRimanyi2004}) for $\Sigma^i$.

Part (ii) is sometimes called the {\em factorization formula} \cite[p.\ 58, Example 23]{MacDonald1995}. 

Part (iii)  is an immediate consequence of the fact  that  $\rho^{n\to p}$ maps $\Z$-module $\langle s_\lambda: \lambda \text{ is nice} \rangle$ isomorphically to the   $\Z$-module of polynomials divisible by $\rho^{n\to p}(s_{n^p})$. 
\begin{remark}\label{rmk:HomAB}
If $a_i$ and $b_i$ denote the Chern classes of a rank $n$ and $p$ vector bundle $A$ and $B$, then $\rho^{n\to p}(s_{n^p})=e(\Hom(A,B))$, where $e$ denotes the Euler class.
\end{remark}


Returning to the setting of Thom polynomials, we will repeatedly use the following observation to calculate stable Thom polynomials:

\begin{observation}\label{obs:tptostable}
The factorization formula implies that if $\eta\subset J^k(n,n+l)$ is a $k$-determined contact singularity, then the Thom polynomial $[\eta(n,n+l)]$ determines the stable Thom polynomial $\Tp(\eta(l))$  if $\codim[\eta\subset J^k(n,n+l)]<(n+1)(n+l+1)$. 
\end{observation}
For example in \cite{FeherKomuves} this was a key ingredient in calculating the Thom polynomials of second order Thom-Boardman singularities. In the next section we show some easy and not widely known examples of its use in nonnegative relative codimension.

\subsection{One-dimensional source space}\label{sec:1dimsource}
In this section we will study the case of the one-dimensional source space, which as we will see is in sharp contrast with the case of one-dimensional target space. The former consists of the Morin singularities $A_k$, whereas the case of the one-dimensional target space consists of all function singularities.\footnote{Morin singularities in negative codimension are the corank one singularities. In Thom-Boardman language: in positive codimension $l>0$ these singularities are of class $\Si^1(l)$, and in negative codimension $l<0$ these are of class $\check{\Si}^1(l)=\Si^{1-l}(l)$, where this dual notation is introduced and discussed in Definition \ref{def:TB}.}

The Morin singularities $A_k(l)$ for $l\geq 0$ have representatives with one dimensional source space: 
\[A_k(1,l+1)=\big(x\mapsto (x^{k+1},\underbrace{0\stb 0}_l)\big).\]
The form of this singularity is a special case of the zero-unfolding, see Section \ref{sec:zerounfolding}. The Morin-singularities coincide with the Thom-Boardman classes $\Si^{1^k}$, so their Thom polynomial is a special case of the following Proposition, \cite[Example 7.16]{FeherRimanyi2004} \cite[\S 6.1]{FeherRimanyi2012}.
\begin{proposition} \label{prop:sigma-n-ad-k}
 Using the  notation $n^k=(n,n,\dots,n)$ for the partition $k$ times $n$ we have
	\[
	[\Si^{n^k}(n,p)]=\prod_{i=1}^ke(\Hom(\Sym^i(\C^n)),\C^p).
	\]
\end{proposition}
\begin{proof}
$\Si^{n^k}(n,p)\subset J_0^N(n,p)$ is the kernel of the projection  $j^k: J_0^N(n,p)\to J_0^k(n,p)$, implying that its cohomology class is the Euler class of the image. 
\end{proof}
We will focus on the case of one-dimensional source:
\begin{corollary}If $\al$ and $\be_1\stb\be_p$ denote the Chern roots of $A^1$ and $B^p$, i.e.\ $a_1=\al_1$ and $b_j=e_j(\be_1\stb \be_p)$ are the elementary symmetric polynomials, then
	\begin{equation}\label{Morin1dim-source}
		[A_k(1,l+1)]=\prod_{i=1}^{k}\prod_{j=1}^{l+1}(\beta_j-i\alpha).
	\end{equation}
\end{corollary}

For small $k$ and $l$ equation \eqref{Morin1dim-source} is enough to calculate the stable Thom polynomial $\Tp(A_k(l))$. Since $\rho^{1\to l+1}$ has no kernel below degree $2(l+2)$, we can calculate $A_2(l)$ and $A_3(0)$. For $\Tp(A_3(1))$ the coefficient of $s_{2,2,2}$ will not be determined from $[A_3(1,2)]$.

\begin{proposition} We have the following stable Thom polynomials.
	\begin{enumerate}[(i)]
		\item Ronga's theorem:
\[\Tp(A_2(l))=\sum_{i=0}^{l+1}2^is_{2^{l+1-i},1^{2i}}.\]
		\item 
		\[\Tp(A_3(0))=6s_{1,1,1}+5s_{2,1}+s_3\]
		\item \[ \Tp(A_3(1))=36s_{1^6}+30s_{2,1^4}+6s_{3,1,1,1}+19s_{2,2,1,1}+5s_{3,2,1}+s_{3,3}+vs_{2,2,2}, \]
		where the coefficient of $s_{2,2,2}$ is not determined by $[A_3(1,2)]$. In fact, it is known that $v=5$.
	\end{enumerate}
\end{proposition}
\begin{proof}
	\emph{Case of $A_2(l)$:}
\[ [A_2(1,l+1)]=\prod_{j=1}^{l+1}(\beta_j-\alpha)\prod_{j=1}^{l+1}(\beta_j-2\alpha)=
\prod_{j=1}^{l+1}(\beta_j-\alpha)\sum_{i=0}^{l+1}(-1)^i2^i\alpha^ie_{l+1-i}(\beta_j), \]
where $e_{l+1-i}(\beta_j)$ denotes the $l+1-i$-th elementary symmetric polynomial of the $\beta_j$'s. Therefore, by the Factorization formula we obtain the formula of Ronga.	\smallskip

\emph{Case of $A_3(0)$}: 
\[ [A_3(1,1)]=(\beta-\alpha)(\beta-2\alpha)(\beta-3\alpha)=(\beta-\alpha)(6\alpha^2-5\alpha\beta+\beta^2),\]
and by the Factorization formula we obtain (ii).\smallskip

\emph{Case of $A_3(1)$}:
	\begin{equation*}\label{a3(1)}
		\begin{split}
			[A_3(1,2)]=(\beta_1-\alpha)(\beta_2-\alpha)(\beta_1-2\alpha)(\beta_2-2\alpha)(\beta_1-3\alpha)(\beta_2-3\alpha)=\\
			(\beta_1-\alpha)(\beta_2-\alpha)\big(36\alpha^4-30s_1(\beta)\alpha^3+(6s_2(\beta)+19s_{1,1}(\beta))\alpha^2-5\al s_{2,1}(\beta)+s_{2,2}(\beta)\big),
		\end{split}
	\end{equation*}
	where the $s_\lambda(\beta)$'s denote the Schur polynomials in the variables $\beta_1,\beta_2$, e.g.
	\[ s_2(\beta)=(\beta_1+\beta_2)^2-\beta_1\beta_2  \]
	and so on. Therefore, by the Factorization formula we obtain (iii) -- note that the coefficient of $s_{2,2,2}$ is not determined by the Factorization formula.

	The fact that  $[A_3(1,2)]$ is divisible by $(\beta_1-\alpha)(\beta_2-\alpha)$ implies that the coefficient of $s_6$ is zero by (iii) of the Factorization formula Theorem \ref{thm:schur-factorization}.
\end{proof}

\begin{example} For $A_4(1,1)$ we have
\[ [A_4(1,1)]=(\beta-\alpha)(\beta-2\alpha)(\beta-3\alpha)(\beta-4\alpha)=(\beta-\alpha)(-24\alpha^3+26\alpha^2\beta-9\alpha\beta^2+\beta^3),\]
therefore, by the Factorization formula we obtain
\[ \Tp(A_4(0))=24s_{1,1,1,1}+26s_{2,1,1}+9s_{3,1}+s_4+us_{2,2}, \]
where the coefficient of $s_{2,2}$ (which is $u=10$ \cite[Theorem 2.2]{Gaffney1983}) is not determined by $[A_4(1,1)]$.
\end{example}

\begin{remark} If we increase $l$ by one then the kernel of $\rho^{1\to l+1}$ starts in 2 degrees higher, and the degree of $\Tp(A_k(l))$ increases by $k$. This implies that for $k>2$ the one dimensional source space method leaves more and more coefficients of $\Tp(A_k(l))$ undetermined. For the case of negative relative codimension, we will explore the \quot{dual} notion, i.e.\ when the target space is one dimensional, which behaves very differently.
\end{remark}

\begin{remark} For higher dimensional source space we can see that the factorization formula immediately implies the Giambelli-Thom-Porteous formula.
\end{remark}

\section{Thom series in non-negative relative codimension }
\label{sec:p>n} 
A key goal of Thom polynomial theory is to establish closed formulas for families of  Thom polynomials. The classical example is that iterated trivial unfoldings of a singularity are the image of the same stable Thom polynomial (Theorem \ref{damon}). A new parameter can be introduced by the \emph{zero unfolding operator}. This leads to Thom series for $l\geq 0$ that we briefly recall.

\subsection{The zero unfolding}\label{sec:zerounfolding}
For  $f=(f_1,\dots,f_p)\in J^k(n,p)$ we define its zero unfolding as $\sigma_0(f):=(f_1,\dots,f_p,0)$. Zero unfolding doesn't change the quotient algebra, but increases the relative codimension $l=p-n$ by one. This means that given a singularity $\eta(l_0)$, there are  singularities $\eta(l)$ for $l\geq l_0$. Notice that this convention is consistent with the notation $A_i(l)$ introduced earlier. See Table \ref{fig:stabilizations} for a summary of stabilizations and the corresponding stable Thom polynomials.

\begin{table}
	\[
	\renewcommand{\arraystretch}{1.35}
	\begin{array}[t]{@{}l@{}}
		\makebox[0pt][l]{%
			\raisebox{2.2ex}[0pt][0pt]{%
				\(\scriptstyle\xrightarrow{\;\text{trivial unfolding}\;}\)
			}%
		}%
		\makebox[0pt][r]{%
			\raisebox{1.2ex}[0pt][0pt]{%
				\rotatebox[origin=lt]{-90}{%
					\(\scriptstyle\xrightarrow{\;\text{zero unfolding}\;}\)
				}%
			}%
			\hspace{0.5em}%
		}%
		\begin{array}[t]{c|ccc|c}
			& n=1 & n=2 & n=3 & \operatorname{Tp}(\ell)\\
			\hline
			\ell=0
			&
			x \mapsto x^3
			&
			(x,u) \mapsto (x^3,u)
			&
			(x,u,v) \mapsto (x^3,u,v)
			&
			s_2+2s_{1,1}
			\\[1.2em]
			
			\ell=1
			&
			x \mapsto (x^3,0)
			&
			(x,u) \mapsto (x^3,u,0)
			&
			(x,u,v) \mapsto (x^3,u,v,0)
			&
			s_{2,2}+2s_{2,1,1}+4s_{1^4}
			\\[1.2em]
			
			\ell=2
			&
			x \mapsto (x^3,0,0)
			&
			(x,u) \mapsto (x^3,u,0,0)
			&
			(x,u,v) \mapsto (x^3,u,v,0,0)
			&
	s_{2,2,2}+2s_{2,2,1,1}+4s_{2,1^4}+8s_{1^6}
		\end{array}
	\end{array}
	\]
	\caption{The positive codimension stabilizations of $A_2(1,1)$ (their $\mathcal{K}$-orbits) and their stable Thom polynomials.} \label{fig:stabilizations}
\end{table}
\subsection{The Thom series: $\ell\geq 0$} The Thom polynomials of the sequence of zero unfoldings of a given $\eta$ follow a pattern, they can be arranged into the so called Thom series.
\begin{examples}
	Here we list a few simpler examples of Thom series; stable Thom polynomials as $l$ varies. Giambelli-Thom-Porteous:
\[   \Tp(\Sigma^{i}(l))=s_{i^{i+l}},\]
\[  \Tp(A_2(l))=s_{2^{l+1}}+2s_{2^{l},1,1}+\cdots+2^{l+1}s_{1^{2l+2}},\]
\[    \Tp(III_{2,3}(l))=\sum_{i=0}^{l+1}2^{i+l}s_{3^{l+1-i},2^{i+1},1^i}.\]
Sources for these results are \cite{Ronga1972}, \cite{Ozturk}. For further Thom series computations see \cite{LascouxPragacz}, \cite{FeherRimanyi2012}, \cite{BercziSzenes}, \cite{Kazarian}.
\end{examples}

\begin{remark} The general rule is that if $\Tp(\eta(l+1))=\sum_{\lambda} e_\lambda s_\lambda$ then $\Tp(\eta(l))=\sum_{\lambda} e_\lambda s_{\lambda^\flat}$,
where

\[ \twocase{\lambda^\flat=}{0}{\lambda_1<\mu}{(\lambda_2,\dots)}{\lambda_1=\mu},\]
and $\mu$ is the dimension of the quotient algebra of $\eta$ (for this fact and for the definition of the quotient algebra see \cite{FeherRimanyi2012}). This shows that as $l$ increases the \quot{old coefficients} don't change but new coefficients appear. We know that $\Tp(\eta(\binom{\mu-1}{2}))$ contains all the information but the computation of such stable Thom polynomials is not straightforward. 
\end{remark}
\section{Thom series in negative relative codimension}
\label{sec:p<n}
The zero unfolding operator increases the codimension of the singularity by the dimension of the quotient algebra. Singularities in negative relative codimension have infinite dimensional quotient algebra (since it can be identified with the maximal ideal of the coordinate ring of the preimage of zero, which is positive dimensional now). Therefore the zero unfolding of a singularity with  negative relative codimension has no Thom polynomial. Therefore we need to study other families of singularities.

\subsection{The quadratic unfolding} \label{sec:quadratic-unfolding} If $\eta:\C^n\to\C$ is a function singularity then we can define its quadratic unfolding $\qu(\eta):=\eta+x_{n+1}^2:\C^{n+1}\to\C$. For repeated quadratic unfoldings of $\eta$ we use the notation $\eta(-r)$, where $-r$ is the relative codimension, so $r$ is a nonnegative number.

One possible justification for studying the quadratic unfolding is that it doesn't change the Tjurina algebra of the singularity (which coincides with the unfolding space), which is a complete invariant of given source dimensional function singularities\cite{MatherYau1982}. This is analogous to the nonnegative relative codimension case, where the quotient algebra is a complete invariant of given relative codimensional singularities and the zero unfolding does not change the quotient algebra. 
  
 Since the unfolding space doesn't change, quadratic unfolding increases the codimension---and the degree of the Thom polynomial---by one, i.e. for every function singularity $\eta$ there is a non negative integer $\gamma=\gamma(\eta)$ such that 
 \[\codim(\eta(-r))=r+\gamma.\]
If a singularity appears in an $m$-parameter modulus, then we will consider the closure of the union of the orbits to be $\eta$. For non-moduli $\eta$, $\gamma(\eta)$ is the Tyurina number \cite[p.\ 112]{GreuelLossenShustin}, traditionally given in the lower index of the notation, e.g.\ $\gamma(A_k)=k$. For an $m$-parameter modulus $\eta$, the lower index is not the same as $\gamma(\eta)$, the lower index is $\gamma(\eta)+m$.

For non-function singularities the replacement for quadratic unfolding is less clear. We explain what we know in Section \ref{sec:beyond}.

\begin{table}
	\[
	\renewcommand{\arraystretch}{1.35}
	\begin{array}[t]{@{}l@{}}
		\makebox[0pt][l]{%
			\raisebox{2.2ex}[0pt][0pt]{%
				\(\scriptstyle\xrightarrow{\;\text{trivial unfolding}\;}\)
			}%
		}%
		\makebox[0pt][r]{%
			\raisebox{1.2ex}[0pt][0pt]{%
				\rotatebox[origin=lt]{-90}{%
					\(\scriptstyle\xrightarrow{\;\text{quadratic unfolding}\;}\)
				}%
			}%
			\hspace{0.5em}%
		}%
		\begin{array}[t]{c|cc|c}
			& p=1 & p=2 & \Tp(\eta) \\
			\hline
			\ell=0
			&
			x \mapsto x^3
			&
			(x,u) \mapsto (x^3,u)
			&
s_2+	2s_{1,1}
			
			\\[1.2em]
			
			\ell=-1
			&
			(x,y) \mapsto x^3+y^2
			&
			(x,y, u) \mapsto (x^3+y^2,u)
			&
		2s_3+2s_{2,1}
			\\[1.2em]
			
			\ell=-2
			&
			(x,y,z) \mapsto x^3+y^2+z^2
			&
			(x,y,z,u) \mapsto (x^3+y^2+z^2,u)
			&
3s_{4}+		2s_{3,1}

		\end{array}
	\end{array}
	\]
	\caption{Stabilizations of $A_2(1,1)$ in negative codimension, and their stable Thom polynomials. (In Schur basis the formulas are smaller.)}\label{fig:neg_stabilizations}
\end{table}

\subsection{The Thom series: $\ell< 0$} The Thom polynomials of the sequence of the quadratic unfoldings of a given $\eta$ follow a pattern, which is different from the zero unfolding case. Before stating the general theorem, we give an example of a Thom series to illustrate the general phenomenon: recall that $A_3(0)=[x\mapsto x^4]$. For $r\geq 1$:
\[  \Tp(A_3(-r))=6s_{r+1,1,1}+6s_{r+1,2}+(6r+5)s_{r+2,1}+\frac{1}{2}(r+1)(3r+2)s_{r+3}  
\]
We can observe the following: the number of Schur polynomials doesn't change but the coefficients are no longer constants: they are polynomials in $r$. We discovered this by computing these Thom polynomials via restriction equations, see Section \ref{sec:restriction}. Later we realized that there are---somewhat hidden---earlier results, via Thom-Boardman classes \cite[\S 6]{FeherKomuves}.

\begin{theorem}\label{thm:poly4p=1} 
	For a stable function singularity $\eta$ the stable Thom polynomial $\Tp(\eta(-r))$ is a polynomial in $r$ in the following sense:
	
	\[ \Tp(\eta(-r))=\sum_{|\alpha|+i+1=\gamma} g_{\alpha,i}(r)s_{r+1+i,\alpha},\]
	where $i\geq 0$, $\al=(\al_2\stb \al_n)$ a partition and $g_{\alpha,i}(r)$ are polynomials in $r$ of degree $\leq i$. We call the right hand side the \emph{Thom series of $\eta$}.
	
	\begin{itemize}
		\item This formula is valid for $r\geq \ga-2$ (In this case, for all $(\al,i)$ in the indexing set, $(r+1+i,\al)$ is a partition). 
		\item For $r\leq\gamma -2$, the sum is meant to exclude the terms $g_{\al,i}(r)s_{r+1+i,\al}$ where $(r+1+i,\al)$ is not a partition (i.e.\ $r+1+i<\al_2$). With this convention, the equality also holds for $r=\gamma-3$.
		\item For $r<\gamma-3$, the equality holds modulo the kernel of $\rho^{r+1\to 1}$. In particular, for every nice partition $s_{r+1+i,\alpha}$ (i.e.\ $\al_2\leq r+1$) its coefficient is equal to $g_{\al,i}(r)$. 
	\end{itemize}

\end{theorem}
We demonstrate the subtle points on validity of the theorem on a small example, and discuss further examples later.
\begin{example}
	We will see later (Theorem \ref{thm:thom_series}) that 
	\begin{align*}
		\Tp(A_4(-r))=&24s_{r+1,1,1,1}
		+36s_{r+1,2,1}
		+24s_{r+1,3}
		+(30r+26)s_{r+2,1,1}\\
		&+{\color{red}(30r+14)}s_{r+2,2}+(15r^2+20r+9)s_{r+3,1}+\frac{1}{2}(r+1)(5r^2+5r+2)s_{r+4},
	\end{align*}
	where the range of validity is described as follows. In this case, $\ga=4$. 
	
	By the first point of the theorem, the formula holds for $r\geq 2$. 
	
	The second point of the theorem states that for $r=1$, the formula holds if we omit the non-partition $s_{r+1,3}$. 
	
	Finally, by the third point of the theorem, for $r=0$ the equality holds up to the kernel of $\rho^{r+1,1}$, i.e.\  the coefficient of $s_{2,2}$ is not determined by the theorem (non-partitions are excluded by the second point). And indeed, $14$ is not the correct value for the coefficient of $s_{2,2}$ (in red), which can be shown to be 10.
	\end{example}
\begin{remark}\label{rmk:interpolate}
	Notice that by the first point of the theorem, the ``support'' of the Thom series is finite: for $r\geq \ga-2$, there is a fixed, finite number of partitions whose coefficients are nonzero. Indeed, as $r$ increases, the codimension of $\eta(-r)$ increases by one, and this dependence is covered by $s_{r+1,i,\al}$. Also, from a computational perspective, the knowledge of $i+1$ many values of $g_{\al,i}(r)$ allows one to interpolate the entire polynomial. 
	
	For small values of $r$, the theorem doesn't cover the remaining coefficients, and we will use different methods to compute these, see Section \ref{sec:restriction}.
\end{remark}
\section{From Legendre Thom polynomials to Thom polynomials}
\label{sec:ltp-to-tp}
For any stable contact function-singularity $\eta(n,1)$, there exists a Legendre version of Thom polynomials $K_\eta\in \Z[k_1,k_2\stb u]$, called the \emph{Legendre Thom polynomial} \cite{Vassiliev1993}, \cite{Kazarian1995}, \cite{Kazarian2003}. Kazarian described their relationship to the Thom polynomials $[\eta(n,1)]$ in a structure theorem \cite{Kazarian2003}, as we will now explain.
\subsection{The homomorphism $L_n$}
We first introduce the key algebraic operation that forms the basis of Kazarian's structure theorem.
\begin{definition}\label{def:Ln}
	Define the ring homomorphism
	\[
	L_n:\Z[k_1, k_2\stb u]\to \Z[a_1\stb a_n,u]
	\]
	by $L_n(u)=u$ and the formal power series
	\[
	L_n(1+k_1+k_2+\ldots)=\frac{(1+u)^n-(1+u)^{n-1}a_1\pm\ldots +(-1)^na_n}{1+a_1+\ldots+a_n}.
	\]
	whose graded terms are set to be equal ($\deg(k_i)=\deg(a_i)=i$, $\deg(u)=1$.)
\end{definition}
By its definition, $L_n$ has a nontrivial kernel, with the following relations between the $k_i$'s:
\[ \tag{R}
\left(1+k_{1}+k_{2}+\ldots\right)\left(1-\frac{k_{1}}{1+u}+\frac{k_{2}}{(1+u)^{2}}-\ldots\right)=1.\]
or written using the generating function $T(w)=1+k_1w+k_2w^2+\ldots$ with the formal variable $w$, the relation is \[T(w)T\left(\frac{-w}{1+uw}\right)=1.\] 
\begin{example}
	For $n=2$ we have
	\[1+k_1+k_2+\ldots=\frac{(1+u)^2-(1+u)a_1+a_2}{1+a_1+a_2}, \ \text{ i.e. } \  L_2(k_1)=2u-2a_1,\ L_2(k_2)=u^2-3ua_1+2a_1^2. \]
\end{example}
The definition and the polynomiality of the binomial coefficients immediately implies that
\begin{lemma} \label{small-n=0}
		$L_n(k_d)$ is a polynomial in $n$ in the following sense: There exist polynomials
\[	c_{I,q}(n)\in \mathbb Q[n],
	\qquad
	\deg_n c_{I,q}(n)\le q.
	\]
	such that
	\[
	L_n(k_d)=\sum_{|I|+q=d} c_{I,q}(n)\,a^Iu^q,
	\]
	where \(I=(I_1,I_2,\ldots)\), \(a^I=\prod_j a_j^{I_j}\), and
	\[
	|I|=\sum_{j\ge 1} jI_j,
	\]	
	with the convention that $a_j=0$ for $j>n$.	
\end{lemma}

For example  
\begin{equation} \label{eq:L_n(k_1)}
	L_n(k_1)=-2a_1+nu,\quad  \text{ and }\quad  L_n(k_2)=2a_1^2-(2n-1)a_1u+\binom n2 u^2.
\end{equation}

\begin{remark} In fact it is not difficult to show that $\deg_n c_{I,q}(n)=q$ if $I=(0,0,\ldots)$, or $I_j\neq 0$ for some odd $j$, otherwise $\deg_n c_{I,q}(n)=q-1$. For instance, in the example above, the coefficient of $a_2$ in $L_n(k_2)$ vanishes, i.e.\ $c_{(0,1),0}(n)=0$ which has degree $-1$.
\end{remark}

\subsection{Kazarian's structure theorem}
\begin{theorem}[Kazarian \cite{Kazarian2003}]\label{thm:kazarian} For any stable contact function-singularity $\eta(n,1)$ we can assign a polynomial $K_\eta\in\Z[k_1,k_2,\dots,u]/(R)$, such that
	\[[\eta(n,1)]=e(J^1(n))L_n(K_\eta),\]
where $L_n$ and (R) are defined in Definition \ref{def:Ln}. Here $J^1(n)=\Hom(\C^n,\C)$ and $e(J^1(n))$ denotes the $\GL(n)\times\GL(1)$-equivariant Euler class
   	\[ e(J^1(n))=\sum_{i=0}^{n}(-1)^ia_iu^{n-i},\]
   	and $u=b_1$.
\end{theorem}

Here $K_\eta$ is the \emph{Legendre Thom polynomial} of $\eta$. As Kazarian points out, every polynomial in the $k_i$ and $u$ variables is equivalent modulo $(R)$ to one which is square-free in the $k_i$'s, so we can look for these polynomials $K_\eta$ in a square free form.

 Kazarian calculated the polynomials $K_\eta$ for function singularities of low codimension in \cite[Table 1]{Kazarian2003}. Later he calculated more examples which were published in \cite[\S 11]{MikoszPragaczWeber2011}. 
\begin{example}
For example $K_{A_1}=1, \ K_{A_2}=k_1,\ K_{A_3}=3k_2+uk_1$. Applying the theorem we obtain for example that
\[[A_3(2,1)]=(u^2-ua_1+a_2)(5u^2-11ua_1+6a_1^2).\]	
\end{example}

Based on Kazarian's structure theorem, we now turn to the proof of the main Theorem \ref{thm:poly4p=1}.

\subsection{Proof of Theorem \ref{thm:poly4p=1}}
\begin{proposition}\label{prop:poly} For any stable contact function singularity $\eta\in J(n,1)$ there exist polynomials $g_{\al,i}(n)$ of degree $\leq i$ such that for all $n\geq 1$
	\begin{equation}\label{eq:etan1}
		[\eta(n,1)]=\rho^{n\to 1}\left(\sum_{\substack{|\al|+i+1=\ga\\ i\geq0}} g_{\alpha,i}(n)s_{n+i,\alpha}\right),
	\end{equation}
	where $\al$ are partitions, $\ga=\ga(\eta)$ and we use the convention that $s_{n+i,\alpha}=0$ if $(n+i,\alpha)$ is not a partition (we do not apply the straightening law!).
\end{proposition}
\begin{proof}
Lemma \ref{small-n=0} implies that the classes $L_n(K_\eta)$ are  polynomials in $n$:
\[
 L_n(K_\eta)=\sum_{|\la|+i+1=\ga} h_{\lambda,i}(n)s_{\lambda}(a)u^i,
\]
where $h_{\lambda,i}(n)$ is a polynomial in $n$ of degree  $\leq i$. The equation is valid for all $n\geq 1$ since $s_{\lambda}(a)=0$ if $l(\lambda)>n$. By Kazarian's theorem, we have
\[[\eta(n,1)]=e(J^1(n))\left(\sum_{|\la|+i+1=\ga} h_{\lambda,i}(n)s_{\lambda}(a)u^i\right).\]
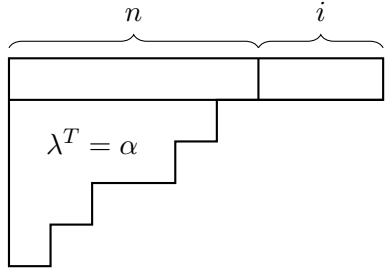
\begin{figure}
\begin{tikzpicture}[scale=0.55, every node/.style={font=\small}]
	\def\N{6} 
	\def\I{3} 
	\pgfmathtruncatemacro{\Top}{\N+\I}
	
	\def\rowA{\Top}
	\def\rowB{5}
	\def\rowC{4}
	\def\rowD{2}
	\def\rowE{1}
	
	\draw[thick]
	(0,0) --
	(\rowA,0) --
	(\rowA,-1) --
	(\rowB,-1) --
	(\rowB,-2) --
	(\rowC,-2) --
	(\rowC,-3) --
	(\rowD,-3) --
	(\rowD,-4) --
	(\rowE,-4) --
	(\rowE,-5) --
	(0,-5) --
	cycle;
	\draw[thick] (0,-1) -- (\Top,-1);
	\draw[thick] (\N,0) -- (\N,-1);
	
	\draw[decorate, decoration={brace, amplitude=5pt}]
	(0,0.25) -- (\N,0.25)
	node[midway, above=6pt] {$n$};
	
	\draw[decorate, decoration={brace, amplitude=5pt}]
	(\N,0.25) -- (\Top,0.25)
	node[midway, above=6pt] {$i$};
	
	\node at (2.0,-2.0) {$\lambda^T=\alpha$};
	
\end{tikzpicture}
\caption{If $\mu=(n+i,\al_2,\stb \al_q)$, then $\rho^{n\to 1}(s_\mu)=(-1)^{|\al|}e(J^1(n))s_{\la}(a)u^i$}\label{fig:factorization_n1}
\end{figure}
Noticing that for $A=\C^n$ and $B=\C$ we have $J^1(\C^n)=\Hom(A,B)$, by Remark \ref{rmk:HomAB} we can apply the Factorization Formula (Theorem  \ref{thm:schur-factorization}) for $p=1$ and obtain \eqref{eq:etan1} with the following conventions (see Figure \ref{fig:factorization_n1}): \[\la^T=\al=(\al_2\stb \al_q),\] 
i.e.\ $\al$ is the conjugate partition of $\la$, reindexed to start with $\alpha_2$, and $g_{\alpha,i}(n)=(-1)^{|\alpha|}h_{\lambda,i}(n)$.
\end{proof}
\begin{proof}[Proof of Theorem \ref{thm:poly4p=1}]
Write $n=r+1$. First, observe that in \eqref{eq:etan1} only partitions $\al$ appear with $\al_2\leq \ga-1$. So if $r+1\geq \ga-1$, then $(r+1+i,\al)$ is a partition. By Theorem \ref{damon} 
\[ \rho^{n\to 1}\Tp(\eta(1-n))=[\eta(n,1)].\]
Theorem \ref{thm:kazarian} shows that $[\eta(n,1)]$ is divisible by the Euler class, so by  Theorem \ref{thm:schur-factorization} only the thick coefficients  of $\Tp(\eta(-r))$ are undetermined. We claim that if $r\geq \ga -3$, then no $(n,1)$-thick partitions appear in \eqref{eq:etan1}. Indeed, a partition $(n+i,\alpha)$ is thick if and only if $\alpha_2\geq n+1=r+2$, and the partitions appearing in \eqref{eq:etan1} satisfy $\ga-1\geq \al_2$.
For $r\geq \ga-3$, we have
\[
\al_2\leq \ga-1\leq r+2,
\]
so thick partitions can only appear if $\al_2=\ga-1=r+2$. However, in Proposition \ref{prop:poly} the sum goes through $|\al|+i+1=\ga$, which implies $i=0$. The corresponding summand is indexed by $(n+i,\al)=(\ga-2,\ga-1)$, which is not a partition and therefore zero by the convention used in Proposition \ref{prop:poly}. This finishes the proof of Theorem \ref{thm:poly4p=1}. 
\end{proof}
The  coefficient polynomials can be explicitly given. Instead of a complicated formula we demonstrate this on the case of $A_3$:
\begin{example}  Using that $K_{A_3}=3k_2+uk_1$ and \eqref{eq:L_n(k_1)} we obtain
	
	\[\begin{aligned}
		L_n(K_{A_3})
		&=3L_n(k_2)+uL_n(k_1)\\
		&=3\left(2a_1^2-(2n-1)a_1u+\binom n2u^2\right)
		+u(-2a_1+nu)\\
		&=6a_1^2+(1-6n)a_1u+\frac{n(3n-1)}2u^2.
   	\end{aligned}\]
	
	Using that $a_1^2=s_2(a)+s_{1,1}(a)$ we get
	
	\[L_n(K_{A_3})
	=
	6s_2(a)+6s_{1,1}(a)+(1-6n)s_1(a)u+\frac{n(3n-1)}2u^2,\]
	
	and, applying the Schur Factorization Theorem \ref{thm:schur-factorization} we obtain
	
	\[
	\Tp(A_3(-r))
	=
	6s_{r+1,1,1}
	+
	6s_{r+1,2}
	+
	(6r+5)s_{r+2,1}
	+
	\frac{(r+1)(3r+2)}2s_{r+3},
	\]
	
	since $r=n-1$. In this case $\gamma-3=0$, so the formula is valid for all $r$.
\end{example}
\begin{remark} Notice the difference between the zero and the quadratic unfolding. As we explained in Section \ref{sec:1dimsource} the restriction $\rho^{1\to p}$ loses more and more information of $\Tp(A_k(p-1))$ for $k>2$ as $p$ increases: $\rho^{1\to p}$ is injective in degrees $<(1+1)(p+1)$ and the degree of $\Tp(A_k(p-1))$ is $pk$. 
	
	On the other hand, the restriction $\rho^{n\to1}$ is injective in degrees $<(n+1)(1+1)$, which increases by 2 when $n$ is increased by 1. In contrast, the degree of $\Tp(\eta(1-n))$ increases only by 1. Hence, for sufficiently large n, the class $[\eta(n,1)]$ determines $\Tp(\eta(1-n))$.
 \end{remark}

\section{The relationship between the three types of Thom polynomials}
\label{sec:thompolynomialsconnection}

Summarizing, given a function singularity $\eta$ we can assign several polynomials to it: the sequence of Thom polynomials $[\eta(n,1)]$ defined in \eqref{eq:TP}, the sequence of stable Thom polynomials $\Tp(\eta(1-n))$ (Theorem \ref{damon}) and the Legendre Thom polynomial $K_\eta$ (Theorem \ref{thm:kazarian}). The content of this section is to describe the relationship between these classes, see Figure \ref{fig:tprelations}.\smallskip

The relationship is the following:
\begin{enumerate} 
	\item $K_\eta$ determines $[\eta(n,1)]$ for all $n$ by Kazarian's Theorem \ref{thm:kazarian}. 
	\item The knowledge of $[\eta(n,1)]$ for some $n$ with $\ga(\eta)-1<\binom{n+2}{2}$ determines $K_\eta$, by Corollary \ref{l-factorization}, see the discussion below.
	\item The stable Thom polynomial $\Tp(\eta(1-n))$ determines $[\eta(n,1)]$, since 
	\[ \rho^{n\to 1}\Tp(\eta(1-n))=[\eta(n,1)],\]
	by Damon's Theorem \ref{damon}.
	\item  $[\eta(n,1)]$ determines $\Tp(\eta(1-n))$ if $n\geq \gamma(\eta)-2$. More generally, $[\eta(n+i,1+i)]$ determines $\Tp(\eta(1-n))$, if $n+\ga(\eta)\leq (n+i+1)(i+2)$.
\end{enumerate}

By the discussion in Section \ref{sec:factorization}, even if the condition $n\geq \ga(\eta)-2$  is not satisfied, the knowledge of $[\eta(n,1)]$ determines a subset of the coefficients of $\Tp(\eta(1-n))$: the $(n,1)$-nice Schur coefficients of $\Tp(1-n)$, while the thick coefficients are not determined.

A similar statement holds for $K_\eta$: even if the condition $\ga(\eta)-1<\binom{n+2}{2}$  is not satisfied, the knowledge of $[\eta(n,1)]$ determines a \emph{subspace} of the coefficients of $K_\eta$. Even more is true: in the basis of $Q$-polynomials, $[\eta(n,1)]$ determines a \emph{subset} of the coefficients of $K_\eta$, namely the $n$-nice $Q$-coefficients of $K_\eta$, while the thick coefficients are not determined. This in particular implies (2), since the lowest degree $n$-thick class appears in degree $\binom{n+2}{2}$. See Section \ref{sec:thm:q-factorization-formula} for the definitions, and Corollary \ref{l-factorization} for this statement.\smallskip

The Thom series (Theorem \ref{thm:poly4p=1}) and the Legendre Thom polynomial $K_\eta$ contain the same information: the Thom series was defined from $K_\eta$ in the proof of Proposition \ref{prop:poly}, and conversely, the Thom series determines $[\eta(n,1)]$ (Proposition \ref{prop:poly}) which for large enough $n$ determines $K_\eta$.\smallskip

\begin{figure}
	\centering
	\begin{tikzpicture}[>=stealth]
		\def\d{5.2}
		\node (A) at (0,0) {$\Tp(\eta(1-n))$};
		\node (B) at (\d,0) {$[\eta(n,1)]$};
		\node (C) at (2*\d,0) {$K_\eta$};
		
		\draw[->, bend left=25] (A) to[looseness=0.7]  node[above] {\tiny $\rho^{n\to1}$ (Thm.~\ref{damon})} (B);
		\draw[->, bend left=25] (B) to[looseness=0.7]  node[below,align=center] {\tiny Obs. \ref{obs:tptostable}\\\tiny if $n\geq \ga(\eta)-2$} (A);
		
		\draw[<-, bend left=25] (B) to[looseness=0.7] node[above] {\tiny $L_n$ (Thm.~\ref{thm:kazarian})} (C);
		\draw[<-, bend left=25] (C) to[looseness=0.7]  node[below] {\tiny if $\ga(\eta)-1<\binom{n+2}{2}$} (B);
	\end{tikzpicture}
	\caption{The relationship between the stable Thom polynomials, Thom polynomials and Legendre Thom polynomials.}
	\label{fig:tprelations}
\end{figure}
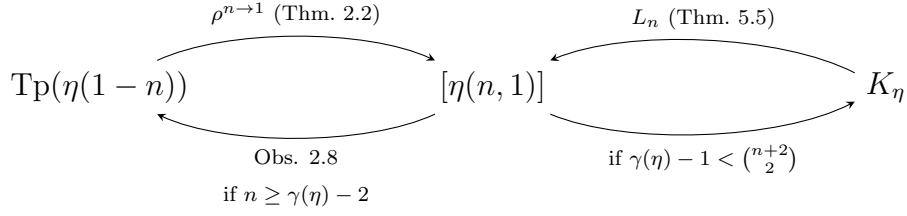

The relationship above (in particular (4)) also implies that even if $K_\eta$ is known, the stable Thom polynomial $\Tp(\eta(1-n))$ is not determined for small $n$. For several such cases, we calculated the stable Thom polynomials using the restriction equation method of Rim\'anyi. We will discuss these calculations in Section \ref{sec:restriction}.

\part{Computations}

\section{Singularity types in negative relative codimension}\label{sec:singularitytypes}
In this section we discuss the singularity types whose Thom polynomials we will consider. The types of singularities are discussed here, and the computations are carried out in the rest of the paper.

\subsection{Singularities given by a representative}
We are interested in the Thom series of contact singularities $\eta(-r)$ in negative codimension. Since contact singularities are $\mathcal{K}$-invariant subvarieties, the simplest ones are (closures of) $\mathcal{K}$-orbits. See Table \ref{fig:neg_singularities} for a list of these contact singularities up to codimension 7 (and some beyond). The $p=1$ cases are the function singularities whose classification up to contact equivalence is known up to codimension 10 \cite{Arnold1976}, \cite[p.\ 24]{AGLV}, \cite{AGZV}. 

The notation for the function singularities in Table \ref{fig:neg_singularities} is that 
\[
\codim \eta(-r)=\gamma(\eta)+r,
\]
for instance $\codim A_3(-4)=7$ (see Section \ref{sec:quadratic-unfolding}). For each stable singularity $\eta(-r)$, there is a smallest $r$ for which $\eta$ appears inside $J^k(n, n-r)$. Only the $A_k$ function singularities start from $r=0$. For other function singularities $r$ is at least 1, with $P_8, P_9$ starting from $r=2$ (note that these orbits appear in moduli). The stable Thom polynomials of these singularities for $\ell=-1$ have been computed up to codimension 7, \cite{Ando1996}, \cite{Ohmoto2009}. \\

We will compute the Thom series of these singularities in Theorem \ref{thm:thom_series}, based on Theorem \ref{thm:kazarian} and Observation \ref{obs:tptostable}. Note that for small $n$, Kazarian's theorem is not sufficient to determine the entire stable Thom polynomial $\Tp(\eta(1-n))$. To this end, we will use restriction equations, see Sections \ref{sec:small_r} and \ref{sec:restriction}.\smallskip

Non-function contact singularities also appear already for $\ell=-1$, the first one is $S_5(-1)$ in codimension 6. However the stabilization operation for non-function singularities is less clear, and we will not consider Thom series of such singularities in this paper. We explain what we know in Section \ref{sec:beyond}. 

\begin{table}[ht]
	\centering
	
	\begin{minipage}[c]{0.49\textwidth}
		\centering
		\begin{tikzpicture}[
			>=stealth,
			cell/.style={
				draw=black!35,
				rounded corners=1pt,
				minimum width=24mm,
				minimum height=8.5mm,
				inner sep=1.5pt,
				align=center,
				font=\scriptsize
			},
			head/.style={font=\scriptsize},
			jtag/.style={
				anchor=north east,
				text=red!70!black,
				font=\tiny,
				inner sep=1pt
			},
			arrowlab/.style={font=\scriptsize}
			]
			\matrix (M) [
			matrix of nodes,
			nodes in empty cells,
			row sep=1.3mm,
			column sep=1.5mm,
			ampersand replacement=\&
			] {
				\& |[head]| {$p=1$} \& |[head]| {$p=2$} \\
				|[head]| {$\ell=0$}
				\& |[cell]| {$A_k$}
				\& |[cell]| {$I_{k,l}$} \\
				|[head]| {$\ell=-1$}
				\& |[cell]| {$
					\begin{gathered}
						D_k,\ E_r\quad\\[-0.2em]
						X_q
					\end{gathered}$}
				\& |[cell]| {$S_k$} \\
				|[head]| {$\ell=-2$}
				\& |[cell]| {$P_8,\ P_9$}
				\& |[cell]| {$T_{a,b,c,d}$} \\
			};
			
			\node[jtag] at (M-2-2.north east) {$J(1,1)$};
			\node[jtag] at (M-2-3.north east) {$J(2,2)$};
			\node[jtag] at (M-3-2.north east) {$J(2,1)$};
			\node[jtag] at (M-3-3.north east) {$J(3,2)$};
			\node[jtag] at (M-4-2.north east) {$J(3,1)$};
			\node[jtag] at (M-4-3.north east) {$J(4,2)$};
			
			\draw[->]
			($(M-1-2.north west)+(0,1.8mm)$) --
			node[above,arrowlab] {trivial unfolding}
			($(M-1-3.north east)+(0,1.8mm)$);
			
			\draw[->]
			($(M-2-2.north west)+(-20mm,0)$) --
			node[midway,above,sloped,arrowlab] {quadratic unfolding}
			($(M-4-2.south west)+(-20mm,0)$);
			
		\end{tikzpicture}
	\end{minipage}
	\hspace{0.015\textwidth}
	\begin{minipage}[c]{0.47\textwidth}
		\centering
		\begin{tikzpicture}
			\node[
			draw=black!18,
			fill=black!2,
			rounded corners=2pt,
			inner xsep=3mm,
			inner ysep=2.5mm,
			text width=0.97\linewidth,
			align=left
			] {
				{\scriptsize
					\[
					\renewcommand{\arraystretch}{1.18}
					\begin{array}{@{}r@{\;}l@{\qquad}l@{}}
						A_k&: (x^{k+1})
						& k\geq 1, \\[0.15em]
						
						I_{k,l} &: (xy,\ x^k+y^l)
						& 2\leq k\leq l, \\[0.15em]
						
						D_k &: (x^2y+y^{k-1})
						& k\geq 4, \\[0.15em]
						
						E_6 &: (x^3+y^4)
						& \\[-0.05em]
						E_7 &: (x^3+xy^3)
						& \\[-0.05em]
						E_8 &: (x^3+y^5)
						& \\[0.15em]
						
						X_9 &: (x^4+\lambda x^2y^2+y^4)
						& \lambda^2\neq 4, \\[0.15em]
						
						X_q &: (x^4+x^2y^2+y^{q-5})
						& q\geq 10, \\[0.15em]
						
						S_k &: (x^2+y^2+z^{k-3},\ yz)
						& k\geq 5, \\[0.15em]
						
						P_8 &: (x^3+y^3+z^3+\lambda xyz)
						& \lambda^3+27\neq 0, \\[0.15em]
						
						P_9 &: (x^3+y^3+z^4+\lambda xyz)
						& \lambda\neq 0, \\[0.15em]
						
						T_{a,b,c,d} &: (xy+z^a+w^b,\ zw+x^c+y^d)
						& a,b,c,d\geq 2.
					\end{array}
					\]
				}
			};
		\end{tikzpicture}
	\end{minipage}
\bigskip	
	\caption{First appearances of contact singularities in $J(n,p)$, $\ell=p-n$. Subsequent occurrences are obtained by trivial and quadratic unfolding. In this paper we only consider singularities in the first column, and quadratic unfolding is also defined only there. }
	\label{fig:neg_singularities}
\end{table}
\subsection{Thom-Boardman classes}\label{sec:TB_discussion}
The next class of singularities that we consider are Thom-Boardman classes. These form a coarser stratification than contact orbits, and they are indexed by $\Sigma^I$, $I=(i_1\geq i_2\geq \ldots\geq i_k)$. For the definition, see \cite{Boardman1967} or \cite{Mather1973}. 

We will see that they provide an infinite family of function singularities. In negative codimension it is convenient to introduce the following \emph{dual notation} for Thom-Boardman classes:
\begin{definition}\label{def:TB}
	\[\check\Sigma^{s,J}(-r):=\Sigma^{s+r,J}(-r),\qquad r\geq 0,  s+r\geq j_1.\]
\end{definition}
In particular, the cases $\check\Sigma^{1,J}(-r)$ are representable as function singularities. Using this  dual notation, a key observation is (see \cite{Mather1973}):
\begin{observation}\label{obs:TBfunction}
	The quadratic unfolding of the Thom-Boardman class
	$\check{\Sigma}^{1,J}(-r)$ is $\check{\Sigma}^{1,J}(-(r+1))$, where $J=(j_1\stb j_k)$ is a partition with $j_1\leq r+1$.
\end{observation}

The codimensions of the Thom-Boardman classes are known, we use the notation of \cite[p.\ 46]{AGZV}.
\begin{theorem}[Boardman, \cite{Boardman1967}, 6.5] Let $I=(i_1,\dots,i_k)$ be a partition and let $\mu(I)$ denote the number of non-empty subpartitions of $I$. E.g. $\mu((i))=i$, $\mu((2,1))=4$. Then the codimension of $\Sigma^{I}(l)$ is
	\[\nu_I(l)=(l+i_1)\mu(i_1,\dots,i_k)-(i_1-i_2)\mu(i_2,\dots,i_k)-(i_2-i_3)\mu(i_3,\dots,i_k)-\cdots -(i_{k-1}-i_k)\mu(i_k).\]
\end{theorem}
For instance, using the dual notation we have
\begin{corollary}\label{cor:ThomBoardmankJ}
	\begin{align}\label{eq:ThomBoardmankJ}
 \codim\big(\check\Sigma^{k,J}(-(r+1))\big)&=\codim\big(\check\Sigma^{k,J}(-r)\big)+1+(k-1)(\mu(J)+1).\\
\codim \big(\check{\Si}^{1,J}(-(r+1))\big)&=\codim\big( \check{\Si}^{1,J}(-r)\big)+1\\
\codim \check{\Si}^{1,j}(-r)&=r+\binom{j+1}{2}+1
	\end{align}
\end{corollary}
\begin{proof}
	Use the straightforward identity
	\[ \mu(i+1,J)=\mu(i,J)+\mu(J)+1.\]
\end{proof}
As a special case of \eqref{eq:ThomBoardmankJ}, $J=\emptyset$ we obtain the sequence of $\check\Sigma^{k}(-r)$ with $\Tp(\check\Sigma^{k}(-r))=s_{(r+k)^k}$ (Giambelli-Thom-Porteous formula). 
\begin{remark} Despite the name there is no duality between the negative and the positive  relative codimension cases except for order one:
	\[\Tp(\check\Sigma^i(-r))=\Tp(\Sigma^{i+r}(-r))=s_{(i+r)^i},\] 
	while 
	\[\Tp(\Sigma^i(r))=s_{(i)^{i+r}}.\] 
\end{remark}

See Table \ref{fig:thomboardman} for the function singularities and their Thom-Boardman classes, \cite{Ando1992}, \cite[\S 3]{Ando2004}. The Thom series of Thom-Boardman classes of order two have been computed in \cite{FeherKomuves}. We review some of these results in Section \ref{sec:TB} and compute some Legendre Thom polynomials.

\begin{table}
	\[
	\begin{array}{c|ccccc}
		& A_k(-r)
		& D_k(-r)
		& E_6(-r)
		& E_7(-r)
		& E_8(-r)
		\\ \hline \\[-0.8em]
		\check{\Si}^{1,I}(-r)
		& \check{\Sigma}^{1^k}
		& \check{\Sigma}^{1,2}
		& \check{\Sigma}^{1,2,1}
		& \check{\Sigma}^{1,2,1}
		& \check{\Sigma}^{1,2,1,1}
		\\ \hline \\[-0.8em]
		\Si^{r+1,I}(-r)
		& \Sigma^{r+1,1^{k-1}}
		& \Sigma^{r+1,2}
		& \Sigma^{r+1,2,1}
		& \Sigma^{r+1,2,1}
		& \Sigma^{r+1,2,1,1}
		\\ \hline \\[-0.8em]
		\text{Open dense}
		& \text{Yes}
		& \text{For } k=4
		& \text{Yes}
		& \text{No}
		& \text{Yes}
	\end{array}
	\]
	
	\[
	\begin{array}{c|ccccc}
		& X_{9+k}(-r)
		& P_8(-r)
		& P_9(-r)
		& J_{10}(-r)
		& O_{16}(-r)
		\\ \hline \\[-0.8em]
		\check{\Si}^{1,I}(-r)
		& \check{\Sigma}^{1,2,2}
		& \check{\Sigma}^{1,3}
		& \check{\Sigma}^{1,3}
		& \check{\Sigma}^{1,2,1,1}
		& \check{\Sigma}^{1,4}
		\\ \hline \\[-0.8em]
		\Si^{r+1,I}(-r)
		& \Sigma^{r+1,2,2}
		& \Sigma^{r+1,3}
		& {\Sigma}^{r+1,3}
		& \Sigma^{r+1,2,1,1}
		& \Sigma^{r+1,4}

		\\ \hline \\[-0.8em]
		\text{Open dense}
		& \text{For }k=0
		& \text{Yes}
		& \text{No}
		& \text{No}
				& \text{Yes}
	\end{array}
	\]
	\caption{Function singularities and their Thom-Boardman classes. The last row denotes whether the contact orbit is open dense in its Thom-Boardman class. The first table contains simple singularities, and the second one consists of moduli of $\mathcal{K}$-orbits.}
	\label{fig:thomboardman} 
\end{table}
\subsection{Binary and ternary singularities}
For small $n$, we can compute the Thom polynomials $[\eta(n,1)]$ of another class of singularity types, which we now describe.

This singularity type is based on the linear inclusion $i:\pol^k(\C^2)\to J^k(2,1)$. The vector space $\pol^k(\C^2)$ admits a $\GL(2)$-invariant stratification into coincident root strata: For a partition $\lambda$ of $k$ the stratum $Y_\lambda$ contains the homogeneous degree $k$ polynomials where $\lambda$ describes the multiplicities of the linear factors.  The subvarieties $\eta_\lambda:=i(\overline{Y}_\lambda)$ are contact-invariant; indeed, since $j^{k-1}f=0$ for $\eta_\lambda$, contact transformations only contribute degree $>k$ terms.

Notice that for $\lambda=(1^k)$ the singularity class $\eta_\lambda(2,1)$ coincides with the closure of the Thom-Boardman class $\Sigma^{2^{k-1}}(2,1)$.  This is a special case of the fact that $\Sigma^{(r+1)^{k-1}}(-r)$ is the class of degree $k$ homogeneous polynomials in $r+1$ variables. Notice that in most cases there is no open orbit in this class, e.g.\ there is a moduli of orbits, which is open.

As special cases, we have $\eta_{1,1}=A_1(2,1)$, $\eta_{1,1,1}=D_4(2,1)$. The next example is
\[\eta_{1,1,1,1}=X_{1,0}(2,1)=X_9(2,1).\]
This is no longer a single orbit, the $j$-invariant of the four roots distinguishes the orbits in the one parameter family. Representatives of the family can be given as $x^4+ax^2y^2+y^4$, where $a^2\neq4$ to assure that the roots are different (called $X_9$ in \cite[p.\ 171]{AGZV}). See Table \ref{fig:binary} for a list of some further binary singularities and the corresponding singularity type.\\

\begin{table}
	\[
	\begin{array}{c|cccccccccc}
	\la	& 1^k
		& 1^2
		& 1^3
		& 1^4
		& 2
		& (2,1)
		& 3
		& (2,1,1)
		& (3,1)
		& (4)
		\\ \hline \\[-0.8em]
		& \rightarrow
		& A_1
		& D_4
		& X_{1,0}=X_9
		& A_2
		& D_5
		& E_6
		& X_{1,1}=T_{2,4,5}=X_{10}
		& Z_{11}
		& W_{12}
		\\ \hline \\[-0.8em]
		\check{\Si}^I:I
		& \check{\Sigma}^{1,2^{k-2}}
		& \check{\Sigma}^{1}
		& \check{\Sigma}^{1,2}
		& \check{\Sigma}^{1,2,2}
		& \check{\Sigma}^{1,1}
		& \check{\Sigma}^{1,2}
		& \check{\Sigma}^{1,2,1}
		&\check{\Sigma}^{1,2,2}
		&\check{\Sigma}^{1,2,2}
		&\check{\Sigma}^{1,2,2,1}
		\\\hline\\[-0.8em]
		\Si^I:I
		& \Sigma^{2^{k-1}}
		& \Sigma^{2}
		& \Sigma^{2,2}
		& \Sigma^{2,2,2}
		& \Sigma^{2,1}
		& \Sigma^{2,2}
		& \Sigma^{2,2,1}
	&{\Sigma}^{2,2,2}
	&{\Sigma}^{2,2,2}
	&{\Sigma}^{2,2,2,1}
		\\
	\end{array}
	\]
	\caption{Binary $J(2,1)$ singularities $\eta_\la$, the corresponding function singularity type and their Thom-Boardman class.}\label{fig:binary}
\end{table}

We also consider ternary singularities, which corresponds to the representation of $\GL_3$ on $\Pol^3(\C^3)$. We also consider an extension of binary singularities, where the locus in $J^k(\C^2)$ is described by several degree terms, sharing a root with different multiplicities, which we call \emph{shared-root loci} and the corresponding singularities \emph{multi-binary singularities} respectively; see Section \ref{sec:multibinary}.

For binary and ternary singularities, the structure allows us to compute specific Thom polynomials $[\eta(2,1)]$, but computing the stable Thom polynomials is not straightforward: the representation theoretic argument (the equivariant Euler class) used for computing $[\eta(2,1)]$ does not generalize when trying to compute the first trivial unfolding $[\eta(3,2)]$. From the computations of $[\eta(2,1)]$ we can compute new $\Tp(\eta(-1))$ and $K_\eta$, or at least specify some of their coefficients, cf.\ Section \ref{sec:thompolynomialsconnection}, see Section \ref{sec:binary} for the results.

	\section{Calculating the family of stable Thom polynomials of function singularities for $\gamma<7$}\label{sec:calculations}
The Thom polynomial $[\eta(n,1)]$ can be calculated whenever the Legendre Thom polynomial $K_\eta$ is known using Theorem \ref{thm:kazarian}.  Kazarian calculated $K_\eta$ up to codimension 8. This implies the Thom series by Theorem \ref{thm:poly4p=1}, which we compute in Section \ref{sec:kazariantp}. We determine the remaining stable Thom polynomials in Section \ref{sec:small_r} using Rim\'anyi's method of restriction equations.
\subsection{Kazarian's method}\label{sec:kazariantp}
Using Theorem \ref{thm:poly4p=1}, we obtain the following Thom series. 
\begin{theorem}\label{thm:thom_series}
	The Thom series of the first few contact singularities are as follows. The equalities are valid for $r\geq \ga-2$, and also for $\ga-3$ if we adopt the convention that if ${\lambda}$ is not a partition, then the term $s_\lambda$ is omitted. For $r<\ga-3$, the formulas are valid up to the kernel of $\rho^{r+1\to 1}$, such classes are colored red and marked with an asterisk. For further details, see Theorem \ref{thm:poly4p=1}: 
	\begin{alignat*}{3}
	\Tp(A_2(-r))
	&={}&{}&
	2s_{r+1,1}
	+(r+1)s_{r+2}
	&\qquad& (r\geq 0),\\
	\Tp(A_3(-r))
	&={}&{}&
	6s_{r+1,1,1}
	+6s_{r+1,2}
	+(6r+5)s_{r+2,1}+\frac{1}{2}(r+1)(3r+2)s_{r+3}
	&\qquad& (r\geq \ga-3=0),
	\displaybreak[3]\\[2.5ex]
	\Tp(A_4(-r))
	&={}&{}&
	24s_{r+1,1,1,1}
	+36s_{r+1,2,1}
	+24s_{r+1,3}
	&\qquad& (r\geq \ga-3=1),
	\\
	&&&
	{}+(30r+26)s_{r+2,1,1}
	+{\color{red}({30}r+14)^*}s_{r+2,2}
	&&
	\\
	&&&
	{}+(15r^2+20r+9)s_{r+3,1}
	&&
	\\
	&&&
	{}+\frac{1}{2}(r+1)(5r^2+5r+2)s_{r+4}
	&&
	\displaybreak[3]\\[2.5ex]
	%
	\Tp(D_4(-r)) &={}&{}&
	4 s_{r+1,2,1} + 2 r s_{r+2,1,1} + \left(2 r + 4\right) s_{r+2,2} &\qquad& (r\geq \ga-3=1),\\ &&&{}+r\left(r + 2 \right) s_{r+3,1} +\frac16r (r+1)(r+2) s_{r+4}	\\
	\Tp(D_5(-r))
	&={}&{}&
	24s_{r+1,2,1,1}
	+24s_{r+1,2,2}
	+24s_{r+1,3,1}
	&\qquad& (r\geq \ga-3=2),
	\\
	&&&
	{}+12r\,s_{r+2,1,1,1}
	+{\color{red}({12}r+24)^*}s_{r+2,3}
	+(36r+28)s_{r+2,2,1}
	&&
	\\
	&&&
	{}+2r(6r+7)s_{r+3,1,1}
	+2(r+2)(6r+1)s_{r+3,2}
	&&
	\\
	&&&
	{}+2r(r+2)(2r+1)s_{r+4,1}
	&&
	\\
	&&&
	{}+\frac{1}{6}r(r+1)(r+2)(3r+1)s_{r+5}
	&&
	\displaybreak[3]\\[2.5ex]
	\Tp(A_5(-r)) &={}&{}& 120\,s_{r+1,1,1,1,1} + 228\,s_{r+1,2,1,1} + 108\,s_{r+1,2,2} 	&\qquad& (r\geq \ga-3=2),\\
	&&&{}+228\,s_{r+1,3,1} + 120\,s_{r+1,4} + \left( 174r + 154 \right)\,s_{r+2,1,1,1} \\
	&&&{}+ {\color{red}\left( 282r + 152 \right)^*}\,s_{r+2,2,1} + {\color{red}\left( 174r + 22 \right)^*}\,s_{r+2,3} \\
	&&&{}+ \left( 114r^{2} + 153r + 71 \right)\,s_{r+3,1,1} + {\color{red}\left( 114r^{2} + 87r + 47 \right)^*}\,s_{r+3,2} \\
	&&&{}+ \left( 38r^{3} + 60r^{2} + 48r + 14 \right)\,s_{r+4,1} \\
	&&&{}+ \frac{1}{4}\,\left( r + 1 \right)\left( 19r^{3} + 21r^{2} + 16r + 4 \right)\,s_{r+5}.	\displaybreak[3]\\[2.5ex]
	\Tp(A_6(-r)) &={}&{}& 720\,s_{r+1,1,1,1,1,1} + 1632\,s_{r+1,2,1,1,1} + 1260\,s_{r+1,2,2,1} + 1980\,s_{r+1,3,1,1} &\qquad& (r\geq \ga-3=3),\\
	&&&{}+ 1260\,s_{r+1,3,2} + 1632\,s_{r+1,4,1} + 720\,s_{r+1,5} + 12\,\left( 98r + 87 \right)\,s_{r+2,1,1,1,1} \\
	&&&{}+ {\color{red}\left(2436r+1462 \right)^*}\,s_{r+2,2,1,1} + 14\,{\color{red}\left( 90r + 41 \right)^*}\,s_{r+2,2,2} \\
	&&&{}+ \left( 2436r + 526 \right)\,s_{r+2,3,1} + 12\,\left( 98r - 17 \right)\,s_{r+2,4} \\
	&&&{}+ \left( 903r^{2} + 1253r + 580 \right)\,s_{r+3,1,1,1} + {\color{red}\left( 1533r^{2} + 1281r + 659 \right)^*}\,s_{r+3,2,1} \\
	&&&{}+ {\color{red}\left( 903r^{2} + 161r + 314 \right)^*}\,s_{r+3,3}\\ 
	&&&{}+ \frac{1}{2}\,\left( 812r^{3} + 1267r^{2} + 1057r + 310 \right)\,s_{r+4,1,1}\\
	&&&{}+ {\color{red}\frac{1}{2}\,\left( 812r^{3} + 721r^{2} + 791r + 152 \right)^*}\,s_{r+4,2} \\
	&&&{} +\frac{1}{6}\,\left( 609r^{4} + 994r^{3} + 1113r^{2} + 560r + 120 \right)\,s_{r+5,1} \\
	&&&{}+ \frac{1}{120}\,\left( r + 1 \right)\left( 1218r^{4} + 1267r^{3} + 1533r^{2} + 602r + 120 \right)\,s_{r+6}.
	\displaybreak[3]\\[2.5ex]
	\Tp(D_6(-r)) &={}&{}& 
	96\,s_{r+1,2,1,1,1} + 144\,s_{r+1,2,2,1} + 144\,s_{r+1,3,1,1} + 144\,s_{r+1,3,2} + 96\,s_{r+1,4,1}
	&\qquad& (r\geq \ga-3=3),
	\\ 
	&&&{}+ 48 \,r \,s_{r+2,1,1,1,1} + 64\,\left( 3r + 2 \right)\,s_{r+2,2,1,1} + 16\,\left( 9r + 5 \right)\,s_{r+2,2,2} \\
	&&&{}+ {\color{red}64\,\left( 3r + 2 \right)^*}\,s_{r+2,3,1} + {\color{red}48\,\left( r + 2 \right)^*}\,s_{r+2,4} + 4\,r\,\left( 15r + 16 \right)\,s_{r+3,1,1,1} \\
	&&&{} + 12\,\left( 11r^{2} + 14r + 4 \right)\,s_{r+3,2,1} + {\color{red}4\,\left( r + 2 \right)\left( 15r - 2 \right)^*}\,s_{r+3,3} \\
	&&&{}+ 2\,r\,\left( 16r^{2} + 29r + 14 \right)\,s_{r+4,1,1} + 2\,\left( r + 2 \right)\left( 16r^{2} + 3r + 4 \right)\,s_{r+4,2} \\
	&&&{}+ \frac{2}{3}\,r\left( r + 2 \right)\left( 12r^{2} + 8r + 5 \right)\,s_{r+5,1} \\
	&&&{}+ \frac{4}{15}\,r\,\left( r + 1 \right)\left( r + 2 \right)\left( 3r^{2} + r + 1 \right)\,s_{r+6}.
	\displaybreak[3]\\[2.5ex]
	\Tp(E_6(-r))
	&={}&{}&
	24s_{r+1,2,1,1,1}
	+60s_{r+1,2,2,1}+{60s_{r+1,3,1,1}+60s_{r+1,3,2}+24s_{r+1,4,1}}
	&\qquad& (r\geq \ga-3=3),
	\\
	&&&
	{}+12r\,s_{r+2,1,1,1,1}
	+(72r+30)s_{r+2,2,1,1}
	+(60r+42)s_{r+2,2,2}
	&&
	\\
	&&&
	{}+{\color{red}(72r+78)^*}s_{r+2,3,1}
	+{\color{red}(12r+24)^*}s_{r+2,4}
	&&
	\\
	&&&
	{}+3r(7r+5)s_{r+3,1,1,1}
	+(51r^2+75r+3)s_{r+3,2,1}
	&&
	\\
	&&&
	{}+{\color{red}3(r+2)(7r+3)^*}s_{r+3,3}
	&&
	\\
	&&&
	{}+\left(
	39r
	+117\binom{r}{2}
	+72\binom{r}{3}
	\right)s_{r+4,1,1}
	&&
	\\
	&&&
	{}+\left(
	-3
	+57r
	+135\binom{r}{2}
	+72\binom{r}{3}
	\right)s_{r+4,2}
	&&
	\\
	&&&
	{}+3r^2(r+1)(r+2)s_{r+5,1}
	&&
	\\
	&&&
	{}+\left(
	3r
	+33\binom{r}{2}
	+93\binom{r}{3}
	+99\binom{r}{4}
	+36\binom{r}{5}
	\right)s_{r+6}.
	&&
	\end{alignat*}
\end{theorem}

\subsection{Stable Thom polynomials for small $r$}\label{sec:small_r}
We now discuss what these formulas look like for small $r$, i.e.\ the values of the coefficients in red for $\Tp(\eta(-r))$. In short, if the index $\la$  is not a partition, then $s_\la$ is interpreted as 0, and if for given $n$, the partition $\la$ is $(n,1)$-thick, then its coefficient is not determined by the Thom series, as stated in Theorem \ref{thm:poly4p=1}. Using Rim\'anyi's restriction equations (see Section \ref{sec:restriction}) and by computing the symmetries of the singularities we determined these missing coefficients for small $r$.   We now list the special cases where the coefficients of the thick partitions must be adjusted. Compare some of these formulas with \cite{PragaczWeber} which are the (unstable) Thom polynomials $[\eta(2,1)]$, which explains why the thick coefficients are not displayed there (e.g.\ the coefficients of $s_{3,3,1}$ and $s_{4,3}$ in $D_6(-1)$ as described below).

\localheading{$\bf A_4(-r)$}
In the formula for the Thom series of $A_4(-r)$, for $r=0$ (i.e.\ $(n,p)=(1,1)$) the sequences  $(r+1,3)$ and $(r+1,2,1)$ are not partitions,  so the corresponding coefficients are not defined. The partition $(r+2,2) =(2,2)$ is $(n,p)$-thick, so its coefficient is not determined by the theorem. Indeed, $30r+14=14$ is not equal to the correct value 10. 

For $r=1$ the condition $r\geq \gamma-3$ is satisfied, so all coefficients are correct. $(r+1,3)$ is not a partition, so the coefficient is not defined.	

\localheading{$\bf A_5(-r)$}
For $r=0$, the coefficient of $s_{r+3,2}$ is 35 instead of 47 and that of $s_{r+2,2,1}$ is 92 instead of 152; $s_{r+2,3}$, $s_{r+1,2,1,1}$, $s_{r+1,3,1}$, $s_{r+1,4}$ and $s_{r+1,2,2}$ are not partitions.

For $r=1$, the coefficient of $s_{r+2,3}$ is 124 instead of 196 and $s_{r+1,4}$ and $s_{r+1,3,1}$ are not partitions.

\localheading{$\bf D_5(-r)$}
The singularities $D_k$ appear for $r\geq 1$, so the first exceptional case is $D_5(-1)$. For $r=1$ the coefficient of $s_{r+2,3}$ is $24$ instead of $36$. 

\localheading{$\bf D_6(-r)$}
For $D_6$ there are two exceptional cases: $r=1,2$. For $r=1$, the coefficient of $s_{3,3,1}$ is  224 (instead of $64(3r+2)|_{r=1}=320$) and the coefficient of $s_{4,3}$ is 150 (instead of $4(r+2)(15r-2)|_{r=1}=156$). For $r=2$, the coefficient of $s_{4,4}$ is 160 instead of 192. 

\localheading{$\bf A_6(-r)$}
For $A_6$ there are three exceptional cases: $r=0,1,2$. For $r=0$, several coefficients mismatch and the correct value is (see e.g.\ \cite{tpp} -- note that there the transposed convention is used, cf.\ \cite[Appendix B]{FeherMatszangosz2025})
\begin{align*}
\Tp(A_6(0))	=&720 s_{1,1,1,1,1,1} + 1044 s_{2,1,1,1,1} + 770 s_{2,2,1,1} + 266 s_{2,2,2}{}& \\
&+{} 580 s_{3,1,1,1} + 455 s_{3,2,1} + 70 s_{3,3} + 155 s_{4,1,1} + 84 s_{4,2} + 20 s_{5,1} + s_{6}&
\end{align*}
For $\Tp(A_6(-1))$, the coefficient of $s_{3,3,1}$ is $1890$ (instead of 2962), the coefficient of $s_{4,3}$ is $1134$ (instead of $1378$). For $\Tp(A_6(-2))$, the coefficient of $s_{4,4}$ is 1540 (instead of $2148$).
\localheading{$\bf E_6(-r)$}
For $E_6$, there are two exceptional cases, $r=1$ and $r=2$. For $r=1$ the formula holds, with the exception that the coefficient of $s_{r+2,3,1}$ is $78$ and the coefficient of $s_{r+3,3}$ is $54$. Also, $(r+2,4)$ is not a partition, so that term doesn't appear.  The value of $\Tp(E_6(-1))=\Tp(\Sigma^{2,2,1}(-1))$ was first calculated by Kazarian via methods of \cite{Kazarian}, (private comm.). 

For $r=2$ the coefficient of $s_{r+2,4}$ is 48, which agrees with the value of $12r+24$. 
\begin{observation} \label{actual-expected}
	In all of the above examples, the value of the coefficient predicted by Theorem \ref{thm:thom_series} and the actual value of the coefficient satisfy $\it{actual}\leq \it{predicted}$. However it is not true for $X_9$, see Remark \ref{rem:X9}.
\end{observation}
\subsection{Restriction equations}\label{sec:restriction}
This method was pioneered by Rim\'anyi in \cite{Rimanyi}, and Thom polynomials and stable Thom polynomials were computed in \cite{PragaczWeber} and \cite{Ohmoto2009} respectively using this method for $l=-1$. We demonstrate the computation on the simple example of $A_2(-1)$, which demonstrates the idea of the method. The main input of the computation is the symmetry of a function singularity. Going beyond $l=-1$, the key information is how the symmetry changes under quadratic stabilization; we give details on this in Appendix \ref{sec:weights}.
\begin{example}
	To calculate $\Tp(A_2(-1))$ we need to restrict it to all singularities with codimension at most 3, the codimension of $A_2(-1)$. These are $A_0:(x,y)\mapsto x, A_1:(x,y)\mapsto xy$ and  $A_2:(x,y)\mapsto x^3+y^2$. 
	Write the unknown Thom polynomial of $A_2(-1)$ as 
	\[\Tp(A_2(-1))=k\cdot c_1^3+l\cdot c_1c_2+m\cdot c_3.\]
 The symmetry of $A_0$ has rank 2, with generators $t_1,t_2$; the weights are $(t_1,t_2;t_1)$. So the total Chern class is
	\[c(A_0)=\frac{1+t_1}{(1+t_1)(1+t_2)}=1-t_2+t_2^2-t_2^3+\ldots\]
	and the restriction is
	\[\Tp(A_2(-1))|_{A_0}=k\cdot c_1^3+l\cdot c_1c_2+m\cdot c_3=(-k-l-m)t_2^3,\]
	whose vanishing gives the first equation. 
	Next, the symmetry group of $A_1$ also has rank 2, and the weights are $(t_1,t_2;t_1+t_2)$, so  \[c(A_1)=\frac{1+t_1+t_2}{(1+t_1)(1+t_2)}=1+0-t_1t_2+t_1t_2(t_1+t_2)+\ldots\]
	and the restriction is
	\[\Tp(A_2(-1))|_{A_1}=m(t_1^2t_2+t_1t_2^2),\]
	whose vanishing implies $m=0$. Finally, restricting $A_2$ to itself gives the principal equation. The weights of the genotype give $(2t,3t;6t)$, so
	\[
	c(A_2)=\frac{(1+6t)}{(1+2t)(1+3t)}=1+t-11t^2+49t^3+\ldots.
	\]
	 To calculate the right hand side we need the symmetries of the  prototype (miniversal unfolding) of $ A_2$, which is $(x,y,u)\mapsto (x^3+y^2+ux,u)$. The weights are $(2t,3t,4t;6t, 4t)$, so
	 the principal equation $\Tp(A_2)|_{A_2}=e(A_2^{source})=2t\cdot 3t\cdot 4t=24t^3$ gives the equation
	\[
	(k-11l+49m)t^3=24t^3,
	\]
	which gives $k-11l=24$ and $k+l=0$, whose solutions are $l=-2$ and $k=2$, so
	\[
	\Tp(A_2(-1))=2c_1^3-2c_1c_2.
	\]
\end{example}
In summary, for any other singularity $\eta$ the homogeneous  equations are obtained by restricting to singularities of codimension less than or equal to the codimension of $\eta$. For these the genotype weights in Table \ref{tab:first-appearance} should be used. To compute the restriction to $\eta$, the right hand side is given by the product of the source weights of the prototype of $\eta$.

\section{Legendre Thom polynomials in the $Q$-basis} \label{sec:ltp-in-Q}
\newcommand{\QQ}{\tilde{Q}}
A key insight of \cite{MikoszPragaczWeber2009} and \cite{MikoszPragaczWeber2011} was to use a $Q$-polynomial basis for describing Legendre Thom polynomials; they satisfy positivity properties and they are more natural by their connection to the Thom polynomials $[\eta(n,1)]$ via the factorization formula as we will see.\smallskip

Before recalling their connection to Legendrian characteristic classes from \cite{MikoszPragaczWeber2009}, \cite{MikoszPragaczWeber2011}, we recall the definition and some basic properties of $Q$-polynomials. $Q$-polynomials were introduced by Schur in his study of projective representations of symmetric and alternating groups. Roughly speaking, Schur's $Q$-polynomials $Q_\la$ (indexed by strict partitions $\la$) play the same role in the Schubert calculus of Lagrangian Grassmannians as Schur polynomials $s_\la$ do in the Schubert calculus of Grassmannians. For further details, see \cite{LascouxPragacz}, \cite{PragaczRatajski}. We now recall the definitions.
\subsection{$Q$-polynomials}\label{sec:Q}

Let  $x_i$ be formal variables $i=1\stb n$. Define symmetric polynomials $Q_i$ by
\begin{equation}\label{eq:xi}
q(z):=1+Q_1z+Q_2z^2+\ldots=\prod_{i=1}^n\frac{1+x_iz}{1-x_iz}.
\end{equation}
Denoting the $j$-th elementary symmetric polynomial of the $x_i$ by $c_j$ we obtain
\begin{equation}
	q(z)=\frac{1+c_1z+c_2z^2+c_3z^3+\cdots}{1-c_1z+c_2z^2-c_3z^3+\cdots},
\label{eq:q_ci}	
\end{equation}
i.e.\ $Q_1=2c_1,\ Q_2=2c_1^2$,\ $Q_3=2(c_1^3-c_1c_2+c_3)$ etc. This shows that we can define $Q_i$ independently of the number of $x_i$ variables. \smallskip

More formally, define the \emph{ring of $Q$-polynomials on $n$-variables} $\Ga_n$ as the subring of symmetric polynomials generated by $Q_1,Q_2,\ldots$. Defining the \emph{ring of $Q$-functions} $\Ga$ to be the inverse limit of $\Ga_n$, one has that
\[\Gamma=\Z[Q_1,Q_2,\ldots]/q(z)q(-z)=1,\]
where $q(z)$ was defined above. Being an inverse limit, it has projections $\pi_n:\Ga\to\Ga_n$, with $\pi_n(Q_i)=Q_i$ (with a slight abuse of notation). 

Using \eqref{eq:q_ci}, $\Gamma$ is realized as a subring of $\Z[c_1,\dots,c_j,\dots]$ generated by the $Q_i$'s. Then $\pi_n$ is the restriction of evaluating $c_i$ as the $i$-th elementary symmetric polynomial of $x_1,\dots,x_n$.

\begin{remark}
	Setting $Q_i=c_i(S-S^\vee)$ (with $S$ over $BU$ for $\Ga$ and $S_n$ over $BU_n$ for $\Ga_n$), we are led to the usual definition of $Q$-polynomials, see \cite[III.\S 8]{MacDonald1995}. In fact, $\Ga$ is naturally identified with the cohomology of the infinite Lagrangian Grassmannian $Sp/U$.
	The connection is that $x_i$ are the Chern roots of $S$.
\end{remark}

The ring of $Q$-functions $\Ga$ is generated freely as a $\Z$-module by \emph{Schur's $Q$-functions $Q_\la$}, indexed by strict partitions $\la$ \cite[III.~\S 8]{MacDonald1995}:
\begin{definition}
	
	Set $Q_{(0)}=1$, $Q_{(i)}=Q_i$ and for $i\geq j\geq0$, \[Q_{(i,j)}=Q_iQ_j+2\sum_{s=1}^j(-1)^sQ_{i+s}Q_{j-s}\] and for strict partitions $\la$ set $Q_\la=\operatorname{Pf}(Q_{(\la_i,\la_j)})$, where $\operatorname{Pf}$ denotes the Pfaffian (after adjoining a zero part if $\ell(\lambda)$ is odd). 	
\end{definition}

\subsection{$Q$-polynomials and Legendre characteristic classes}\label{sec:QL}
One of the key insights of \cite{MikoszPragaczWeber2011} is that the ring of Legendre characteristic classes can be regarded as the ring of $Q$-polynomials twisted by the square-root of a line bundle (cf.\ also \cite{HarrisTu}). We now recall this construction.
Consider the ring isomorphism
\[
M:\Ga[t]\xrightarrow{\iso}\Z[k_1,k_2,\ldots, u/2]/(R)
\]
where if $k_i=c_i(K)$ and $u=c_1(B)$, then $M$ is defined by $t\mapsto c_1(\sqrt{B^\vee})=-u/2$ and\footnote{Taking the square root is formal, it simply means that we divide the Chern root by 2.} $Q_i\mapsto c_i(K\otimes \sqrt{B^\vee})$ (where $K$ is regarded as a rank zero virtual bundle), equivalently:
\[
M(q(z))=T\left(\frac{z}{1-zu/2}\right)
\]
where $T(w)=1+k_1w+k_2w^2+\ldots$. In other words
\[   M(Q_i)
=
\sum_{j=1}^i
\binom{i-1}{j-1}
\left(\frac u2\right)^{i-j}k_j, \qquad  M^{-1}(k_i)=\sum_{j=1}^i \binom{i-1}{j-1}t^{\,i-j}Q_j.
\]
\begin{proposition}[cf.\ {\cite[Remark 15]{MikoszPragaczWeber2011}}]\label{prop:inQ}
	The homomorphism $L_n$ fits into the following commutative diagram:
	\begin{equation}
	\begin{tikzcd}[column sep=large,row sep=large]
	\Gamma[t]
	\arrow[r,two heads,"\pi_n\otimes \id_{\Z[t]}"]
	\arrow[d,"\iso"',"{ M}"]
	&
	\Gamma_n[t]
	\arrow[d,hook,"\iota_E"]
	\\
	\mathbb Z[k_1,k_2,\ldots,u/2]/(R)
	\arrow[r,"L_n"]
	&
	\mathbb Z[a_1,\ldots,a_n,u/2].
	\end{tikzcd}
	\end{equation}
	where $\iota_E$ is defined by mapping $t\mapsto -u/2$ and by evaluating the variables $x_i$ of \eqref{eq:xi} at the Chern
	roots of $E^\vee$, where $E=A\otimes \sqrt{B^\vee}$, or equivalently,
	\[
	\iota_E(Q_i)=c_i(E^\vee-E).
	\]
\end{proposition}
\begin{proof}
	The map $\iota_E$ is injective, since after passing to Chern roots it is
	induced by the invertible change of variables $x_i=\alpha_i-u/2$, where
	$\alpha_i$ are the Chern roots of $A$.
	
	For commutativity, interpret the maps as pulling back virtual bundles. Write  $Q=S-S^\vee$, i.e.\ $Q_i=c_i(Q)\in \Ga$. Then
	\[
	M:Q\mapsto K\otimes \sqrt{B^\vee},\qquad L_n:K\mapsto (A^\vee\otimes B\ominus A)
	\]
	Then note that:
	\begin{equation}\label{eq:MPW}
	L_n(M(Q))=(A^\vee\otimes B\ominus A)\otimes \sqrt{B^\vee}=E^\vee\ominus E
	\end{equation}
	as required.
\end{proof}
We introduce the analogue of the transformation $L_n$ for $Q$-polynomials.
\begin{definition} Let $\mathcal{L}_n:=L_n\circ M$ and $\tilde K_\eta:=M^{-1}K_\eta$.
\end{definition}	
Then we have 
\begin{equation}\label{eq:L-tildeL}
	L_n(K_\eta)=\mathcal{L}_n(\tilde K_\eta)
\end{equation}
i.e.\ $K_\eta$ and $\tilde{K}_\eta$ encode the same Legendre Thom polynomial in different bases, and	
\begin{equation} \label{eq:lnq}
\mathcal{L}_n(Q_\lambda)=Q_\lambda\!\left(\frac{u}{2}-\alpha_1,\ldots,
\frac{u}{2}-\alpha_n\right),
\end{equation}
where $\alpha_i$ are the Chern roots of $A$, i.e. $a_i=e_i(\alpha_1,\dots,\alpha_n)$. Note that this specialization extends from $\Ga_n$ to the whole ring of symmetric polynomials (together with $\mathcal{L}_n(t)=-u/2$).

\subsection{Factorization for \texorpdfstring{$Q$}{Q}-polynomials}
\label{sec:thm:q-factorization-formula}

In this section we record the analogue of the usual Schur factorization formula Theorem \ref{thm:schur-factorization} for
Schur $Q$-polynomials.  Let
\[
\delta_n=(n,n-1,\ldots,2,1)
\]
denote the staircase partition. The image $\pi_n(Q_\lambda)$ behaves differently depending on the shape of the partition $\lambda$. So to express $\pi_n(Q_\lambda)$ we introduce the following:

\begin{definition}\label{def:q-thick}
	Let $n$  be fixed. Then strict partitions fall into three distinct categories: 
	\begin{enumerate}
		\item If $\delta_n \not\subset \lambda$ then we call $\lambda$ \emph{thin}.
		\item If $\delta_n \subset \lambda$ but $\de_{n+1}\not\subset\lambda $ then we call $\lambda$ \emph{nice}. 
		\item If  $\de_{n+1}\subset\lambda $ then we call $\lambda$ \emph{thick}.	
	\end{enumerate}
\end{definition}
	These names refer to the shape of the Young diagrams, analogously to Definition \ref{def:thick}. Notice that a nice partition can be uniquely written in the form  $(\delta_n +\nu)$, where $l(\nu)\le n$. Then the analogue of Theorem \ref{thm:schur-factorization} is the following:

\begin{theorem}\label{thm:q-factorization}\mbox{}   Let $\lambda$ denote strict partitions. The ring homomorphisms $\pi_n$ satisfy the following properties:
	\begin{enumerate}
		\item[(i)\,] (thick property) $\ker (\pi_n)=\Z\langle Q_\lambda: \lambda \text{ is thick}\rangle$, where $\langle\ \rangle$ means the generated $\Z$-module.
		\item[(ii)] (nice  property) For a nice $\lambda=(\delta_n +\nu)$,
		\[ 
		\pi_n(Q_\lambda)
		=\pi_n(Q_{\delta_n})s_\nu(x_1,\dots,x_n),\]
		where since $s_\nu\not\in \Ga_n$, part of the statement is that the product is an element of $ \Ga_n$. Note that
		\[
		\pi_n(Q_{\delta_n})=Q_{\delta_n}(x_1,\ldots,x_n)
		=\prod_{1\le i\le j\le n}(x_i+x_j).
		\]
		
		\item[(iii)] (thin property) Suppose that for $P\in \Gamma$ we have $\pi_n(P)=\pi_n(Q_{\delta_n})R$. Then in the $Q$-basis all thin coefficients of P are zero.
	\end{enumerate}
\end{theorem}
\begin{proof} The description of the kernel was given by Pragacz in \cite[Proposition 2.2]{Pragacz1991} The $Q$-factorization formula (ii) is due to Stanley \cite{Stanley1984Problem4}.
	
	For lack of a reference, we sketch a proof of (iii) 
	
	From \cite[ Ch.~III, \S 2, (2.6)]{MacDonald1995} (with $t=-1$ \cite[Ch.~III, \S 8]{MacDonald1995}) we have
	\[
	P_\lambda(x_1\stb x_n)
	=
	s_\lambda(x_1\stb x_n)+\sum_{\substack{\mu<\lambda,\\ |\mu|=|\lambda|}} w_{\lambda\mu}s_\mu(x_1\stb x_n),
	\qquad w_{\lambda\mu}\in\mathbb Z,
	\]
	with respect to the dominance order. We have $P_\lambda=2^{-\ell(\lambda)}Q_\lambda$, so
	\begin{equation} \label{eq:Q-in-Schur}
		Q_\lambda(x_1\stb x_n)
		=
		2^{\ell(\la)}s_\lambda(x_1\stb x_n)+\sum_{\substack{\mu<\lambda,\\ |\mu|=|\lambda|}} b_{\lambda\mu}s_\mu(x_1\stb x_n),
		\qquad b_{\lambda\mu}\in\mathbb Z.
	\end{equation}
	Since the Schur polynomials
	$s_\mu(x_1,\ldots,x_n)$, $\ell(\mu)\le n$, form a $\mathbb Z$-basis of
	symmetric polynomials in $n$ variables, the polynomials
	$P_\lambda(x_1,\ldots,x_n)$, $\ell(\lambda)\le n$, are linearly
	independent. Using that $Q_\lambda=2^{\ell(\lambda)}P_\lambda$ for strict $\lambda$, gives linear independence for $\ell(\la)\leq n$.
	
	The vanishing for $\ell(\la)>n$ follows from the same Schur expansion: only $\mu$ dominated
	by $\lambda$ occur; hence $\ell(\mu)\ge \ell(\lambda)>n$, so every
	$s_\mu(x_1,\ldots,x_n)$ vanishes.
\end{proof}

The $Q$-factorization Theorem \ref{thm:q-factorization} and \eqref{eq:lnq} implies the following
\begin{corollary}  \label{l-factorization}
	Let $\lambda$ denote strict partitions. The ring homomorphisms $\mathcal{L}_n$ satisfy the following properties:
	\begin{enumerate}
		\item[(i)\,] (thick property) $\ker (\mathcal{L}_n)=\Z[t]\langle Q_\lambda: \lambda \text{ thick}\rangle$, where $\langle\ \rangle$ means the generated $\Z[t]$-module.
		\item[(ii)] (nice  property) For a nice strict partition $\lambda=(\delta_n +\nu)$
		\[ 
		\mathcal{L}_n(Q_\lambda)
		=\mathcal{L}_n(Q_{\delta_n})\mathcal{L}_n(s_\nu).\]
		Note that
		
		\[
		\mathcal{L}_n(Q_{\delta_n})=
		\prod_{1\le i\le j\le n}(u-\al_i-\al_j),
		\]
		which is the $\GL(n)\times\GL(1)$-equivariant Euler class of $\pol^2(\C^n)$.
		
		\item[(iii)] (thin property) Suppose that for $P\in \Gamma[t]$ we have $\mathcal{L}_n(P)=\mathcal{L}_n(Q_{\delta_n})R$. Then in the $Q$-basis all thin coefficients of P are zero.
	\end{enumerate}
\end{corollary}
This corollary will be our main tool to (partially) calculate the Legendre Thom polynomial $\tilde K_\eta$ from $[\eta(n,1)]$ for various singularities.
\begin{remark}[Stable avoiding ideal]
	From another point of view, $(n,p)$-thick partitions are elements of the stable avoiding ideal of $\Si^n(l)$, for $l=p-n$ (for these notions, we refer to \cite{Pragacz1988}, \cite[Thm 2.3]{FeherRimanyi2004} and \cite[\S 4]{FeherMatszangosz2025}). Analogously, the corank $\geq {n+1}$ Legendre maps have Legendre Thom polynomial $Q_{\de_{n+1}}$. Furthermore, Pragacz in \cite[Thm 1.1]{Pragacz1996} showed that the symmetric degeneracy locus of corank $n+1$ has stable avoiding ideal generated additively as
	\[\mathcal{A}^s_{\geq n+1}=\Z\bra Q_\mu:\mu\supset \de_{n+1}\ket=\ker \pi_n\]
	where the second equality is the content of Theorem \ref{thm:q-factorization} (i). So while the Thom polynomial $[\eta(n,p)]$ determines the stable Thom polynomial up to elements of the stable avoiding ideal $\Si^n(l)$, the Thom polynomial of $[\eta(n,1)]$ determines the Legendre Thom polynomial up to elements of the stable avoiding ideal of corank $\geq n+1$ symmetric maps.	
\end{remark}
\section{Thom polynomials of binary, multi-binary and ternary singularities}\label{sec:binary}
In this section we compute the unstable Thom polynomials $[\eta(n,1)]$ of binary, multi-binary  and ternary singularities, and identify these with Thom polynomials of contact singularities. These Thom polynomials $[\eta(n,1)]$ are not always sufficient to deduce the stable Thom polynomial $\Tp(\eta(1-n))$, and especially the entire family of stable Thom polynomials $\{\Tp(\eta(-r))\}$, but it is a first step. On the other hand, this is another class of examples which gives an infinite family of Thom polynomials in negative codimension.

\subsection{Binary homogeneous singularities}

Recall from Section \ref{sec:singularitytypes} that the binary homogeneous singularities $\eta_\la$ are defined as the image of the coincident root stratum-closure $\clos{Y_\la}$ on $\Pol^k(\C^2)$ via the inclusion
\[i:\pol^k(\C^2)\to J^k(2,1).\]

The $\GL(2)$-equivariant cohomology classes of the strata $Y_\lambda$ are known (see \cite{FeherNemethiRimanyi}), and the $\GL(2)\times \GL(1)$-equivariant cohomology classes $[Y_\lambda\subset \pol^k(\C^2)]$ can be calculated by the substitution 
\begin{equation}\label{eq:root_subs}
[Y_\lambda\subset \pol^k(\C^2)]_{\GL_2\times\GL_1}=[Y_\lambda\subset \pol^k(\C^2)]_{\GL_2}\bigg(\alpha_i\mapsto u/|\lambda|-\alpha_i\bigg)	
\end{equation}
in the Chern roots $\al_i$ of the rank 2 bundle. By identifying the cohomology of $H_{\mathcal{K}}$ with $H_{\GL_2\times \GL_1}$, and since $i_!1=e(J^{k-1}(2,1))$, 
we have the following Proposition.
\begin{proposition}\label{prop:binary}
\begin{equation}
	[\eta_\lambda(2,1)]=e(J^{k-1}(2,1))[Y_\lambda]_{\GL_2\times \GL_1}.
\end{equation}	
\end{proposition}
\begin{corollary}\label{cor:1k}
	\[ [\eta_{(1^k)}(2,1)]=e(J^{k-1}(2,1))=\prod_{i=1}^{k-1}e(\pol^i(\C^2))=\prod_{1\leq i+j<k} (u-i\alpha_1-j\al_2),\]
	where $u$ is the generator of $H^*_{\GL(1)}$ and $\alpha_1,\al_2$ are the source Chern roots. 
\end{corollary}
\begin{proof}
	$Y_{(1^k)}$ is the open stratum in $\Pol^k(\C^2)$.
\end{proof}
Notice that $\eta_{(1^k)}(2,1)=\Sigma^{2^{k-1}}(2,1)$, so this formula is also a special case of Proposition \ref{prop:sigma-n-ad-k}.
\begin{corollary} \label{cor:21k}
		\[ [\eta_{(2,1^{k-2})}(2,1)]=e(J^{k-1}(2,1))(k-1)(2u-ka_1).\]
\end{corollary}
\begin{proof}
	Since $[Y_{(2,1^{k-2})}]_{\GL_2}=k(k-1)a_1$ (see \cite{FeherNemethiRimanyi}), therefore by \eqref{eq:root_subs}, we have
	\[  [Y_{(2,1^{k-2})}]_{\GL(2)\times \GL(1)}=k(k-1)a_1|_{\alpha_i\mapsto u/k-\alpha_i}=(k-1)(2u-ka_1).\]	
\end{proof}

We recover some Thom polynomials and also obtain some new ones using this Proposition (cf.\ also Table \ref{fig:binary}). Since the values of $[Y_\la]_{\GL_2}$ can be computed for arbitrarily large $\la$ via the methods of \cite{FeherNemethiRimanyi}, this method allows one to compute the values of infinitely many Thom polynomials $[\eta_\la(2,1)]$; we do this up to $|\la|\leq 4$ in Table \ref{tab:binary-crl-strata}.

\begin{table}[htbp]
	\centering
	\small
	\setlength{\tabcolsep}{5pt}
	\renewcommand{\arraystretch}{1.18}
	\begin{tabularx}{\textwidth}{
			>{\centering\arraybackslash}m{2.1cm}
			>{\centering\arraybackslash}m{3.0cm}
			>{\centering\arraybackslash}m{2.6cm}
			Y}
		\toprule
		\textbf{Partition $\lambda$} &
		\textbf{Binary CRL class $[Y_\la]_{\GL_2}$} &
		\textbf{Arnold name} &
		\textbf{Legendre TP $\tilde{K}_\eta$ (for $Q_\la$ with $\ell(\la)\leq 2$)} \\
		\midrule

		$(1,1)$
		& $1$
		& $A_1$
		& \LTPcell{1} \\
		
		$(2)$
		& $2c_1$
		& $A_2$
		& \LTPcell{Q_1}\\
		
		$(1,1,1)$
		& $1$
		& $D_4$
		& \LTPcell{Q_{2,1}} \\
		
		$(2,1)$
		& $6c_1$
		& $D_5$
		& \LTPcell{6Q_{3,1} + 4tQ_{2,1}} \\
		
		$(3)$
		& $6c_1^2+3c_2$
		& $E_6$
		& \LTPcell{6Q_{4,1} + 9Q_{3,2} + 9tQ_{3,1} + 3t^2Q_{2,1}}\\
		
		$(1,1,1,1)$
		& $1$
		& $X_9$
		& \LTPcell{18Q_{5,2} + 27Q_{4,3} + 6tQ_{5,1} + 42tQ_{4,2} + 11t^2Q_{4,1}\\{} + 21t^2Q_{3,2} + 6t^3Q_{3,1} + t^4Q_{2,1}}
		\\
		
		$(2,1,1)$
		& $12c_1$
		& $X_{10}=X_{1,1}=T_{2,4,5}$
		& \LTPcell{216Q_{6,2} + 540Q_{5,3} + 72tQ_{6,1} + 792tQ_{5,2} + 828tQ_{4,3} + 204t^2Q_{5,1} \\{}+ 888t^2Q_{4,2} + 204t^3Q_{4,1} + 324t^3Q_{3,2} + 84t^4Q_{3,1} + 12t^5Q_{2,1}\\{}}\\
		
		$(2,2)$
		& $12c_1^2+16c_2$
		& $Y^1_{1,1}=T_{2,5,5}$
		& \LTPcell{216Q_{7,2} + 1044Q_{6,3} + 972Q_{5,4} + 72tQ_{7,1} + 1320tQ_{6,2} + 3192tQ_{5,3} \\{}+ 324t^2Q_{6,1} + 2516t^2Q_{5,2} 
			+ 2496t^2Q_{4,3} + 520t^3Q_{5,1} + 1936t^3Q_{4,2} \\{}+ 380t^4Q_{4,1} + 556t^4Q_{3,2} + 128t^5Q_{3,1} + 16t^6Q_{2,1}\\{}} \\
		
		$(3,1)$
		& $24c_1^2$
		& $Z_{11}$
		& \LTPcell{432Q_{7,2} + 1512Q_{6,3} + 1080Q_{5,4} + 144tQ_{7,1} + 2160tQ_{6,2} \\{}+ 4320tQ_{5,3} + 552t^2Q_{6,1} + 3768t^2Q_{5,2} + 3432t^2Q_{4,3} + 816t^3Q_{5,1} \\{}+ 2832t^3Q_{4,2} + 576t^4Q_{4,1} + 816t^4Q_{3,2} + 192t^5Q_{3,1} + 24t^6Q_{2,1}\\{}} \\
		
		$(4)$
		& $24c_1^3+32c_1c_2$
		& $W_{12}$
		& \LTPcell{
			432Q_{8,2} + 2520Q_{7,3} + 4032Q_{6,4} + 144tQ_{8,1} + 3216tQ_{7,2} \\{}+ 11112tQ_{6,3} + 8328tQ_{5,4} + 792t^2Q_{7,1} + 8320t^2Q_{6,2} + 16408t^2Q_{5,3} \\{}+ 1688t^3Q_{6,1} + 9944t^3Q_{5,2} + 8864t^3Q_{4,3} + 1800t^4Q_{5,1} \\{}+ 5744t^4Q_{4,2} + 1016t^5Q_{4,1} + 1368t^5Q_{3,2} + 288t^6Q_{3,1} + 32t^7Q_{2,1}
		}\\
		
		\bottomrule
	\end{tabularx}
	\caption{Binary coincident-root loci indexed by partitions of size at most $4$,
		their $GL_2$-equivariant classes, the corresponding
		generic Arnold contact singularities, and  their Legendre TP's. \emph{The values of the TP's are correct up to the ideal $(Q_\la:\ell(\la)\geq 3)$}; undetermined coefficients first appear for $X_9$, cf.\ Remark \ref{rmk:indeterminacy}.}
	\label{tab:binary-crl-strata}
\end{table}

\begin{theorem}
	Fix a binary singularity $\eta_\la$ corresponding to some binary stratum $Y_\la$. Then 
\begin{itemize}
	\item The value of the Thom polynomial $[\eta_\la(2,1)]$ is entirely determined by the binary class $[Y_\la]_{\GL_2}$ (see Table \ref{tab:binary-crl-strata}) using Proposition \ref{prop:binary} and \eqref{eq:root_subs}.
	\item  The Thom polynomial $[\eta_\la(2,1)]$ determines the value of the Legendre Thom polynomial $\tilde{K}_{\eta_\la}$ up to the ideal of $Q$-polynomials indexed by strict partitions of length $\geq 3$. 
	\item The Thom polynomial $[\eta_\la(2,1)]$ determines the value of the stable Thom polynomial $\Tp(\eta_\la(-1))$ up to the ideal of Schur polynomials indexed by thick partitions (i.e.\ partitions containing $(3,3)$). 
\end{itemize}
See Table \ref{tab:binary-crl-strata} for the values of the Legendre Thom polynomials for $\la$ up to size $4$ (up to the ideal described above). The values of the stable Thom polynomials up to $E_6$ are covered by Theorem \ref{thm:thom_series} (up to the ideal described above). 
\end{theorem}
	We illustrate the straightforward proof on three examples, $\la=(1,1,1)$,  $\la=(2,1)$ and $\la=(2,1,1)$; the proof is the same for all other singularities by taking the appropriate value of $[Y_\la]$.
\begin{example} \label{ex111}
	For $\la=(1,1,1)$ we have $ \eta_{1,1,1}=D_4$ and 
\[ [\eta_{1,1,1}(2,1)]=e(\pol^1(\C^2))e(\pol^2(\C^2)).\]
Since
\[ e(\pol^2(\C^2))=(u-2\alpha_1)(u-\alpha_1-\al_2)(u-2\al_2)=u^3-3u^2a_1+u(4a_2+2a_1^2)-4a_1a_2,\]
using the factorization formulas we obtain 
\[[\eta_{1,1,1}(2,1)]=[D_4(2,1)]=\rho^{2\to1}(s_5+3s_{4,1}+2s_{3,1,1}+6s_{3,2}+4s_{2,2,1}),\]
but $\rho^{2\to1}$ is injective in this range so we obtained
the stable Thom polynomial. Also the Legendre Thom polynomial is determined because  $\mathcal{L}_2(Q_{2,1})=e(\pol^2(\C^2))$
and $\mathcal{L}_2$ is injective in this degree by Corollary \ref{l-factorization}, therefore $\tilde K_{\eta_{1,1,1}}=Q_{2,1}$.
\end{example}

\begin{example} \label{ex21}
	For $\la=(2,1)$, we have $\eta_{2,1}=D_5$  and 
		\[[D_5(2,1)]=e(\Pol^1(\C^2))e(\Pol^2(\C^2))(4u-6a_1),\]
 since we substitute   $\al_i\mapsto u/3-\al_i$ into $[Y_{2,1}]=6c_1$ (see \eqref{eq:root_subs}). By Kazarian's structure theorem, we obtain
\begin{equation}\label{eq:L2KD5}
	L_2(K_{D_5})=e(\Pol^2(\C^2))(4u-6a_1).
\end{equation}
We want to compare this with $\mathcal{L}_2(\tilde{K}_{D_5})$ using \eqref{eq:L-tildeL}. We can write $\tilde K_{D_5}=xQ_{3,1}+ytQ_{2,1}$, where the coefficient of $Q_4$ is zero by the vanishing of thin coefficients. Now by Corollary \ref{l-factorization}
\[ \mathcal L_2(tQ_{2,1})=-\frac u2 e(\Pol^2(\C^2)),\qquad  \mathcal L_2(Q_{3,1})=\mathcal L_2(Q_{2,1})\mathcal L_2(s_1)=e(\Pol^2(\C^2))\cdot (u-a_1),
\]
using that \[\mathcal L_2(s_1)=s_1(u/2-\al_1,u/2-\al_2)=u-a_1.\]
Therefore $\mathcal{L}_2(\tilde{K}_{D_5})=L_2(K_{D_5})$ is the equation
\[
e(\Pol^2(\C^2))(x(u-a_1)-yu/2)=e(\Pol^2(\C^2))(4u-6a_1)
\]
We obtain the equations $x=6$, $8=2x-y$, i.e.\ $\tilde{K}_{D_5}=6Q_{3,1}+4tQ_{2,1}$.
\end{example}	

\begin{example}  \label{ex211}
For $\la=(2,1,1)$, we have $\eta_{2,1,1}=X_{10}$ and $[Y_{2,1,1}]_{\GL_2}=12c_1$. Making the substitution $\al_i\mapsto u/4-\al_i$ in the Chern roots, we get \[[Y_{2,1,1}]_{\GL_2\times \GL_1}=12[(u/4-\al_1)+(u/4-\al_2)]=6u-12a_1.\]
Then by Proposition \ref{prop:binary}, we have
\[[\eta_{2,1,1}(2,1)]=e(\pol^1(\C^2))e(\pol^2(\C^2))e(\pol^3(\C^2))[Y_{2,1,1}]_{\GL_2\times \GL_1}\]
Using the Schur factorization Theorem \ref{thm:schur-factorization} we obtain the values of the stable Thom polynomial up to the thick partitions. From  Corollary \ref{l-factorization} we obtain the values of the Legendre Thom polynomial up to the thick partitions.  We omit the details of these purely algebraic computations, see Table \ref{tab:binary-crl-strata} for the result.
\end{example}

 Note that the examples after $E_6$ are beyond Mather's nice range \cite{Mather1971}.
\begin{remark} \label{rem:X9}
Kazarian has computed the entire Legendre Thom polynomial $K_{X_9}$, and it turns out that all coefficients of thick partitions are zero; i.e.\ the polynomial in Table \ref{tab:binary-crl-strata} is the full Legendre Thom polynomial of $X_9$, see \cite[\S 11]{MikoszPragaczWeber2011}. In fact, using different methods, Kazarian also computed the entire stable Thom polynomial $\Tp(X_9(-1))$ via methods of \cite{Kazarian} (private comm.). Here some of the actual values are larger than the ones predicted by the Thom series, see also Observation \ref{actual-expected}.
\end{remark}

\begin{remark}\label{rmk:indeterminacy}
To illustrate the indeterminacy of the stable and Legendre Thom polynomials, we enumerate the number of thick strict partitions (for $n=2$) in different degrees:
\[
6:Q_{3,2,1},\qquad 7:Q_{4,2,1}, \qquad 8:Q_{5,2,1}, Q_{4,3,1},\qquad 9: Q_{6,2,1}, Q_{5,3,1}, Q_{4,3,2}, 
\]
so together with $t$-powers, the number of undetermined coefficients of Legendre Thom polynomials in degrees $6,7,8,9,10$ is $1,2,4,7,12$. 

The number of thick partitions (for $n=2,p=1$), (and therefore the number of undetermined coefficients of the stable Thom polynomials) in degrees $6,7,8,9,10$ is $1,2,5,9,16$.
\end{remark}

This leaves a number of coefficients undetermined: we do not know the entire value of the stable Thom polynomial $\Tp(X_{10}(-1))$ or the Legendre Thom polynomial $K_{X_{10}}$.

\subsection{Ternary homogeneous singularities and beyond} For ternary degree $3$ polynomials we can use the classification of cubic curves and that their Thom polynomials were calculated in \cite[\S 8]{Komuves}, see also \cite[Appendix C]{LeePatelTseng2023}. Let $f$ be a $\GL(3)$-invariant subvariety of cubic curves. Then $f$ also corresponds to a function singularity $\eta_f$, and
\begin{equation}\label{eq:ternary}
	[\eta_f(3,1)]=[f]e(\pol^1(\C^3))e(\pol^2(\C^3)), 
\end{equation}
where $[f]$ is the $\GL(3)\times\GL(1)$-equivariant cohomology class of $f$. For $f=\pol^3(\C^3)$ we have $\eta_f(3,1)=\Sigma^{3,3}(3,1)=P_8$. Since $[f]=1$ in this case we have
\[ [\eta_f(3,1)]=e(\pol^1(\C^3))e(\pol^2(\C^3)). \]

Using the same method, we recover some Thom polynomials and obtain some new ones. We restrict our attention to the reduced cases (merely due to the size of the resulting Legendre Thom polynomials.)

\begin{table}[htbp] \centering \small \setlength{\tabcolsep}{3pt} \renewcommand{\arraystretch}{1.20} \begin{tabularx}{\textwidth}{ >{\centering\arraybackslash}m{1.45cm} >{\centering\arraybackslash}m{1.55cm} >{\centering\arraybackslash}m{2.35cm} >{\centering\arraybackslash}m{1.85cm} Y} \toprule 
		\textbf{Stratum} & \textbf{Curve} & \textbf{Ternary class} & \textbf{Arnold name} & \textbf{Legendre TP $\tilde{K}_\eta$ } \\
		\midrule 
		Smooth & \diagSmooth & $1$ & $P_8$ &
		\LTPcell{ Q_{3,2,1} } \\[2ex] 
		
		nodal & \diagNodal & $12c_1$ & $T_{3,3,4}$ & 
		\LTPcell{ 12Q_{4,2,1} +12tQ_{3,2,1} } \\[2ex] 
		
		$\nu$ & \diagNu & $24c_1^2$ & $Q_{10}$ & 
		\LTPcell{ 24Q_{5,2,1} +24Q_{4,3,1} +48tQ_{4,2,1} +24t^2Q_{3,2,1} } \\[2ex] 
		
		$\theta$ & \diagTheta & $18c_1^2+9c_2$ & $T_{3,4,4}$ & \LTPcell{ 18Q_{5,2,1} +27Q_{4,3,1} +42tQ_{4,2,1} +21t^2Q_{3,2,1} } \\[2ex] 
		
		$\Omega$ & \diagOmega & $36c_1^3+18c_1c_2$ & $S_{11}$ & \LTPcell{ 36Q_{6,2,1} +90Q_{5,3,1} +54Q_{4,3,2} +120tQ_{5,2,1} +138tQ_{4,3,1} \\{}+126t^2Q_{4,2,1} +42t^3Q_{3,2,1} } \\[2ex] 
		
		$A$ & \diagA & $12c_1^3+6c_1c_2+27c_3$ & $T_{4,4,4}$ & \LTPcell{ 12Q_{6,2,1} +30Q_{5,3,1} +45Q_{4,3,2} +40tQ_{5,2,1} +55tQ_{4,3,1} \\ {}+45t^2Q_{4,2,1} +15t^3Q_{3,2,1} } \\[2ex] 
		
		$\Zh$ & \diagZh & $12c_1^4+6c_1^2c_2+27c_1c_3$ & $U_{12}$ & \LTPcell{ 12Q_{7,2,1} +42Q_{6,3,1} +30Q_{5,4,1} +75Q_{5,3,2} +52tQ_{6,2,1} \\{}+125tQ_{5,3,1} +100tQ_{4,3,2} +85t^2Q_{5,2,1} +100t^2Q_{4,3,1} \\ {}+60t^3Q_{4,2,1} +15t^4Q_{3,2,1} +?\cdot Q_{4,3,2,1} } \\[2ex] 
		\bottomrule 
	\end{tabularx} 
	\caption{Ternary cubic strata, their $\GL_3$-equivariant classes, the corresponding Arnold types, and the associated Legendre Thom polynomials. \emph{All coefficients are determined except the coefficient of $Q_{4,3,2,1}$ in $U_{12}$}. The non-reduced examples (double-line, triple-line) can also be computed similarly, but more coefficients remain undetermined.
	} \label{tab:ternary-cubic-strata} 
\end{table}

\begin{theorem}
	Fix a ternary singularity $\eta_f$ corresponding to some ternary cubic stratum $f$. Then we have:
	\begin{itemize}
		\item The value of the Thom polynomial $[\eta_f(3,1)]$ is entirely determined by the ternary class $[f]_{\GL_3}$ (see Table \ref{tab:ternary-cubic-strata}) using \eqref{eq:root_subs} and \eqref{eq:ternary}.
	\item  The Thom polynomial $[\eta_f(3,1)]$ determines the value of the Legendre Thom polynomial $\tilde{K}_{\eta_f}$ up to the ideal of $Q$-polynomials indexed by strict partitions of length $\geq 4$. 
	\item The Thom polynomial $[\eta_f(3,1)]$ determines the value of the stable Thom polynomial $\Tp(\eta_f(-2))$ up to the ideal of Schur polynomials indexed by thick partitions (i.e.\ partitions containing $(4,4)$). 
		\end{itemize}
	For the values of the Legendre Thom polynomials	(in the reduced cases) see Table \ref{tab:ternary-cubic-strata}.
\end{theorem}

\begin{remark} The examples above are special cases of the following construction: If $Y\subset \pol^d(\C^n)$ is a $\GL(n)$-invariant subvariety, then $i(Y)\subset J^d(n,1)$ is contact invariant. An interesting class of such $Y$'s are the $\tau$-discriminants $D_\tau(d)$, where $\tau=(\eta_1,\dots,\eta_k)$ is a multisingularity of function singularities of $n-1$ variables. 
	
	If the corresponding Thom polynomials and residual polynomials are known then the equivariant cohomology class of
	\[D_\tau(d):=\{F\in \pol^d(\C^n): (F=0) \subset\P^{n-1}  \text{ has $\eta_i$ singularities}\}\]
	can be calculated, see upcoming work \cite{FeherMatszangoszupcoming}. 
	
	For the binary homogeneous case, for example $Y_{3,2,2,1}$ is the $\tau$-discriminant $D_\tau(8)$ for $\tau=A_2A_1^2$. 
	
	For the ternary case one can notice that all orbits in Table \ref{tab:ternary-cubic-strata} are $\tau$-discriminants, for example for $\theta$ we have $\tau=A_1^2$ and for $\Omega$  we have $\tau=A_3$. 
	
	For monosingularities we can even iterate this process: compute the class of an $\eta$-discriminant, deduce an $\eta'$ Thom polynomial, then compute a new $\eta'$-discriminant, \ldots. However, the codimension grows very quickly.
\end{remark}

\subsection{Multi-binary singularities}\label{sec:multibinary}
\newcommand{\II}{\mathcal{I}}
Certain function singularity types are designated not by a single given degree behaviour, but the interaction between several:
\begin{example}
	Recall that $f=f_1+f_2+\ldots$ is in the closure of $X_{10}$ iff $f_i=0$ for $i\leq 3$ and $f_4\in (2,1,1)$, i.e.\ there exists $l\in \Pol^1(\C^2)$ and $q\in \Pol^2(\C^2)$, such that $f_4=l^2q$. On the other hand, $f$ is in the closure of $X_{11}$ iff  $f_4=l^2q$ and $f_{5}=lh$ for the same $l$ and some $h\in \Pol^{4}(\C^2)$.
\end{example}

In this section we compute the Thom polynomial of $X_{11}$ and other, similarly defined singularities. We split the section in two parts: first we do the computation of $X_{11}$ in detail, and in the second part we state a general theorem and give examples.

\localheading{The Thom polynomial of $X_{11}$} 

To compute the $\GL_2\times\GL_1$-equivariant class of such ``multi-binary" singularities we can generalize the idea of projective Thom polynomials of \cite{FeherNemethiRimanyi}: 
Let $G$ be a linear algebraic group acting on the vector spaces $A,B$. Assume that $Y\subset V:=A\oplus B$ is a bihomogeneous subvariety (i.e.\ $\C^\times\times\C^\times$-invariant), which therefore has a bi-projectivization $\PP_{A,B}Y\subset \PP A\times \PP B$. 
\begin{proposition}
	Let $\Phi:\tilde{Y}\to \PP A\times \PP B$ be a $G$-equivariant resolution of $\PP_{A,B}Y$. Then
	\[
	[Y\subset A\oplus B]_G=\int_{\tilde{Y}}\Phi^*(e_G(Q_A)\cdot e_G(Q_B))
	\]
	where $Q_A$ and $Q_B$ are the tautological quotient bundles over $\PP A$ and $\PP B$.
\end{proposition}
For the proof, see Proposition \ref{prop:biproj}.
 Based on this proposition, we will now compute the Thom polynomial of $[X_{11}(2,1)]$ in two steps:  first we compute the class of the \emph{shared-root locus} 
\[[Y_{11}\subset \Pol^4(\C^2)\oplus \Pol^5(\C^2)]_{\GL_2\times \GL_1},\] and then use the  analogue of the Proposition \ref{prop:binary} used earlier:
\begin{equation}\label{eq:X11}
	[X_{11}(2,1)]=e(J^{3}(2,1))[Y_{11}]_{\GL_2\times \GL_1}.
\end{equation}

Write $\PP_k:=\PP\Pol^k(\C^2)$ and let\footnote{In this setting, contrary to \eqref{eq:root_subs}, the $\GL_2$-equivariant class does not directly determine the $\GL_2\times \GL_1$-invariant one, so we work with $\GL_2\times\GL_1$.} $G=\GL_2\times \GL_1$ . Set $A=\Pol^4(\C^2)$, $B=\Pol^5(\C^2)$ and let
\[\PP_{A,B}Y_{11}=\{(p,r)\in\PP_4\times \PP_5:p=l^2q, r=lh \text{ for  some } l\in \PP_1,q\in \PP_2, h\in \PP_4\}\]
By its definition, this has a coincident-root-locus style $G$-equivariant resolution $\tilde{Y}_{11}=\PP_1\times \PP_2\times \PP_4$ given by $\Phi:\tilde{Y}_{11}\to \PP_4\times \PP_5$:
\[\Phi:(l,q,h)\mapsto (l^2q,lh).\]
Using this resolution, we can compute the shared-root locus $[Y_{11}]_{G}$ as follows.
\begin{lemma}\label{lemma:Y11}
\[
[Y_{11}]_{G}=\int_{\PP_1\times \PP_2\times \PP_4}\Phi^*(e_G(Q_{\PP_4})\cdot e_G(Q_{\PP_5}))=\int_{\PP_1}e_G\left( \frac{\Pol^4(\C^2)}{S_{\PP_1}^{\otimes2}\otimes \Pol^2(\C^2)}\right)e_G\left( \frac{\Pol^5(\C^2)}{S_{\PP_1}^{}\otimes \Pol^4(\C^2)}\right)
\]
\end{lemma}
\begin{proof}
The first equality is Proposition \ref{prop:biproj}, so we focus on the second equality. We first recall a simple identity: let $A\subset V$ be a sub-vector bundle over $X$, then for $p:\PP A\to X$ since $p_!e(Q_{\PP_A})=1$,
\begin{equation}\label{eq:integration}
p_!(e(V/S_{\PP A}))=e(V/A).
\end{equation}

Since $Q_{\PP_k}=\Pol^k(\C^2)/S_{\PP_k}$, and $\Phi^*S_{\PP_4}=S_{\PP_1}^{\otimes 2}\otimes S_{\PP_2}$
\[
\Phi^*Q_{\PP_4}=\Pol^4(\C^2)/S_{\PP_1}^{\otimes 2}\otimes S_{\PP_2}
\]	
which by \eqref{eq:integration} implies that
\[
\int_{\PP_2} e_G(\Pol^4(\C^2)/S_{\PP_1}^{\otimes 2}\otimes S_{\PP_2})=e_G\left(\frac{\Pol^4(\C^2)}{S_{\PP_1}^{\otimes 2}\otimes \Pol^2(\C^2)}\right)
\]
A similar argument applies to $\Phi^*S_{\PP_5}=S_{\PP_1}\otimes S_{\PP_4}$.
\end{proof}
The rest of the proof is a computation:
\begin{proposition}
	\[[Y_{11}]_{G}=60a_1^2-40a_2-47a_1u+11u^2\]
\end{proposition}
\begin{proof}
By the splitting principle,
	\[\frac{\Pol^4(\C^2)}{S_{\PP_1}^{\otimes2}\otimes \Pol^2(\C^2)}\iso S_1\otimes Q_1^{\otimes3}\oplus Q_1^{\otimes4},\qquad  \frac{\Pol^5(\C^2)}{S_{\PP_1}^{}\otimes \Pol^4(\C^2)}\iso Q_1^{\otimes5}\]

	so that if $H_G^*(\PP_1)=\Z[x,a_1,a_2,u]/(x^2-a_1x+a_2)$ with $x=c_1(S_1^\vee)$ and $q=c_1(Q_1^\vee)=a_1-x$, 
\[
\int_{\PP_1}(u-x-3q)(u-4q)(u-5q)
\]
By equivariant localization, we compute by setting $F(x,q)=(u-x-3q)(u-4q)(u-5q)$,
\[
\int_{\PP_1}F(x,q)=\frac{F(\al,\be)-F(\be,\al)}{\al-\be},
\]
where $\al,\be$ are Chern roots: $a_1=\al+\be$ and $a_2=\al\be$. Carrying out the computation we obtain the result; e.g.\ the coefficient of $u^2$ is $\frac{(12\be+\al)-(12\al+\be)}{\be-\al}=11$.
\end{proof}
Thus we obtain the Thom polynomial of the multi-binary singularity $[X_{11}(2,1)]$ by \eqref{eq:X11}, see Table \ref{tab:multi-binary-crl-strata} below.
\localheading{The general case} 
The numbers in $\Phi:\PP_1\times \PP_2\times \PP_4\to \PP_4\times \PP_5$ were essentially arbitrary; the example of $X_{11}$ generalizes as follows. Given a tuple of numbers $I=((a_i), (b_i): a_i\geq b_i\geq 1, i=1\stb r)$, with $r\geq 2$ consider resolutions of the form:
\[
\Phi_I:\PP_1\times \prod_{i=1}^r\PP_{a_i-b_i}\to \prod_{i=1}^r\PP_{a_i}
\]
\[
(l,q_1\stb q_r)\mapsto (l^{b_1}q_1\stb l^{b_r}q_r).
\]
The case of $Y_{11}$ is the tuple $((4,5), (2,1))$. To pass from the projective Thom polynomial to the affine version, we have the obvious analogue of Proposition \ref{prop:biproj}. We also have the analogue of Lemma \ref{lemma:Y11}:
\begin{lemma}\label{lemma:Y}
	\[
	[Y_I]_{G}=\int_{\PP_1}\prod_{i=1}^re_G\left( \frac{\Pol^{a_i}(\C^2)}{S_{\PP_1}^{\otimes b_i}\otimes \Pol^{a_i-b_i}(\C^2)}\right)=\int_{\PP_1}\prod_{i=1}^r\prod_{j=0}^{b_i-1}(u-jx-(a_i-j)q)
	\]
\end{lemma}

To compute Thom polynomials of contact singularities, we are interested in shared-root loci indexed by tuples of the form 
\begin{equation}\label{eq:tuple_index}
	\II_k(b_1>\ldots>b_r):=((k+i-1), (b_i):k+i-1\geq b_i, i=1\stb r).\	
\end{equation}
For such $\II_k(b_1>\ldots>b_r)$, we have the analogue of Proposition \ref{prop:binary}. The key reason is that the preimage of such shared-root strata in  the jet-space is $\mathcal{K}$-invariant. (This was immediate in the case of binary singularities, since contact equivalence only adds higher degree contributions.)
\begin{lemma}
Let $I:=\mathcal{I}_k(b_1>\ldots>b_r)$ be the index of a shared-root locus $Y_I\subset \bigoplus_{i=1}^r \Pol^{k+i-1}(\C^2)$ and let $i:\bigoplus_{i=1}^r \Pol^{k+i-1}(\C^2)\to J^{k+r-1}(\C^2)$ denote the inclusion. Then $\eta_I:=i(\clos{Y_I})$ is $\mathcal{K}$-invariant and
\begin{equation}\label{eq:multibinary_to_singularity}
	[\eta_I]_{\mathcal{K}}= e(J^{k-1}(\C^2))[Y_I]_{\GL_2\times\GL_1}
\end{equation}
\end{lemma}
\begin{proof}
	The nontrivial part of the statement is that $i(\clos{Y_I})$ is $\mathcal{K}$-invariant. The locus $\clos{Y_I}$ for $I=\II_k(b)$ is by definition the locus, where there exists a line $l$, such that the multiplicity of $l$ in the degree $k+i-1$ part is at least $b_i$, for $i=1\stb r$. We want to show that this property is invariant under the contact group. Write a generic contact transformation as
	\[
	f\to
	K.f=\lambda(f\circ\varphi),
	\qquad K=(\varphi,\lambda),\qquad 
	\varphi=\varphi_1+\varphi_2+\varphi_3+\cdots,
	\qquad
	\lambda=\la_0+\lambda_1+\lambda_2+\cdots.
	\]
	with $\la_0\in \C^\times, \varphi_1\in \GL_2$. 
	The contributions of $f_{k+i-1}$ to $(K.f)_{k+j-1}$ are sums over terms of the Taylor expansion of the form
	\[
	\la_b\cdot (D^tf_{k+i-1}\circ \varphi_1)[\varphi_{k_1}\stb \varphi_{k_t}]
	\]
	with $k_i\geq 2$ and $b+i+\sum (k_i-1)=j$. Since $k_i\geq  2$, this implies $t\leq j-i$.  
	
	Write $f_{k+i-1}=l^{b_i}q_i$. Then $\varphi_1^*l$ has some nontrivial multiplicity in $(K.f)_{k+j-1}$. For each such term $D^t$ lowers the multiplicity of $\varphi_1^{*}l$ in $(K.f)_{k+j-1}$ by at most $t$. Therefore the multiplicity of $\varphi_1^{*}l$ in $(K.f)_{k+j-1}$ is at least
\[ b_i-t\geq b_i-(j-i)\geq b_j\]
for each contribution as required.
\end{proof}
Applying Lemma \ref{lemma:Y} and \eqref{eq:multibinary_to_singularity} we can compute the Thom polynomials $[\eta_I(2,1)]$ of multi-binary loci  for different values of $I$. We state this in the following theorem.

\begin{table}[htbp]
	\centering
	\small
	\setlength{\tabcolsep}{3pt}
	\renewcommand{\arraystretch}{1.18}
	\begin{tabularx}{\textwidth}{
			>{\centering\arraybackslash}m{2.2cm}
			>{\centering\arraybackslash}m{4.2cm}
			>{\centering\arraybackslash}m{1.6cm}
			Y}
		\toprule
		\textbf{Index $\II_k(b_i)$} &
		\textbf{Shared-root class $[Y_I]_{\GL_2\times \GL_1}$} &
		\textbf{Arnold name} &
		\textbf{Legendre TP $\tilde{K}_\eta$ (for $Q_\la$ with $\ell(\la)\leq 2$)} \\
		\midrule
		
$\II_2(2,1)$
& \LTPcell{6a_1^2 - 11a_1u + 5u^2}
& $A_3$
& \LTPcell{3Q_{2} + tQ_{1}} \\\addlinespace[5pt]

		$\II_3(2,1)$
		& \LTPcell{24a_1^2 - 26a_1u - 12a_2 + 8u^2}
		& $D_6$
		& \LTPcell{24Q_{4,1} + 12Q_{3,2} + 32tQ_{3,1} + 12t^2Q_{2,1}} \\\addlinespace[5pt]
		
		$\II_3(3,1)$
		&\LTPcell{-24 a_{1}^{3} + 50 a_{1}^{2} u - 12 a_{1} a_{2} \\
			- 33 a_{1} u^{2} + 7 a_{2} u + 7 u^{3}}
		& $E_7$
		& \LTPcell{24Q_{5,1} + 60Q_{4,2} + 56tQ_{4,1} + 66tQ_{3,2} + 42t^2Q_{3,1} + 10t^3Q_{2,1}}\\\addlinespace[5pt]
		
		$\II_3(3,2)$
		&\LTPcell{72 a_{1}^{4} - 174 a_{1}^{3} u - 12 a_{1}^{2} a_{2} \\
			+ 153 a_{1}^{2} u^{2} + 39 a_{1} a_{2} u - 60 a_{1} u^{3} \\
			- 24 a_{2}^{2} - 15 a_{2} u^{2} + 9 u^{4}}
		& $E_8$
		& \LTPcell{72Q_{6,1} + 204Q_{5,2} + 108Q_{4,3} + 216tQ_{5,1} + 414tQ_{4,2} \\
			+ 246t^2Q_{4,1} + 246t^2Q_{3,2} + 126t^3Q_{3,1} + 24t^4Q_{2,1}}\\\addlinespace[5pt]
		
		$\II_3(3,2,1)$
		&\LTPcell{
			-360 a_{1}^{5}
			+ 942 a_{1}^{4} u 
			+ 420 a_{1}^{3} a_{2}\\
			{}- 949 a_{1}^{3} u^{2} - 1007 a_{1}^{2} a_{2} u\\
			{}+ 478 a_{1}^{2} u^{3} + 300 a_{1} a_{2}^{2}\\
			{}+ 679 a_{1} a_{2} u^{2} - 125 a_{1} u^{4}\\
			{}- 154 a_{2}^{2} u - 140 a_{2} u^{3}
			+ 14 u^{5}
		}
		& $J_{10}$
		& \LTPcell{360Q_{7,1} + 1020Q_{6,2} + 660Q_{5,3} \\
			+ 1296tQ_{6,1} + 2782tQ_{5,2} + 1194tQ_{4,3} \\
			+ 1898t^2Q_{5,1} + 2912t^2Q_{4,2} + 1438t^3Q_{4,1} \\
			+ 1238t^3Q_{3,2} + 568t^4Q_{3,1} + 92t^5Q_{2,1}}\\\addlinespace[3pt]
		
		$\II_4(2,1)$
		& \LTPcell{60a_1^2-40a_2-47a_1u+11u^2}
		& $X_{11}$
		& \LTPcell{1080Q_{7,2} + 3060Q_{6,3} + 1620Q_{5,4} + 360tQ_{7,1} + 4908tQ_{6,2} \\
			+ 8490tQ_{5,3} + 1296t^2Q_{6,1} + 8236t^2Q_{5,2} + 7044t^2Q_{4,3} \\
			+ 1862t^3Q_{5,1} + 6224t^3Q_{4,2} + 1312t^4Q_{4,1} + 1832t^4Q_{3,2} \\
			+ 442t^5Q_{3,1} + 56t^6Q_{2,1}} \\\addlinespace[3pt]
		
		$\II_4(3,1)$
		& \LTPcell{-120 a_{1}^{3} + 154 a_{1}^{2} u + 80 a_{1} a_{2} \\- 69 a_{1} u^{2} - 40 a_{2} u + 11 u^{3}}
		& $Z_{12}$
		& \LTPcell{2160Q_{8,2} + 8280Q_{7,3} + 9360Q_{6,4} + 720tQ_{8,1} + 12696tQ_{7,2} \\
			+ 32916tQ_{6,3} + 20220tQ_{5,4} + 3312t^2Q_{7,1} + 28880t^2Q_{6,2} \\
			+ 47540t^2Q_{5,3} + 6316t^3Q_{6,1} + 32644t^3Q_{5,2} + 26536t^3Q_{4,3} \\
			+ 6348t^4Q_{5,1} + 18736t^4Q_{4,2} + 3508t^5Q_{4,1} + 4548t^5Q_{3,2} \\
			+ 996t^6Q_{3,1} + 112t^7Q_{2,1}} \\\addlinespace[3pt]
		
		$\II_4(4,1)$
		&\LTPcell{120 a_{1}^{4} - 274 a_{1}^{3} u + 160 a_{1}^{2} a_{2} \\+ 219 a_{1}^{2} u^{2} - 152 a_{1} a_{2} u \\- 74 a_{1} u^{3} + 36 a_{2} u^{2} + 9 u^{4}}
		& $W_{13}$
		& \LTPcell{2160Q_{9,2} + 14760Q_{8,3} + 32760Q_{7,4} + 20160Q_{6,5} + 720tQ_{9,1} \\
			+ 19176tQ_{8,2} + 85500tQ_{7,3} + 119376tQ_{6,4} + 4752t^2Q_{8,1} \\
			+ 63248t^2Q_{7,2} + 184756t^2Q_{6,3} + 127844t^2Q_{5,4} + 12796t^3Q_{7,1} \\
			+ 103920t^3Q_{6,2} + 184284t^3Q_{5,3} + 18284t^4Q_{6,1} + 92412t^4Q_{5,2} \\
			+ 77472t^4Q_{4,3} + 14980t^5Q_{5,1} + 43512t^5Q_{4,2} + 7028t^6Q_{4,1} \\
			+ 8964t^6Q_{3,2} + 1744t^7Q_{3,1} + 176t^8Q_{2,1}}\\\addlinespace[3pt]
		
		\bottomrule
	\end{tabularx}
	\caption{Shared-root loci indexed by the tuples defined in \eqref{eq:tuple_index},
		their $\GL_2\times\GL_1$-equivariant classes, the corresponding
		generic Arnold contact singularities, and their Legendre TP's. \emph{The values of the TP's are valid up to the ideal  $(Q_\la:\ell(\la)\geq 3)$}; undetermined coefficients first appear for $E_7$, cf.\ Remark \ref{rmk:indeterminacy}.}
	\label{tab:multi-binary-crl-strata}
\end{table}

\begin{theorem}
			Fix a multi-binary singularity $\eta_I$ corresponding to some shared-root stratum $Y_I$, $I=\II_k(b_i)$. Then 
		\begin{itemize}
			\item The value of the Thom polynomial $[\eta_I(2,1)\subset J^{k+r-1}]$ is entirely determined by the class of the shared-root-locus $[Y_I]_{\GL_2\times \GL_1}$ (see Table \ref{tab:multi-binary-crl-strata}) using \eqref{eq:multibinary_to_singularity}.
			\item  The Thom polynomial $[\eta_I(2,1)]$ determines the value of the Legendre Thom polynomial $\tilde{K}_{\eta_I}$ up to the ideal of $Q$-polynomials indexed by strict partitions of length $\geq 3$. 
			\item The Thom polynomial $[\eta_I(2,1)]$ determines the value of the stable Thom polynomial $\Tp(\eta_I(-1))$ up to the ideal of Schur polynomials indexed by thick partitions (i.e.\ partitions containing $(3,3)$). 
		\end{itemize}
		See Table \ref{tab:multi-binary-crl-strata} for the values of the Legendre Thom polynomials for small values of $I=\II_k(b_i)$ (up to the ideal described above). 
\end{theorem}
The example of $[J_{10}(2,1)]$ is noteworthy from the point of view of Arnold's mathematical trinities \cite[p.\ 10]{ArnoldTrinities}, \cite{Arnold1976}:
the underlying singularity $J_{10}$ is $\tilde{E}_8$ from the trinity of parabolic unimodal \cite{Arnold1976} (simple-elliptic normal surface \cite{Saito1974}) singularities
\[
P_8=\tilde{E}_6,\qquad X_9=\tilde{E}_7,\qquad J_{10}=\tilde{E}_8.
\]
Thus we also obtain the Legendre Thom polynomial of $J_{10}$ up to the missing coefficients $Q_{5,2,1}$, $Q_{4,3,1}$, $tQ_{4,2,1}$, $t^2Q_{3,2,1}$. The knowledge of the entire $K_{J_{10}}$ therefore amounts determining these missing coefficients, which would be entirely determined by the Thom polynomial $[J_{10}(3,1)]$.

Deopurkar’s computation of equivariant classes of $\GL(2)$ orbit closures \cite{Deo26} suggests that the methods of this section can be extended to contact invariant loci arising from orbit closures in $\bigoplus_{i=j}^d \Pol^i(\C^2)$. The technique for multi-binary singularities described in this section can also be refined to compute $X_{k}$, $k>11$, but since those singularities are not strictly speaking ``multi-binary", we will study their computation elsewhere.

\section{Thom polynomials of Thom-Boardman classes of order two:  \texorpdfstring{$\check{\Sigma}^{1,j}$}{function singularity}}	\label{sec:TB}
Thom polynomials of Thom-Boardman classes of order 2 were calculated in \cite{FeherKomuves}. These cases provide an infinite family of Thom series of function singularities by Observation \ref{obs:TBfunction}. In this section we give some Legendre Thom polynomials as well as negative Thom series for Thom-Boardman singularities (cf.\ Section \ref{sec:TB_discussion}). Notice that the examples in this section give Thom series in negative codimension for infinitely many singularity types.

\begin{theorem}
	The Legendre Thom polynomial of $\check{\Sigma}^{1,j}$ is 
	\[\tilde K_{\check{\Sigma}^{1,j}}=Q_{\delta_j}.\]
\end{theorem}
\begin{proof}
Indeed by Proposition  \ref{prop:sigma-n-ad-k}
\[[\Sigma^{j,j}(j,1)]=[\check{\Sigma}^{1,j}(j,1) ]=e(\pol^1(\C^j))e(\pol^2(\C^j)).\]

Using the Kazarian structure theorem and Corollary \ref{l-factorization} we obtain the result.
\end{proof}

This was noticed in \cite{MikoszPragaczWeber2009} without identifying the singularity class with  $\check{\Sigma}^{1,j}(j,1)$.
\bigskip

Giving a closed formula for the Thom series of  $\check{\Sigma}^{1,j}$ is a challenge, but it has been calculated in \cite{FeherKomuves}:  Set
\[E_{\lambda/\mu}(n) := C_{\lambda/\mu}\!\cdot\!\prod_{(i,j)\in\lambda/\mu}(n-i+j),\]
where $C_{\lambda/\mu}$ is the determinant
\[ C_{\lambda/\mu}=\det\left[ \frac{1}{\,(\lambda_i-\mu_j-i+j)!\,} \right]_{i,j\in l(\lambda)\times l(\lambda)}=
s_{\lambda/\mu}\left(1,\frac{1}{2},\frac{1}{3!},\frac{1}{4!},\dots\right) \]
(recall that $s_{\lambda/\mu}(c)=\det[c_{\lambda_i-\mu_j-i+j}]$ are the \emph{skew Schur polynomials}).

\begin{theorem} \label{sigmaij}
	The stable Thom polynomial of the second order Thom-Boardman singularity $\check\Sigma^{1,j}(-r)=\Sigma^{r+1,j}(-r)$ is (for $r\geq j-1$)
	\[\Tp(\check\Sigma^{1,j}(-r)) =   \Tp(\Sigma^{r+1,j}(-r)) =  \sum_{\mu\subset\delta} 2^{|\mu|-j(j-1)/2} \cdot E_{\delta/\mu}(r+1) \cdot s_{d-|\mu|,\tilde\mu}, \]
	where $\delta$ is the `staircase' partition $\delta=(j,j-1,\dots,2,1)$, $\widetilde\mu$ denotes the conjugate partition and
	\[ d=\codim\,\Sigma^{r+1,j}(-r)=r+1+|\delta|=r+1+{\textstyle \binom{j+1}{2}}.\]
\end{theorem}

For example the stable Thom polynomials of the singularities $\Sigma^{r+1,j}(-r)$ for $j\le 2$ (cf.\ Table \ref{fig:thomboardman}) are
\begin{align*}
\Tp(\check{\Sigma}^{1,0}(-r))
&= \Tp(A_1(-r)) &{}={}& s_{r+1} \\[1pt]
\Tp(\check{\Sigma}^{1,1}(-r))
&= \Tp(A_2(-r)) &{}={}& 2s_{r+1,1}+(r+1)s_{r+2} \\[1pt]
\Tp(\check{\Sigma}^{1,2}(-r))
&= \Tp(D_4(-r)) &{}={}& 4s_{r+1,2,1}
+ 2r s_{r+2,1,1}
+ 2(r+2)s_{r+2,2}
+ r(r+2)s_{r+3,1}
+ \textstyle\binom{r+2}{3}s_{r+4}.
\end{align*}
Note that $\Tp(\check{\Si}^{1,3}(-r))=\Tp(P_8(-r))$, see \cite{Kazarian2003} for its Legendre Thom polynomial. Notice this is already not in Mather's nice range \cite{Mather1971} due to being a moduli of orbits. \bigskip

Indeed, comparing with \cite{MikoszPragaczWeber2009}, \cite{MikoszPragaczWeber2011}, we see that 
\begin{itemize}
	\item $\tilde K_{\check{\Sigma}^{1,1}}=\tilde K_{A_2}=Q_1$,
	\item $\tilde K_{\check{\Sigma}^{1,2}}=\tilde K_{D_4}=Q_{2,1}$,
	\item $\tilde K_{\check{\Sigma}^{1,3}}=\tilde K_{P_8}=Q_{3,2,1}$.
\end{itemize}
And by identifying $\check{\Sigma}^{1,4}$ as $O_{16}$ \cite{Arnold1973} we obtain the Legendre Thom polynomial: $\tilde{K}_{\check{\Sigma}^{1,4}}=\tilde K_{O_{16}}=Q_{4,3,2,1}$. Arnold in \cite[\S 4.5]{Arnold1973} notes that these singularities are the ones representable as a sum of $j$ cubes.

\begin{remark}
	
	The first uncertainty in the stable Thom polynomial of $\check{\Si}^{1,j}$ appears at $j=3$: the Legendre Thom polynomial of $P_8=Q_{3,2,1}$ implies that
	\begin{align*}
	\Tp(\check\Si^{1,3}(-2))=\Tp(P_8(-2))=\,&8 s_{3,3,2,1} + 4 s_{4,2,2,1} + 12 s_{4,3,1,1} + 20 s_{4,3,2} + 6 s_{5,2,1,1} + 10 s_{5,2,2} + 30 s_{5,3,1} \\
	&{}+ 2 s_{6,1,1,1} + 15 s_{6,2,1} + 20 s_{6,3} + 5 s_{7,1,1} + 10 s_{7,2} + 4 s_{8,1} + s_{9}
	\end{align*}
	up to thick partitions $s_{4,4,1}$ and $s_{5,4}$. But a key result of \cite{FeherKomuves} is  that the  coefficients for thick partitions are all  zero for all $\check{\Si}^{1,j}$.
\end{remark}

\section{Beyond function singularities} \label{sec:beyond}
We expect similar results for non function singularities, but we don't have a good generalization of the quadratic unfolding yet. We can try to guess which Thom-Boardman classes fit into a sequence. One possible candidate is $\check\Sigma^{k,J}(-r)$, $r\geq k-1$ for any given $k>1$ and a partition $J$. Recall that  $\check\Sigma^{k,J}(-r)=\Sigma^{k+r,J}(-r)$. For  $J=(1)$ closed formulas were given in \cite[Theorem 4.10]{FeherKomuves}:

\begin{theorem}\label{thm:sigmai1}
	The stable Thom polynomial of $\Sigma^{i,1}(l)$, $l\geq 1-i$ is
	\begin{equation}\label{eq:Sigmai1}
	\Tp(\Sigma^{i,1}(l)) =
	\sum_{(\nu,\mu)\in I} F_{\nu/\mu}(i)\cdot
	s_{(i^{(l+i)}+\mathcal{C}\widetilde\nu,\widetilde\mu)}
	\end{equation} 
	where $I=\{(\nu,\mu )\::\:\nu\subset (l+i)^i,\; l(\mu)\le i,\; |\nu|-|\mu|=i-1\}$, $\widetilde\nu$ denotes the conjugate partition,  $\mathcal{C}$ is the complement in the $i^{l+i}$ box, and
	\[ \gbinom{n}{k} := \sum_{j=0}^k \binom{n}{j}
	\quad\quad\textrm{and}\quad\quad
	F_{\nu/\mu}(n):=
	\det \left[ \gbinom{\nu_k+n-k}{\mu_l+n-l} \right ]_{k,l\in n\times n}.
	\]
\end{theorem}

The support  of Schur polynomials is no longer fixed: as $r$ increases new non-zero coefficients appear. Indeed, the case of
$\check{\Sigma}^{2,1}(-r)$ already demonstrates this:
By Corollary \ref{cor:ThomBoardmankJ} (i), the codimension of $\check{\Si}^{2,1}(-r)$ increases by 3 as $r$ increases. 
However, in \eqref{eq:Sigmai1} $i^{l+i}=i^2$, so as $r$ increases, for fixed $\mu$ this contributes only 2. The remaining increase is due to the fact that in the indexing set, $|\nu|-|\mu|=i-1$. This means that for fixed $\mu$, the different $\nu$'s have to be identified how they stabilize. In particular, we see a new phenomenon: the support of Schur indices is not finite.\smallskip

In this case the indexing remains tractable: since $\nu\subset 2^i$, it can be written as $\nu=(2^a,1^b)$ with $a+b\leq i$. In this case, we can show a polynomial stabilization of the coefficients of the Thom polynomials \eqref{eq:Si21} -- the determinant defining $F_{\nu/\mu}(i)$  reduces to a fixed-size determinant whose entries are binomial-sum polynomials in $i$. 
\begin{example}[The stable Thom polynomial of $\check{\Si}^{2,1}(2-i)$]
	We apply Theorem \ref{thm:sigmai1} and introduce a stabilization pattern: Fix $i\geq 1$ and consider the set of pairs of $(a,\xi)$, where $\xi$ is a partition. Then---since $i+l=2$---all partitions $(i^{(l+i)}+\mathcal{C}\widetilde\nu,\widetilde\mu)$ occurring in \eqref{eq:Sigmai1}
are of the form
	\[\lambda(i,a,\xi):=(2i-a,i+a-|\xi|+1,\xi)\]
	where $(a,\xi)$ are elements of the following indexing set:
	\[I_i=\{(a,\xi):0\leq a\leq i,\ |\xi|\leq a+1,\ \xi_1\leq i,\ 2a\leq|\xi|+i-1\}.\]
	Indeed:
	\begin{itemize}
		\item writing $\xi=\tilde{\mu}$, we have $\ell(\mu)\leq i$ iff $\xi_1\leq i$;
		\item writing $\nu=(2^a,1^b)$; the condition $|\nu|-|\xi|=i-1$ becomes
		\[|\xi|=|\nu|-i+1=2a+b-i+1,\] so $b\geq 0$ gives the next condition: $2a\leq |\xi|+i-1$;
		\item finally $\nu\subset 2^i$ iff $a+b\leq i$ iff $2a+b-i+1=|\xi|\leq a+i-i+1$.
	\end{itemize}
 These are the conditions defining $I_i$. Next, note that  given $\nu=(2^a,1^b)$ and $\xi=\tilde{\mu}$, the partition appearing in \eqref{eq:Sigmai1} is
	\[(i^2+\mathcal{C}\tilde{\nu},\tilde{\mu})=(2i-a,2i-a-b,\xi)=\la(i,a,\xi).\]
	Then for every fixed pair $(a,\xi)$ satisfying $|\xi|\leq a+1$, setting $\be(a,\xi):=2a+1-|\xi|$ the coefficient of $s_{\la(i,a,\xi)}$ is $p_{a,\xi}(i):=F_{(2^a,1^{i-\be(a,\xi)})/\tilde{\xi}}(i)$, which can be shown to be a polynomial of degree $\leq \be(a,\xi)$:		
		\begin{equation}\label{eq:Si21}
		\Tp(\check{\Sigma}^{2,1}(2-i))=\sum_{(a,\xi)\in I_i}
		p_{a,\xi}(i)
		s_{\lambda(i,a,\xi)}.
				\end{equation}
	So in conclusion, for this appropriate stabilization of indices (i.e.\ $\la(i,a,\xi)$), polynomiality of the coefficients $p_{a,\xi}$ does hold, however the support of Schur polynomials is no longer finite; we will explicitly show this in Remark \ref{rmk:infinitesupport}.
\end{example}
We write out the first few terms of the sum:
\[
\begin{array}{cc}
\begin{aligned}
\Tp(\check\Sigma^{2,1}(2-i))
={}&
\Tp(\Sigma^{i,1}(2-i))
\\[1mm]
={}&
i\,s_{2i,i+1}
+
\frac{i(i-2)(i+1)}2\,s_{2i-1,i+2}
+
\frac{i^2(i-4)(i-1)(i+1)}{12}\,s_{2i-2,i+3}
&&(\xi=\varnothing)
\\[1mm]
&+
2s_{2i,i,1}
+
2(i-1)(i+1)s_{2i-1,i+1,1}
+
\frac{i^2(i-3)(i+1)}2s_{2i-2,i+2,1}
&&(\xi=(1))
\\[1mm]
&+
2(i+1)s_{2i-1,i,1,1}
+
\frac{(i-2)(2i^2+5i+1)}2s_{2i-2,i+1,1,1}
\\
&+
\frac{i(i-4)(i-1)(i^2+3i-1)}6s_{2i-3,i+2,1,1}
&&(\xi=(1,1))
\\[1mm]
&+
2(i+1)s_{2i-1,i,2}
+
\frac{(i-2)(i+1)(2i+1)}2s_{2i-2,i+1,2}
\\
&+
\frac{i(i-4)(i+1)(i^2-i+1)}6s_{2i-3,i+2,2}
&&(\xi=(2))
\\[1mm]
&\quad+\ \text{further terms with }|\xi|\ge 3 \text{ or }a\geq 3
\end{aligned}
\end{array}
\]
\begin{example}\label{ex:coeff}
For instance, we compute the coefficient for $\la(i,1,(1))$, i.e.\ $s_{2i-1,i+1,1}$; it is 
\[F_{(2,1^{i-2}),(1)}(i)=\det\left[ \gbinom{\nu_k+i-k}{\mu_l+i-l}\right]_{k,l\in i\times i}=\det\left[ \gbinom{\ga}{\rho}\right]_{\ga\in C, \rho\in R},
\]
where $C=(i+1, i-1, i-2\stb 2,0)$ and $R=(i,i-2,i-3\stb 0)$. Taking row differences, since $\gbinom{\ga}{\rho}-\gbinom{\ga}{\rho-1}=\binom{\ga}{\rho}$, this is further equal to 
\[
\det\left[ \binom{\ga}{\rho}\right]_{\ga\in C, \rho\in R}+\det\left[ \binom{\ga}{\rho}\right]_{\ga\in C, \rho\in R'}
\]
where $R'=(i-1,i-2\stb 0)$; denote the  two terms by $D_1$ and $D_2$. Note that these are minors of the Pascal matrix $P=\left[\binom{\ga}{\rho}\right]_{k,l=0}^{i+1}$, where $D_1$ corresponds to the deleted rows $(i-1,i+1)$, columns $(1,i)$ and $D_2$ to the deleted rows $(i,i+1)$ and columns $(1,i)$. Then Jacobi's complementary minor identity, using $\det(P)=1$ gives
\[\det(P[I,J])|=\det(P^{-1}[J^c,I^c]),\] 
and since $P^{-1}=\left[(-1)^{\ga+\rho}\binom{\ga}{\rho}\right]$, we have
\[
D_1=\det\left(\begin{array}{cc}
\binom{i-1}{1}&-\binom{i+1}{1}  \vspace{0.5pc}\\
-\binom{i-1}{i}&\binom{i+1}{i}
\end{array}\right)=(i-1)(i+1), \qquad D_2=\det\left(\begin{array}{cc}
\binom{i}{1}&-\binom{i+1}{1}  \vspace{0.5pc}\\
-\binom{i}{i}&\binom{i+1}{i}
\end{array}\right)=i(i+1)-(i+1),
\]
and summing, we get that the  value of the coefficient is $2(i-1)(i+1)$.
\end{example}
\begin{example}
	For $i=1$ we have \[\Tp(\check\Sigma^{2,1}(1))=\Tp(\Sigma^{1,1}(1))=s_{2,2}+2s_{2,1,1}+4s_{1,1,1,1},  \]
	and for  $i=2$ we have
	\[\Tp(\check\Sigma^{2,1}(0))=\Tp(\Sigma^{2,1}(0))=  \]
	\[2s_{4,3}+2s_{4,2,1}+6s_{3,3,1}+6s_{3,2,1,1}+8s_{2,2,2,1}+4s_{2,2,1,1,1}+6s_{3,2,2}.\]
\end{example}

\begin{remark}\label{rmk:infinitesupport} 
	Although for each $i$, $I_i$ is finite, the support of Schur polynomials is infinite. For instance, 
	for any $N\geq 1$, take
	\[
	\xi=(1^N),\qquad a=N-1.
	\]
	Then 
	\[\la(i,a,\xi)=(2i-N+1,\ i,\ 1^N).\] 
	These terms have tails of arbitrarily large length $N$.  Moreover their
	coefficients are nonzero; explicitly, one can show (by a computation very similar to the one in Example \ref{ex:coeff})
	\[
	p_{a,\xi}(i)
	=
	2\left(\binom{i+1}{N-1}-\binom{i-1}{N-3}\right)
	,
	\]
	which is a polynomial of degree $N-1$. It is also possible to show that $p_{a,\xi}$ is the zero polynomial if $\xi_3\geq2$.	
\end{remark}
\section{Enumerative applications} \label{sec:enum}
Thom polynomials can be used to solve problems in enumerative geometry \cite{Kleiman1977}, \cite{Kazarian2003b}.  A companion paper on this topic is under preparation \cite{FeherMatszangoszupcoming}, where the authors will apply the results of the current paper.  Here we only mention some possibilities without explaining the method.
\subsection{Enumerative applications of stable Thom polynomials}
\label{sec:enum-stable}
We start with the following known enumeration: given a generic degree $d$ hypersurface in $\PP^4$, how many 2-planes intersect it in an $E_6$ singularity? This can be computed using the Thom polynomial $[E_6(2,1)]$ to be: $15d(19d-36)(d-3)(d-2)$, see \cite[\S 10]{Kazarian2003b}, \cite{LeePatelTseng2023}. A generalization, using our calculation of $\Tp(E_6(-1))$ is
\begin{theorem} \label{thm:e6pts}
Given a generic, bidegree $(d,e)$ complete intersection 3-fold $X$ in $\PP^5$ and a general point $P$, for $e\cdot d$ large enough, the number of projective 3-planes passing through $P$ and intersecting $X$ in an $E_6$ singularity is
\[
3 \cdot e \cdot d \cdot (95 d^{3} + 270 d^{2} e + 270 d e^{2} + 95 e^{3} - 925 d^{2} - 1778 d e - 925 e^{2} + 2978 d + 2978 e - 3228).
\]
\end{theorem}
 For this result the Legendre Thom polynomial of $E_6$ is not enough, we also need the two extra coefficients from Section \ref{sec:small_r}. For instance, for $d=3,e=4$, it gives 161352. 

 Notice that setting $e=1$, and intersecting the problem with the corresponding hyperplane translates the problem to the problem of Kazarian mentioned above; the Schubert condition $\bra W,P\ket$ gives the 3-plane in $\PP^5$.
Indeed, the formula above with $e=1$ specializes to
\[3d(95d^3+(270-925)d^2+(270-1778+2978)d+(95-925+2978-3228))=15d(19d-36)(d-3)(d-2).\]
\subsection{Enumerative applications of unstable Thom polynomials} \label{sec:enum-unstable}
Results of Section \ref{sec:binary} can be used to calculate the degree of $\eta$-discriminants:
\begin{definition}For a singularity $\eta(n,1)$ and $d$ positive integer the $\eta$-discriminant $D_\eta(d)$ is the variety\[D_\eta(d)=\{ F\in \pol^d(\C^{n+1}):\ (F=0) \text{ admits an $\eta$-type singularity} \}.\]
\end{definition}
A sample result---using Corollary \ref{cor:21k}---is
\begin{theorem} For $\eta=\eta_{(2,1^{s})}(2,1)$ the (projective) degree of the  $\eta$-discriminant is	
\begin{equation}
\alpha(d-s-1)(\beta d-\gamma).
\end{equation}
where 
\[\alpha=\frac{1}{2}(s+1)(s+2)(s+3),\ \ \beta=\frac{1}{2}(s+3)(s+2)+1,\ \ \gamma=\alpha+2s+5.\]
\end{theorem}
\subsection{Enumerative applications of  Thom series}
Using Theorem \ref{thm:thom_series} we can obtain closed formulas for enumerations depending on $n$ and $d$ as well. These ideas will be explored in the forthcoming paper \cite{FeherMatszangoszupcoming}.

\section{Open problems} \label{sec:open}
In this paper we formulated a structural theorem for the Thom series of function singularities, and computed Thom polynomials of a wide range of singularities in various generalities (unstable, stable, Legendre). However, several interesting questions remain.
\begin{enumerate}
	\item It would be desirable to generalize the notion of quadratic unfolding and Thom series for contact singularities of negative codimension which are not function singularities.
	\item In all of our calculations of Thom series the corresponding polynomials $g_{\alpha,i}(r)$ (see Theorem \ref{thm:poly4p=1})  are either constant zero or polynomials in $r$ of degree exactly $i$ (for the vanishing of a coefficient see for example the Thom series of  $D_4$  in Theorem \ref{thm:thom_series}). We don't know if this holds in general.
	\item In all known cases of Legendre Thom polynomials $\tilde{K}_\eta$ we see the following pattern: For $\tilde{K}_\eta=\sum_{i,\lambda}v_{i,\lambda}Q_\lambda t^i$ if $v_{i,\lambda}\neq 0$ and $i>0$ then there is a partition $\mu\supset \lambda$ such that  $v_{i-1,\mu}\neq 0$. Loosely speaking the support is decreasing with $i$ increasing. In fact this conjecture would imply the previous one.
	\item What are the Legendre Thom polynomials of the binary singularities of Table \ref{tab:binary-crl-strata} after $X_9$, i.e.\ what are the coefficients of $t^iQ_\la$ for $\ell(\la)\geq 3$?
	\item As Section \ref{sec:small_r} demonstrates, even if the Legendre Thom polynomials $\tilde{K}_\eta$ are known, several coefficients of the stable Thom polynomial $\Tp(\eta(-r))$ are not determined for small values of $r$. What are the stable Thom polynomials $\Tp(\eta(-r))$ of the binary singularities of Table \ref{tab:binary-crl-strata} after $X_{10}$, e.g.\ for $r=1$?
\end{enumerate}

\appendix 
\addtocontents{toc}{\protect\setcounter{tocdepth}{1}}
\section{Symmetries of the quadratic stabilization}\label{sec:weights}
\renewcommand{\QQ}{\sigma_q}
\newcommand{\wt}{\operatorname{wt}}
\newcommand{\Tdiag}{T^{\mathrm{diag}}}
\newcommand{\rk}{\operatorname{rk}}

The key ingredient for applying restriction equations as described in Section \ref{sec:restriction} is a description of the symmetries of quasihomogeneous map germs in negative codimension. This involves understanding how the symmetries change under quadratic stabilization. Standard background on simple singularities, quasi-homogeneity, and versal unfolding may be found in \cite{AGZV}, \cite{Saito1971}, \cite{GreuelLossenShustin}, \cite{Looijenga1984}.
\subsection{Symmetries}

\begin{definition}
Let $f(x_1\stb x_n)$ be a polynomial and $T=(\C^\times)^r$. Let $T^\vee\iso \Z\bra t_1,\dots,t_r\ket$ denote the group of characters of $T$. We call 
\[ (\alpha_1,\dots,\alpha_n;d)\]
a \emph{weight vector of rank $r$} for $f$ if $\alpha_1,\dots,\alpha_n,d$ are in $T^\vee$ and
\[f(\tau^{\alpha_1}x_1,\stb ,\tau^{\alpha_n}x_n)=\tau^df(x_1\stb x_n)
\]
for all $\tau\in T$.
\end{definition}

Usually there are several weight vectors of $f$. Usually we choose $T$ to be of maximal rank. Composing the weight vector with an automorphism of $T$ gives another weight vector, we try to choose the \quot{simplest} one. Also, multiplying a weight vector with an integer we obtain another weight vector. To avoid divisibility discussions we allow multiplication by a rational number (usually 1/2, and then work in $T^\vee\otimes_\Z \Q$).

 By \cite{Rimanyi} it is possible to choose representatives and  weight vectors for our singularities such that they  determine a system of restriction equations with a unique solution if we choose the tori to be maximal. But we take an opportunistic approach: we give weight vectors which are easy to check, and provide a unique solution to the restriction equations, and don't worry about the maximality of the torus symmetry.

\begin{example}
	For $xy\in A_1(2,1)$ and  $T^\vee=\bra t_1,t_2\ket$ we see that $(t_1,t_2;t_1+t_2)$ is a weight vector.
\end{example}

\subsection{Quadratic stabilization}
This is the simplest case of the Thom-Sebastiani sum $(f,g)\mapsto f(x)+g(y)$ for $g(y)=y^2$, see \cite{SebastianiThom1971}. Iterating the quadratic stabilization twice gives
\[
\QQ^2(f)=f(x)+y^2+z^2\sim_{\mathcal{K}} f+uv,
\]
with $u=y+iz$, $v=y-iz$, i.e.\ the addition of a hyperbolic form.

\begin{remark}
	Two function germs $f\in J(n,1)$ and $g\in J(m,1)$ are stably right equivalent if
	\[
	f\oplus q_r \sim_{\mathcal R} g\oplus q_s
	\]
	for some non-degenerate quadratic forms $q_r,q_s$.  Over $\C$, the choice of the non-degenerate quadratic form is irrelevant. The splitting lemma (or Morse lemma) says that a germ with a non-degenerate quadratic part splits, up to right equivalence, as its residual non-Morse part plus a quadratic form.  Therefore if two stably right equivalent germs have the same number of variables and the same Morse rank, then their residual parts are right equivalent.  Quadratic stabilization changes only the quadratic summand.  For the general finite-determinacy and equivalence framework, see \cite{Mather1969}, \cite{Wall1981}, \cite{GreuelLossenShustin}.
\end{remark}

For the quadratic stabilization we immediately obtain that
\begin{lemma}\label{lemma:weight_of_Q1}
	Assume that the polynomial $f(x_1\stb x_n)$ has weight vector $(\al_1\stb \al_n;d)$ of rank $r$. Then
\begin{enumerate}[(i)]
\item $(\al_1\stb \al_n,d/2;d)$ is a  weight vector for $\QQ(f)=f(x)+y^2$,
\item $(\al_1\stb \al_n,t_{r+1},d-t_{r+1};d)$ is a  weight vector for $\QQ^2(f)=f(x)+uv$.
\end{enumerate}	
\end{lemma}
\subsection{Miniversal deformations and stabilization}\label{sec:miniversal_stabilization}
\newcommand{\OO}{\mathcal{O}}
\newcommand{\Tjur}{\operatorname{Tjur}}
\newcommand{\Jac}{\operatorname{Jac}}
To compute the entire family of stable Thom polynomials in negative codimension via restriction equations, for the principal equation one needs to compute the miniversal deformation of quadratic stabilizations. However, as quadratic stabilization ``commutes" with miniversal deformation in the sense below, it is enough to first determine the weights of the miniversal deformation of the singularity, and then apply the results of the last section.

\begin{proposition}[Stabilization commutes with miniversal deformation]\label{prop:miniversal-stabilization}
	Let $f(x)\in J(n,1)$ be a finite  ICIS. If $F(x,s)$ is a miniversal unfolding  of $f(x)$, then
	\[
	\QQ(F)(x,y,s)=F(x,s)+y^2
	\]
	is a miniversal unfolding of $\QQ(f)=f+y^2$.  The analogous statement holds for $\QQ^2(f)=f+uv$.
\end{proposition}
\begin{proof}
Let $\Tjur(f)=J(n,1)/(f,\partial_1f\stb \partial_nf)$ be the Tjurina algebra of $f$. If $\phi_1\stb \phi_\mu$ is a basis for $\Tjur(f)$, then
\[
F(x,s)=f(x)+\sum_{i=1}^\mu s_i\phi_i(x)
\]
is a (contact) miniversal deformation of $f$ \cite[Corollary 1.17]{GreuelLossenShustin}. Since $\Tjur(f)\iso \Tjur(f+y^2)$, the same classes $\phi_i$ form a basis of $\Tjur(f+y^2)$.
\end{proof}

\subsection{Examples}
We illustrate on a few examples how the weight system changes under quadratic stabilization.

\begin{example}[$A_1$]\
	For $xy\in A_1(-1)$, $(t_1,t_2;t_1+t_2)$ is a weight vector. More generally, for
	\[
	f:=x_1x_{k+1}+\cdots+x_kx_{2k}\in A_1(-2k+1)
	\]
	with $d=\sum_{i=1}^{k} t_i$ we can see that
	\[ (t_1,\dots,t_k,d-t_1,\dots,d-t_k;d)\]
	is a  weight vector for $f$.
\end{example}

\begin{example}[$A_2$]\label{ex:a2_weights}
		$x^3+y^2\in A_2(-1)$ has weight vector
$(2t,3t;6t)$.
The next quadratic unfolding, $x^3+uv\in A_2(-2)$ has weight vector
	$(t_1+t_2,3t_1,3t_2;3(t_1+t_2))$.
	
	We can also use Lemma \ref{lemma:weight_of_Q1}: $(t_1;3t_1)$ is a weight vector for $x^3$, so $(t_1,t_2,3t_1-t_2;3t_1)$ is a  weight vector for  $x^3+uv$. Notice that the coordinate change $t_1\to t_1+t_2,\ t_2\to 3t_1$ moves the second choice to the first one.
\end{example}

\begin{example}[$D_4$]\label{ex:D4}
For $x^2y+y^3\in D_4(-1)$ a weight vector is $(t,t;3t)$, and using  Lemma \ref{lemma:weight_of_Q1} (i) we obtain that $(t,t,3t/2;3t)$---or $(2t,2t,3t;6t)$ if we multiply by 2---is a weight vector for the quadratic unfolding $ x^2y+y^3+z^2\in 	D_4(-2)$. For the double quadratic unfolding $x^2y+y^3+uv \in D_4(-3)$ using  Lemma \ref{lemma:weight_of_Q1} (ii) we obtain that 
\[ (t_1,t_1,t_2,3t_1-t_2;3t_1)\]
	is a weight vector.
\end{example}
Similarly for a given function singularity it is enough to record the source and target weights in their first appearance, and apply Lemma \ref{lemma:weight_of_Q1} for all the quadratic unfoldings.

\subsection{Weights in the $\ell=-1$ case}
By the above discussion, for a given function singularity it is enough to record the source and target weights in their first appearance. For example we have

\begin{proposition}[The $A_k(-1)$ stabilizer]\label{prop:ak-stabilizer}
	For
$
	x^{k+1}+y^2\in A_k(-1)
$
	we have that 
	\[\bigl(2t,(k+1)t;2(k+1)t\bigr)\]
	 is a weight vector (for $k$ odd we can divide by 2).
\end{proposition}

\begin{proposition}[The $D_k(-1)$ stabilizer]\label{prop:dk-stabilizer}
	For $x^2y+y^{k-1}\in 	D_k(-1), \qquad k\ge 4$ we can see that
\[\bigl((k-2)t,2t;2(k-1)t\bigr)\]
is a weight vector (for $k$ even we can divide by 2).	
\end{proposition}

\begin{remark} In Table \ref*{tab:first-appearance} we collected all the data needed to calculate the stable Thom polynomials of Section \ref{sec:small_r}. In the source weights column we listed the weights of the source of the genotype before the vertical line $|$, and the extra weights of the prototype after. The singularities $S_5$ and $S_6$ are not function singularities, so they have 2 target weights.

\end{remark}

\begingroup
\scriptsize
\setlength{\tabcolsep}{2pt}
\setlength\LTleft{0pt}
\setlength\LTright{0pt}
\begin{longtable}{@{}>{\raggedright\arraybackslash}p{0.095\textwidth}>
		{\raggedright\arraybackslash}p{0.155\textwidth}>
		{\raggedright\arraybackslash}p{0.35\textwidth}>
		{\raggedright\arraybackslash}p{0.215\textwidth}>
		{\raggedright\arraybackslash}p{0.095\textwidth}@{}}
	\caption{Weights of singularities for $l=-1$ and codimension$\leq$7.}
\label{tab:first-appearance}\\	
	\toprule
	Singularity & Genotype & Prototype & Source weights & Target \\
	\midrule
	\endfirsthead
	\toprule
	Singularity & Genotype & Prototype & Source weights & Target \\
	\midrule
	\endhead
	\bottomrule
	\endfoot
	
	\(A_1\)
	& \(xy\)
	& \(xy\)
	& \((t_1,t_2)\)
	& \(t_1+t_2\) \\
	
	\(A_2\)
	& \(x^3+y^2\)
	& \(x^3+y^2+ux\)
	& \((2t,3t|4t)\)
	& \(6t\) \\
	
	\(A_3\)
	& \(x^4+y^2\)
	& \(x^4+y^2+ux^2+vx\)
	& \((t,2t|2t,3t)\)
	& \(4t\) \\
	
	\(A_4\)
	& \(x^5+y^2\)
	& \(x^5+y^2+ux^3+vx^2+wx\)
	& \((2t,5t|4t,6t,8t)\)
	& \(10t\) \\
	
	\(D_4\)
	& \(x^2y+y^3\)
	& \(x^2y+y^3+ux+vy+wx^2\)
	& \((t,t|2t,2t,t)\)
	& \(3t\) \\
	
	\(A_5\)
	& \(x^6+y^2\)
	& \(x^6+y^2+ux^4+vx^3+wx^2+zx\)
	& \((t,3t|2t,3t,4t,5t)\)
	& \(6t\) \\
	
	\(D_5\)
	& \(x^2y+y^4\)
	& \(x^2y+y^4+ux+vy+wy^2+zy^3\)
	& \((3t,2t|5t,6t,4t,2t)\)
	& \(8t\) \\
	
	\(S_5\)
	& \((x^2+y^2+z^2,yz)\)
	& \((x^2+y^2+z^2,\ yz+ux+vy+wz)\)
	& \((t,t,t|t,t,t)\)
	& \((2t,2t)\) \\
	
	\(A_6\)
	& \(x^7+y^2\)
	& \(\begin{aligned}[t]
		&x^7+y^2+ux^5+vx^4+wx^3+zx^2+ax
	\end{aligned}\)
	& \((2t,7t|4t,6t,8t,10t,12t)\)
	& \(14t\) \\
	
	\(D_6\)
	& \(x^2y+y^5\)
	& \(\begin{aligned}[t]
		&x^2y+y^5+ux+vy+wy^2+zy^3+ay^4
	\end{aligned}\)
	& \((2t,t|3t,4t,3t,2t,t)\)
	& \(5t\) \\
	
	\(E_6\)
	& \(x^3+y^4\)
	& \(\begin{aligned}[t]
		&x^3+y^4+uxy^2+vxy+wy^2+zx+ay
	\end{aligned}\)
	& \((4t,3t|2t,5t,6t,8t,9t)\)
	& \(12t\) \\
	
	\(S_6\)
	& \((x^2+y^2+z^3,yz)\)
	& \((x^2+y^2+z^3+uz,\ yz+vx+wy+az)\)
	& \((3t,3t,2t|4t,2t,2t,3t)\)
	& \((6t,5t)\) \\
  \end{longtable}
\normalsize
\section{Fibered resolution}
\label{app:fibered}

\begin{proposition} \label{prop:fibered-reso}
	Let $f:Z\to M$ be a smooth, proper $G$-map between smooth complex $G$-manifolds and $\pi_M:E_M\to M$ be a $G$-vector bundle, and $K\subset f^*E_M$ be a $G$-subbundle. Let $F:f^*E_M\to E_M$:
    
   \[
\xymatrix{
	K\subset f^*E_M\ar[r]^-{F}\ar[d]&E_M\ar[r]^{q_M} \ar[d]^{\pi_M}& E_M\\
	Z\ar[r]^{f}&M&
}
\]   
    
    Then
	\begin{equation}\label{eq:push}
		F_! [K\subset f^*E_M]_G=\pi_M^* f_!e_G(f^*E_M/K).
	\end{equation}
	where $K\subset E_M$ via the composition $K\to f^*E_M\to E_M$.
\end{proposition}
\begin{proof}
	Note that \[e_G(f^*E_M/K)=z_Z^*[K\subset f^*E_M]\]where $z_Z$ denotes the zero section of $f^*E_M$. Since $z_M^*$ is an isomorphism (where $z_M$ is the zero section of $E_M$), apply it to both sides of the statement \eqref{eq:push}. Then the claim reduces to $f_!z_Z^*=z_M^*F_!$ which follows from the maps being pullbacks of bundles and zero sections.
\end{proof}
\begin{proposition}
	Let $Y\subset A$ be a $G$-invariant subvariety, such that the image of $G$ contains the scalars. Take the ($G$-invariant) projectivization $\PP Y\subset \PP A$. Let $f:\tilde{Y}\to \PP A$ be a resolution of $\PP Y$. Then
	\[
	[Y\subset A]_G=\co_A^*\int_{\tilde{Y}}e_G(f^*Q_A)
	\]
	where $\co:A\to pt$ is the collapse map.
\end{proposition}
\begin{proof}
	We will apply the above situation to $M:=\PP A$, and the trivial vector bundle $\pi_A:A\to \PP(A)$, and the subbundle $K=f^*S_{\PP A}$. The conclusion above states that
	\[
	F_![f^*S_{\PP A}\subset \tilde{Y}\times A]_G=\pi_A^*f_!e_G(f^*Q_A).
	\]
	Apply $(q_A)_!$ for the other projection $q_A:\PP A\times A\to A$. Then \[f^*S_{\PP A}\xrightarrow{F} \PP A\times A\xrightarrow{q_A}A\] is proper and birational onto $Y$; indeed, over $\PP Y\subset \PP A$, $S_A|_{\PP Y}\to A$ is the blow-up of $Y$ along $0$: the preimage of $0$ is $\PP Y$. So applying $(q_A)_!$, we have
	\[
	[Y\subset A]_G=(q_A)_!\pi_A^* f_!e_G(f^*Q_A)=\co_A^*\int_{\PP A}f_!e_G(f^*Q_A)=\co_A^*\int_{\tilde{Y}}e_G(f^*Q_A)
	\]
\end{proof}
More generally:

Let $G$ be a linear algebraic group acting on the vector spaces $A,B$. Assume that $Y\subset V:=A\oplus B$ is a bihomogeneous subvariety (i.e.\ $\C^\times\times\C^\times$-invariant), which therefore has a bi-projectivization $\PP_{A,B}Y\subset \PP A\times \PP B$. Assume also that $Y$ is irreducible and not contained in $A\oplus 0$ or $0\oplus B$. 
\begin{proposition} \label{prop:biproj}
	Let $\Phi:\tilde{Y}\to \PP A\times \PP B$ be a $G$-equivariant resolution of $\PP_{A,B}Y$. Then
	\[
	[Y\subset A\oplus B]_G=\int_{\tilde{Y}}\Phi^*(e_G(Q_A)\cdot e_G(Q_B))
	\]
	where $Q_A$ and $Q_B$ are the tautological quotient bundles over $\PP A$ and $\PP B$.
\end{proposition}
For the proof, apply Proposition \ref{prop:fibered-reso} to $M:=\PP A\times \PP B$, the trivial bundle $A\oplus B$, the subbundle $K=\Phi^*S_A\oplus \Phi^*S_B$.
\endgroup
\bibliographystyle{alpha}
\bibliography{thom_negative_biblio}
\end{document}

%% file: abbrev.tex
\usepackage{wrapfig}
\usepackage{amsmath,amsxtra}
\usepackage{amssymb}
\usepackage{xy}
\xyoption{all}
\usepackage{epsfig} \usepackage{epstopdf} 
\usepackage{paralist}

\newcommand{\bra}{\langle}
\newcommand{\ket}{\rangle}
\newcommand{\quot}[1]{``#1''}

\newtheorem{fact}{Fact}[section]
\newtheorem{lemma}[fact]{Lemma}
\newtheorem{theorem}[fact]{Theorem}
\newtheorem{defi}[fact]{Definition}
\newtheorem{exer}[fact]{Exercise}
\newtheorem{exa}[fact]{Example}
\newtheorem{exas}[fact]{Examples}
\newtheorem{ob}[fact]{Observation}
\newtheorem{rremark}[fact]{Remark}
\newtheorem{proposition}[fact]{Proposition}
\newtheorem{corollary}[fact]{Corollary}
\newenvironment{remark}{\begin{rremark}\small \rm}{\end{rremark}}
\newenvironment{observation}{\begin{ob} \rm}{\end{ob}}
\newenvironment{definition}{\begin{defi} \rm}{\end{defi}}
\newenvironment{example}{\begin{exa} \rm}{\renewcommand\qedsymbol{$\diamond$}\hfill\qedsymbol\end{exa}}
\newenvironment{examples}{\begin{exas} \rm}{\end{exas}}

\newcommand{\C}{{\mathbb C}}
\newcommand{\Q}{{\mathbb Q}}
\newcommand{\Z}{{\mathbb Z}}
\newcommand{\N}{{\mathbb N}}
\newcommand{\F}{{\mathbb F}}

\renewcommand{\P}{{\mathbb P}}
\newcommand{\Top}{{\mathcal T\hspace{-3pt}op}}

\newcommand{\Ga}{\Gamma}
\newcommand{\al}{\alpha}

\newcommand{\si}{\sigma}

\DeclareMathOperator{\Hom}{Hom}

\DeclareMathOperator{\codim}{codim}

\DeclareMathOperator{\id}{Id}

\DeclareMathOperator{\GL}{GL}

\DeclareMathOperator{\co}{co}

\DeclareMathOperator{\Tp}{Tp}

\DeclareMathOperator{\Sym}{Sym}

\DeclareMathOperator\I{\mathcal I}

\DeclareMathOperator{\pol}{Pol}
\DeclareMathOperator{\Pol}{Pol}

\renewcommand{\epsilon}{\varepsilon}

	\newcommand*\clos[1]{\overline{#1}}
	\newcommand{\PP}{\mathbb{P}}
	
	\newcommand{\be}{\beta}
	\newcommand{\ga}{\gamma}

	\newcommand{\la}{\lambda}
	
	\newcommand{\Si}{\Sigma}

	\newcommand{\stb}{,\ldots,}

	\def\iso{\cong}

	\newcommand{\twocase}[6][1mm]{{#2}\begin{cases} {#3}& \hbox{\rm if}\ \ {#4} \\[{#1}] {#5}& \hbox{\rm if}\ \ {#6}\end{cases}}

	\usepackage{tikz}